\documentclass{article}
\usepackage[journal=JEMS,lang=british]{ems-journal-jems}
\usepackage{aligned-overset}
\usepackage{cancel}
\usepackage{cases}
\usepackage{comment}
\usepackage{diagbox}
\usepackage{nicefrac}
\usepackage{tikz-cd}

\usetikzlibrary{matrix,fit}
\usetikzlibrary{shapes}

\numberwithin{equation}{section}

\newtheorem{theorem}{Theorem}
\newtheorem{proposition}[theorem]{Proposition}
\newtheorem{lemma}[theorem]{Lemma}
\newtheorem{corollary}[theorem]{Corollary}
\theoremstyle{remark}
\newtheorem{remark}[theorem]{Remark}
\theoremstyle{definition}

\theoremstyle{example}
\newtheorem{example}[theorem]{Example}

\newcommand{\st}{\,:\,}
\newcommand{\Real}{\mathbb{R}}

\DeclareRobustCommand{\bvec}[1]{\boldsymbol{#1}}
\newcommand{\uvec}[1]{\underline{\bvec{#1}}}
\newcommand{\cvec}[1]{\bvec{\mathcal{#1}}}

\newcommand{\rd}{\mathrm{d}}
\newcommand{\ud}{\ul\rd}

\newcommand{\sign}{\text{sign}}

\newcommand{\rotation}[1]{\varrho_{#1}}

\DeclareMathOperator{\GRAD}{\bf grad}
\DeclareMathOperator{\CURL}{\bf curl}
\DeclareMathOperator{\DIV}{div}
\DeclareMathOperator{\ROT}{rot}
\DeclareMathOperator{\VROT}{\bf rot}

\newcommand{\compl}{{\rm c}}

\newcommand{\Hcurl}[1]{\bvec{H}(\CURL;#1)}
\newcommand{\Hrot}[1]{\bvec{H}(\ROT;#1)}
\newcommand{\Hdiv}[1]{\bvec{H}(\DIV;#1)}

\newcommand{\ul}{\underline}

\newcommand{\Xcurl}[1]{\underline{\bvec{X}}_{\CURL,#1}^r}

\newcommand{\lproj}[3]{\pi_{#1,#2}^{#3}}

\newcommand{\trimproj}[3]{\pi_{#1,#2}^{-,#3}}

\newcommand{\uCh}[1][]{\uvec{C}_h^k}
\newcommand{\Dh}[1][]{D_h^k}

\newcommand{\CF}{C_F^r}
\newcommand{\CT}{\boldsymbol{C}_T^r}

\newcommand{\trFt}{\bvec{\gamma}_{{\rm t},F}^r}

\newcommand{\Pot}[3]{P_{#1,#2}^{#3}}

\newcommand{\Pcurl}[1][T]{\bvec{P}_{\CURL,#1}^r}

\newcommand{\faces}[1]{\mathcal{F}_{#1}}
\newcommand{\edges}[1]{\mathcal{E}_{#1}}

\newcommand{\FT}{\faces{T}}
\newcommand{\ET}{\edges{T}}
\newcommand{\EF}{\edges{F}}

\newcommand{\normal}{\bvec{n}}

\newcommand{\Poly}[2]{\mathcal{P}_{#1}^{#2}}
\newcommand{\vPoly}[2]{\cvec{P}_{#1}^{#2}}
\newcommand{\Roly}[1]{\cvec{R}_{#1}}
\newcommand{\Goly}[1]{\cvec{G}_{#1}}
\newcommand{\cRoly}[1]{\cvec{R}^\compl_{#1}}
\newcommand{\cGoly}[1]{\cvec{G}^\compl_{#1}}

\newcommand{\RaviartThomas}[1]{\boldsymbol{\mathcal{RT}}_{#1}}
\newcommand{\Nedelec}[1]{\boldsymbol{\mathcal{N}}_{#1}}

\newcommand{\Koly}[2]{\mathcal{K}_{#1}^{#2}}
\newcommand{\kproj}[3]{\pi_{#1,#2}^{\mathcal{K},#3}}

\DeclareMathOperator{\Ker}{Ker}
\DeclareMathOperator{\Image}{Im}

\newcommand{\Mh}{\mathcal{M}_h}
\newcommand{\Th}{\mathcal{T}_h}
\newcommand{\Fh}{\mathcal{F}_h}
\newcommand{\Eh}{\mathcal{E}_h}
\newcommand{\Vh}{\mathcal{V}_h}

\DeclareMathOperator{\tr}{tr}

\newcommand{\vol}{\mathrm{vol}}

\newcommand{\DDR}[1]{\ensuremath{\mathrm{DDR}(#1)}}
\newcommand{\VEM}[1]{\ensuremath{\mathrm{VEM}(#1)}}

\newcommand{\Alt}[1]{{\rm Alt}^{#1}}

\newcommand{\Ext}[2]{\ul E_{#1}^{#2}}
\newcommand{\fExt}[2]{E_{#1}^{#2}}
\newcommand{\Red}[2]{\ul R_{#1}^{#2}}

\newcommand{\norm}[2]{\|#2\|_{#1}}
\newcommand{\seminorm}[2]{|#2|_{#1}}
\newcommand{\vvvert}{\vert\kern-0.25ex\vert\kern-0.25ex\vert}
\newcommand{\tnorm}[2]{\vvvert #2\vvvert_{#1}}

\newcommand{\hud}[1]{\hat{\underline{#1}}}

\begin{document}

\title{An exterior calculus framework for polytopal methods}
\titlemark{EXTERIOR CALCULUS FOR POLYTOPAL METHODS}
\emsauthor{1}{
  \givenname{Francesco}
  \surname{Bonaldi}
  \orcid{0000-0003-0512-0362}
}{F.~Bonaldi}
\emsauthor{2}{
  \givenname{Daniele Antonio}
  \surname{Di Pietro}
  \orcid{0000-0003-0959-8830}
}{D.~A.~Di Pietro}
\emsauthor{3}{
  \givenname{J\'{e}r\^{o}me}
  \surname{Droniou}
  \orcid{0000-0002-3339-3053}
}{J.~Droniou}
\emsauthor{4}{
  \givenname{Kaibo}
  \surname{Hu}
  \orcid{0000-0001-9574-9644}
}{K.~Hu}

\Emsaffil{1}{
  \department{LAMPS}
  \organisation{Université de Perpignan Via Domitia}
  \city{Perpignan}
  \country{France}
  \affemail{francesco.bonaldi@univ-perp.fr}
}

\Emsaffil{2}{
  \department{IMAG}
  \organisation{Univ. Montpellier, CNRS}
  \city{Montpellier}
  \country{France}
  \affemail{daniele.di-pietro@umontpellier.fr}
}

\Emsaffil{3}{
  \department{1}{IMAG}
  \organisation{1}{Univ. Montpellier, CNRS}
  \city{1}{Montpellier}
  \country{1}{France}
  \affemail{jerome.droniou@umontpellier.fr}

  \department{2}{School of Mathematics}
  \organisation{2}{Monash University}
  \city{2}{Melbourne}
  \country{2}{Australia}
}

\Emsaffil{4}{
  \department{School of Mathematics}
  \organisation{University of Edinburgh}
  \city{Edinburgh}
  \country{UK}
  \affemail{kaibo.hu@ed.ac.uk}
}



\begin{abstract}
  We develop in this work the first polytopal complexes of differential forms.
  These complexes, inspired by the Discrete De Rham and the Virtual Element approaches, are discrete versions of the de Rham complex of differential forms built on meshes made of general polytopal elements.
  Both constructions benefit from the high-level approach of polytopal methods, which leads, on certain meshes, to leaner constructions than the finite element method.
  We establish commutation properties between the interpolators and the discrete and continuous exterior derivatives, prove key polynomial consistency results for the complexes, and show that their cohomologies are isomorphic to the cohomology of the continuous de Rham complex.
\end{abstract}

\maketitle




\section{Introduction}

This work is a first step towards merging two extremely successful avenues of research in numerical analysis: finite element differential forms and arbitrary-order polytopal methods.

The well-posedness of important classes of partial differential equations (PDEs), and the development of stable approximations thereof, hinges on the properties of underlying Hilbert complexes \cite{Bruning.Lesch:92}.
The best-known example is provided by the de Rham complex which, for an open connected polyhedral domain $\Omega\subset\Real^3$, reads
\begin{equation}\label{eq:continuous.de.rham}
  \begin{tikzcd}
    \{0\}\arrow{r}{} & H^1(\Omega)\arrow{r}{\GRAD} & \Hcurl{\Omega}\arrow{r}{\CURL} & \Hdiv{\Omega}\arrow{r}{\DIV} & L^2(\Omega)\arrow{r}{} & \{0\},
  \end{tikzcd}
\end{equation}
where
$H^1(\Omega)$ is the space of scalar-valued functions over $\Omega$ that are square-integrable along with their gradient, while
$\Hcurl{\Omega}$ and $\Hdiv{\Omega}$ are the spaces of vector-valued functions over $\Omega$ that are square-integrable along with their curl and divergence, respectively.
Using the framework of differential forms (see Appendix \ref{sec:appendix}), the de Rham complex \eqref{eq:continuous.de.rham} can be generalised to a domain $\Omega$ of any dimension $n$ as:
\begin{equation}\label{eq:diff.de.rham}
  \begin{tikzcd}
    \{0\}\arrow{r}{} & H\Lambda^0(\Omega) \arrow{r}{\rd^0} & {\cdots} \arrow{r}{\rd^{k-1}} & H\Lambda^k(\Omega) \arrow{r}{\rd^k} & {\cdots} \arrow{r}{\rd^{n-1}} & H\Lambda^n(\Omega) \arrow{r}{} & \{0\}.
  \end{tikzcd}
\end{equation}
In what follows, we shall possibly omit the index $k$ from exterior derivatives and spaces in \eqref{eq:diff.de.rham} when no ambiguity can arise.

The de Rham complex enters the well-posedness analysis of PDEs through its cohomology spaces $\Ker\rd^k/\Image\rd^{k-1}$.
A classical result links these spaces to the topological features of the domain and their dimensions to its Betti numbers.
Preserving such homological structures at the discrete level leads to \emph{compatible} methods and is key to the design of stable numerical schemes.

The compatible finite element approximation of the vector-valued spaces appearing in the de Rham complex~\eqref{eq:continuous.de.rham} arose as a research subject in the late 70s \cite{Raviart.Thomas:77,Nedelec:80}.
In the late 80s, links with Whitney forms were identified~\cite{Bossavit:88}.
More recently, the development of Finite Element Exterior Calculus (FEEC) \cite{Arnold.Falk.ea:06,Arnold.Falk.ea:10,Arnold:18} has provided a unified perspective on the generation and analysis of finite element approximations of the de Rham complex \eqref{eq:diff.de.rham}.
Finite Element Systems (FES) are a generalisation of FEEC covering spaces which are not necessarily piecewise polynomial inside mesh elements (but can be, for example, piecewise polynomial on subdivisions of these elements); see \cite{Christiansen.Munthe-Kaas.ea:11,Christiansen.Gillette:16,Christiansen.Hu:18}. FEEC and FES led to the unification of several families of finite elements and heavily hinge on the notion of subcomplex, which makes them naturally geared towards conforming approximations.

While conforming methods are still widely used, their construction relies on polynomial basis functions that can be globally and conformally glued, and can therefore only be carried out on conforming meshes, composed of elements of simple shape (e.g., tetrahedra or hexahedra); extensions to more general meshes, such as the barycentric dual of a simplicial mesh, have been considered, e.g., in \cite{Christiansen:08}.
In recent years, significant efforts have been made to develop and analyse numerical methods that support more general meshes including, e.g., general polytopal elements and non-matching interfaces; a representative but by no means exhaustive list of contributions includes \cite{Brezzi.Buffa.ea:09,Di-Pietro.Ern:10,Droniou.Eymard.ea:10,Eymard.Gallouet.ea:10,Bassi.Botti.ea:12,Di-Pietro.Ern:12,Botti:12,Antonietti.Giani.ea:13,Beirao-da-Veiga.Brezzi.ea:13,Bonelle.Ern:14,Brezzi.Falk.ea:14,Beirao-da-Veiga.Lipnikov.ea:14,Di-Pietro.Ern.ea:14,Cangiani.Dong.ea:17,Di-Pietro.Ern:17,Droniou.Eymard.ea:18,Di-Pietro.Droniou:20,Di-Pietro.Droniou.ea:20,Di-Pietro.Droniou:23*1}.
Polytopal technologies typically introduce some degree of non-conformity, either because they are formulated in a fully discrete setting (like Hybrid High-Order \cite{Di-Pietro.Ern.ea:14,Di-Pietro.Droniou:20} or Discrete de Rham -- DDR methods \cite{Di-Pietro.Droniou.ea:20,Di-Pietro.Droniou:23*1}) or through the use of projections (as in Virtual Element Methods -- VEM \cite{Beirao-da-Veiga.Brezzi.ea:13}).

Despite their non-conformity, polytopal technologies can be used to develop compatible frameworks.
Polytopal discretisations of the de Rham complex \eqref{eq:continuous.de.rham} have been proposed, e.g., in \cite{Beirao-da-Veiga.Brezzi.ea:16,Di-Pietro.Droniou.ea:20,Di-Pietro.Droniou:23*1},
and applied to a variety of models, such as magnetostatics \cite{Beirao-da-Veiga.Brezzi.ea:18,Di-Pietro.Droniou:21}, the Stokes equations \cite{Beirao-da-Veiga.Dassi.ea:22}, and the Yang--Mills equations \cite{Droniou.Oliynyk.ea:23}; they have also inspired further developments, based on the same principles, for other complexes of interest such as variants of the de Rham complex with increased regularity \cite{Zhao.Zhang:21,Di-Pietro:23}, elasticity complexes \cite{Di-Pietro.Droniou:23*3,Botti.Di-Pietro.ea:23}, and the Stokes complex \cite{Beirao-da-Veiga.Mora.ea:19,Hanot:23,Beirao-da-Veiga.Dassi.ea:20}.
Polytopal complexes have additionally been used to construct methods that are robust with respect to the variations of physical parameters, in particular for the Stokes \cite{Beirao-da-Veiga.Dassi.ea:22}, Reissner--Mindlin \cite{Di-Pietro.Droniou:22} and Brinkman \cite{Di-Pietro.Droniou:23*1} problems.
Many of these models have also been tackled using finite element complexes and related methods (see, e.g., \cite{Arnold:18,Christiansen.Hu:18,Christiansen.Gopalakrishnan.ea:20,Arnold.Hu:21}).
However, due to their higher-level design, which does not require the existence and computability of globally conforming piecewise polynomial basis functions, polytopal methods offer distinctive advantages over finite elements.
These include, in addition to the support of general meshes, the possibility to reduce the dimension of discrete spaces, sometimes below their finite element counterparts \cite[Table 3]{Di-Pietro.Droniou:23*2}, through systematic processes such as enhancement or serendipity \cite{Beirao-da-Veiga.Brezzi.ea:18*1,Di-Pietro.Droniou:23*2}.
Such added flexibility comes at a minor (especially when using homogeneous numerical integration \cite{Chin.Lasserre.ea:15}) additional cost with respect to standard finite elements, namely the need to solve local problems to reconstruct discrete counterparts of the exterior derivative and of the corresponding potentials.
As for finite elements, the size of the algebraic systems corresponding to the polytopal discretisation of a given problem can be reduced through static condensation and hybridisation.

The purpose of the present work is to take one step further and show how exterior calculus can be used to generalise the construction and analysis of polytopal complexes.
More specifically, we present two discrete de Rham complexes in arbitrary dimension and with arbitrary approximation degree that generalise those introduced in \cite{Di-Pietro.Droniou:23*1} (DDR) and \cite{Beirao-da-Veiga.Brezzi.ea:18} (VEM).
Three key features set these constructions apart from Finite Element complexes:
\begin{itemize}
\item No explicit spaces of globally conforming differential forms (i.e., subspaces of $H\Lambda(\Omega)$) are needed.
  Instead, we work with \emph{fully discrete spaces} made of vectors of polynomial components on the mesh cells (of various dimensions). The meaning of these components is provided by the interpolators on the fully discrete spaces.
\item Due to the absence of explicit underlying conforming spaces, the differential operator of the complex cannot be the exterior derivative. Instead, a \emph{discrete exterior derivative} is constructed combining the polynomial components to mimic the Stokes formula.
\item \emph{Discrete potentials} are also designed, again mimicking the Stokes formula.
  They are piecewise (discontinuous) polynomial forms on the mesh used, in particular, to define an $L^2$-structure on the discrete spaces (an essential tool to discretise PDEs written in weak form).
\end{itemize}
The choice of the polynomial components in the spaces and the design of discrete exterior derivatives and potentials revolve around two key properties: \emph{polynomial consistency}, which is related to the ability to reproduce exactly polynomial differential forms up to a selected polynomial degree, and \emph{compatibility}, linked to the existence of an isomorphism between the cohomology of the discrete and continuous de Rham complexes.
Notice that, in the finite element framework, polynomial consistency simply corresponds to the fact that suitable polynomial spaces are contained in the (local) finite element space.
While both the DDR- and VEM-inspired constructions heavily rely on discrete versions of the Stokes formula, they do so in a radically different spirit: in the DDR construction, the choice of components in the discrete spaces is inspired by the formula to reconstruct a discrete exterior derivative, which is then used to construct discrete potentials.
In the VEM construction, on the other hand, the space components (and, in particular, those associated with differentials) are chosen based on the formula used to define a discrete potential.
While the choice in the DDR construction leads to leaner spaces, the study of its properties is more elaborated.
Notice that, at this early stage, we haven't tried to identify the virtual (conforming) spaces that underlie the VEM-inspired construction, and we have made no effort whatsoever in trying to reduce the dimension of the discrete spaces through serendipity.

The rest of this work is organised as follows.
In Section \ref{sec:setting} we establish the setting.
In Section \ref{sec:ddr} we present and analyse the discrete complex generalising the DDR construction of \cite{Di-Pietro.Droniou:23*1}.
Section \ref{sec:vem} contains the definition and analysis of the complex generalising the VEM construction of \cite{Beirao-da-Veiga.Brezzi.ea:18}.
In Section \ref{sec:links}, we discuss in greater detail similarities and differences with respect to the FEEC, FES, and Distributional Differential Forms frameworks.

Differential forms of any degree in dimensions 2 and 3 have interpretations in terms of vector fields. To make the exposition self-contained and improve the legibility for the reader not accustomed to differential forms, we recall some facts on these so-called vector proxies in Appendix \ref{sec:appendix}, and we include throughout the exposition a series of examples to illustrate the development in the differential forms framework through vector calculus operators.


\section{Setting}\label{sec:setting}

We present here the main notions used in the construction of the polytopal complexes of differential forms. For the reader not used to the framework of differential forms, we recall in Appendix~\ref{sec:appendix} some basic concepts and definitions.

\subsection{Spaces of differential forms}\label{sec:setting:differential.forms}

Let $M$ denote an $n$-dimensional manifold.
In what follows, $M$ will typically be a cell of a polytopal mesh (see Section \ref{sec:polytopal.mesh} below), and thus a relatively open set in a subspace of $\Real^m$ for some $m\ge n$.
For any natural number $\ell$ such that $0\le \ell\le n$, we will denote by $\Lambda^\ell(M)$ the space of differential $\ell$-forms (often just called $\ell$-forms) on $M$ without explicit regularity requirements.
When relevant, regularity is made explicit by prepending the appropriate space (e.g.,~$L^2\Lambda^\ell(M)$ stands for square-integrable $\ell$-forms).

\subsection{Integration by parts}

We recall the following integration by parts (Stokes) formula:
\begin{multline}\label{eq:ipp}
  \int_M \rd\omega\wedge\mu
  = (-1)^{\ell+1}\int_M \omega \wedge \rd\mu
  + \int_{\partial M} \tr_{\partial M} \omega \wedge \tr_{\partial M}\mu
  \\
  \forall(\omega,\mu) \in C^1\Lambda^\ell(\overline{M})\times C^1\Lambda^{n-\ell-1}(\overline{M}),
\end{multline}
where, for any form degree $m$, $\tr_{\partial M}:C^0\Lambda^{m}(\overline{M}) \to C^0\Lambda^{m}(\partial M)$ is the trace operator, i.e., the pullback of the inclusion $\partial M\subset M$, and $\partial M$ is oriented with respect to $M$.
Formula \eqref{eq:ipp} will provide the starting point to define discrete counterparts of the exterior derivative and of the corresponding potentials on mesh cells.
It will also drive the choice of the components in the discrete spaces, geared at ensuring that the reconstructions preserve certain polynomial differential forms.

\subsection{Hodge star}

Assume now that $M$ is an open set in a subspace of $\Real^m$.
We denote by $\star:\Lambda^\ell(M) \to \Lambda^{n-\ell}(M)$ the Hodge star operator, and we set
\begin{equation}\label{eq:inv.star}
  \star^{-1} \coloneq (-1)^{\ell(n-\ell)}\star,
\end{equation}
a notation justified observing that, for any $\omega\in\Alt{\ell}(V)$ (with $\Alt{\ell}(V)$ denoting the set of alternating $\ell$-linear forms on $\Real^n$), $\star^{-1}\star\omega = \omega$ (see \eqref{eq:star.star} in the appendix).
Notice that, while the Hodge star operator depends on $M$, we won't need to make this dependence explicit as it will be clear from the context.

\subsection{$L^2$-orthogonal projectors}

Integrating the inner product of $\Alt{\ell}(V)$ over $M$ yields the inner product of $L^2\Lambda^\ell(M)$. For any closed subspace $\mathcal X$ of $L^2\Lambda^\ell(M)$, we therefore have an $L^2$-orthogonal projector $\pi_{\mathcal X}:L^2\Lambda^\ell(M)\to \mathcal X$ on $\mathcal X$, defined by the following relation:
For all $\omega\in L^2\Lambda^\ell(M)$, $\pi_{\mathcal X}\omega\in\mathcal X$ satisfies
\begin{equation}\label{eq:def.piX}
  \int_M\pi_{\mathcal X}\omega\wedge\star\mu
  = \int_M\omega\wedge{\star\mu}
  \qquad\forall \mu \in \mathcal X.
\end{equation}
To improve legibility, in the next sections we also introduce specific notations for $L^2$-orthogonal projectors on polynomial subspaces $\mathcal X$ that are particularly relevant to our construction.

\subsection{Polytopal mesh}\label{sec:polytopal.mesh}

From this point on, $\Omega$ will denote a polytopal domain of $\Real^n$.
We let $\Mh$ denote a \emph{polytopal mesh} of $\Omega$, i.e., a collection of disjoint polytopal sets (mesh entities) of dimensions in $[0,n]$, relatively open in their spanned affine space, such that the boundary of each $d$-cell (polytopal set of dimension $d$) is the union of mesh entities of dimension $<d$, and such that any $d$-cell for $d<n$ is contained in the boundary of some $(d+1)$-cell.
For any $d\in[0,n]$, the set collecting all $d$-cells of $\Mh$ is denoted by $\Delta_d(\Mh)$.
Notice that this notion of polytopal mesh essentially coincides with that of CW-complex in algebraic topology.
Thus, when $\Omega$ is a domain in dimension $n=3$, $\Mh$ gathers
the vertices collected in the set $\Vh \coloneq \Delta_0(\Mh)$,
the edges collected in the set $\Eh \coloneq \Delta_1(\Mh)$,
the faces collected in the set $\Fh \coloneq \Delta_2(\Mh)$,
and the elements collected in the set $\Th \coloneq \Delta_3(\Mh)$.
For all $f\in\Mh$, we select a point $\bvec{x}_f\in f$ which, when $\Mh$ belongs to a refined mesh sequence, is assumed at a distance from the boundary of $f$ comparable to the meshsize.

If $f\in\Delta_d(\Mh)$ and $d'\le d$ is an integer, we denote by $\Delta_{d'}(f)$ the set of subcells of $f$ of dimension $d'$.
Hence, if $n=d=3$, so that $f=T\in\Th$ is a polyhedral element of the mesh, $f\in\Delta_{d'}(T)$ is a vertex of $T$ if $d'=0$, an edge of $T$ if $d'=1$, a polygonal face of $T$ if $d'=2$, or $T$ itself if $d'=3$.

For future use, we note the following property.

\begin{lemma}[Projectors on subspaces of differential forms]
  Let $(k,d)$ be integers such that $k\le d\le n$, $f\in\Delta_d(\Mh)$, and $\mathcal X$ be a closed subspace of $L^2\Lambda^{d-k}(f)$.
  Then, it holds:
  For all $\omega\in L^2\Lambda^k(f)$ and all $\mu\in\mathcal X$,
  \begin{equation}\label{eq:remove.projector}
    \int_f \star^{-1}\pi_{\mathcal X}(\star\omega)\wedge \mu
    =\int_f \mu\wedge \star\pi_{\mathcal X}(\star\omega)=
    \int_f \omega\wedge \mu.
  \end{equation}
\end{lemma}

\begin{proof}
  The first relation in \eqref{eq:remove.projector} follows from \eqref{eq:commut.star.wedge}. To prove the second relation, we write
  \begin{align*}
    \int_f \mu\wedge \star\pi_{\mathcal X}(\star\omega)&=\int_f \cancel{\pi_{\mathcal X}}(\star\omega)\wedge\star\mu
    =\int_f \mu\wedge(\star\star\omega)=\int_f \omega\wedge\mu,
  \end{align*}
  where the first equality follows from \eqref{eq:commut.star.wedge} (with $(\omega,\mu)\gets(\pi_{\mathcal X}(\star\omega),\mu)$), the cancellation of the projector is justified by its definition \eqref{eq:def.piX}, the second equality is obtained using \eqref{eq:commut.star.wedge} again, and the conclusion follows from \eqref{eq:star.star} and the anticommutativity \eqref{eq:wedge:anticommutativity} of $\wedge$.
\end{proof}

\subsection{Local polynomial spaces of differential forms}

Let $f\in\Delta_d(\Mh)$, $0\le d\le n$.
For any integer $r\ge 0$, we denote by $\Poly{r}{} \Lambda^\ell(f)$ the space of polynomial $\ell$-forms of total degree $\le r$ on $f$.
We also adopt the standard convention $\Poly{-1}{}\Lambda^\ell(f)\coloneq\{0\}$.
We denote by $\lproj{r}{f}{\ell} : L^2\Lambda^\ell(f) \to \Poly{r}{}\Lambda^\ell(f)$ the $L^2$-orthogonal projector onto $\Poly{r}{}\Lambda^\ell(f)$, defined by \eqref{eq:def.piX} with $\mathcal X=\Poly{r}{}\Lambda^\ell(f)$.

The Koszul differential on $f$ (translated by $\bvec{x}_f$) is denoted by $\kappa$ so that, for all $\omega\in\Lambda^\ell(f)$, $\kappa\omega\in\Lambda^{\ell-1}(f)$ satisfies $(\kappa\omega)_{\bvec{x}}(\bvec{v}_1,\ldots,\bvec{v}_{\ell-1}) = \omega_{\bvec{x}}(\bvec{x}-\bvec{x}_f,\bvec{v}_1,\ldots,\bvec{v}_{\ell-1})$ for all vectors $\bvec{v}_1,\ldots,\bvec{v}_{\ell-1}$ tangent to $f$.
For any $f\in\Delta_d(\Mh)$, $1\le d\le n$, any integer $\ell\in[0,d]$, and any polynomial degree $r\ge 0$, we define the Koszul complement space as
\begin{equation}\label{def:koszul.space}
  \Koly{r}{\ell}(f)\coloneq\kappa\Poly{r-1}{}\Lambda^{\ell+1}(f).
\end{equation}
The indices $r$ and $\ell$ in this notation serve as a reminder that elements in $\Koly{r}{\ell}(f)$ are polynomial $\ell$-forms of polynomial degree $r$. Note also that, since $\Poly{-1}{}\Lambda^\ell(f)=\{0\}$ and  $\Lambda^{d+1}(f)=\{0\}$, we have
\begin{equation}\label{eq:Koly.0.ell=Koly.r.d=0}
  \text{
    $\Koly{0}{\ell}(f)=\Koly{r}{d}(f)=\{0\}$ for all $\ell$ and all $r$, respectively.
  }
\end{equation}
Moreover, since $\kappa\Lambda^0(f)=\{0\}$, we adopt the convention $\Koly{r}{-1}(f) \coloneq \{0\}$ for all $r$.
We denote by $\kproj{r}{f}{\ell}$ the $L^2$-orthogonal projector $L^2\Lambda^\ell(f)\to\Koly{r}{\ell}(f)$, defined by \eqref{eq:def.piX} with $\mathcal X=\Koly{r}{\ell}(f)$.

For all integers $r\ge 0$ and $\ell\in [0,d]$, the following direct decomposition holds (see \cite[Eq.~(3.11)]{Arnold.Falk.ea:06} for $\ell\ge 1$, the case $\ell=0$ can be directly checked):
\begin{subequations}\label{eq:decomposition.Pr}
  \begin{align}\label{eq:decomposition.Pr:ell=0}
    \Poly{r}{}\Lambda^0(f)
    &= \Poly{0}{}\Lambda^0(f) \oplus \Koly{r}{0}(f),
    \\\label{eq:decomposition.Pr:ell>=1}
    \Poly{r}{}\Lambda^\ell(f)
    &= \rd\Poly{r+1}{}\Lambda^{\ell-1}(f) \oplus \Koly{r}{\ell}(f)\quad\mbox{ if $\ell\ge 1$}.
  \end{align}
\end{subequations}
Since $\rd\circ\rd=0$ and $\rd\Poly{0}{}\Lambda^0(f)=\{0\}$ (since the coefficients of the form are constant), this shows that
\begin{equation}\label{eq:dP.dkappaP}
  \rd\Poly{r}{}\Lambda^{\ell}(f)=\rd\Koly{r}{\ell}(f).
\end{equation}
Applying this relation to $(r+1,\ell-1)$ instead of $(r,\ell)$ and recalling that $\rd$ is one-to-one on $\Koly{r+1}{\ell-1}(f)$ (see \cite[Theorem 3.2]{Arnold.Falk.ea:06}), this shows that, for $\ell\ge 1$, the following mapping is an isomorphism:
\begin{equation}\label{eq:isomorphism.Pr.koszul}
  \begin{aligned}
    \Koly{r+1}{\ell-1}(f) \times \Koly{r}{\ell}(f)\xrightarrow{\cong}{}& \Poly{r}{}\Lambda^{\ell}(f),\\
    (\mu,\nu) \mapsto{}& \rd\mu+\nu.
  \end{aligned}
\end{equation}
\begin{example}[Interpretation in terms of vector proxies]
  \label{example:interpretation.spaces}
  In the case $n=3$, thanks to the links between differential forms and vector proxies (see Appendix~\ref{sec:appendix}), we can associate to each space of polynomial differential forms a space of (vector- or scalar-valued) polynomial fields.
  Let us consider decomposition~\eqref{eq:decomposition.Pr:ell>=1}.
  We denote by $f_d$ a $d$-cell of $\Mh$,
  and we use a notation analogous to that of~\cite{Di-Pietro.Droniou:23*1} for polynomial spaces and vector calculus differential operators (with the exception that polynomial degrees are in subscripts instead of superscripts).
  Then, by definition~\eqref{def:koszul.space} of the Koszul space, when $f_3 = T \in \Th = \Delta_3(\Mh)$ is a mesh element, we have
  $$
  \begin{aligned}
    &\rd\Poly{r+1}{}\Lambda^0(f_3) \leftrightarrow \Goly{r}(T) \coloneq \GRAD \Poly{r+1}{}(T),
    &\quad& \Koly{r}{1}(f_3) \leftrightarrow \cGoly{r}(T) \coloneq (\bvec x - \bvec x_T) \times \vPoly{r-1}{}(T), \\
    &\rd\Poly{r+1}{}\Lambda^1(f_3) \leftrightarrow \Roly{r}(T) \coloneq \CURL \vPoly{r+1}{}(T),
    &\quad& \Koly{r}{2}(f_3) \leftrightarrow \cRoly{r}(T) \coloneq (\bvec x - \bvec x_T) \Poly{r-1}{}(T), \\
    &\rd\Poly{r+1}{}\Lambda^2(f_3) \leftrightarrow \DIV \Poly{r+1}{}(T) = \Poly{r}{}(T),
    &\quad& \Koly{r}{3}(f_3) = \{0\},
  \end{aligned}
  $$
  where the first identity in the last line results from the surjectivity of the divergence operator.

  On the other hand, when $f_2 = F \in \Fh = \Delta_2(\Mh)$ is a mesh face, we obtain the following pair of possible correspondences:
  \begin{equation}\label{eq:poly.space.proxy.n=2.a}
    \begin{aligned}
      \rd\Poly{r+1}{}\Lambda^0(f_2) &\leftrightarrow \Goly{r}(F) \coloneq \GRAD_F \Poly{r+1}{}(F),
      \\
      \Koly{r}{1}(f_2) &\leftrightarrow \cGoly{r}(F) \coloneq (\bvec x - \bvec x_F)^\perp \Poly{r-1}{}(F)
    \end{aligned}
  \end{equation}
  or
  \begin{equation}\label{eq:poly.space.proxy.n=2.b}
    \begin{aligned}
      \rd\Poly{r+1}{}\Lambda^0(f_2) &\leftrightarrow \Roly{r}(F) \coloneq \VROT_F \Poly{r+1}{}(F),
      \\
      \Koly{r}{1}(f_2) &\leftrightarrow \cRoly{r}(F) \coloneq (\bvec x - \bvec x_F)\Poly{r+1}{}(F),
    \end{aligned}
  \end{equation}
  where, for any $\bvec{v}\in\Real^2$, $\bvec v^\perp = \rotation{-\pi/2}\bvec v$ is the clockwise rotation of $\bvec v$ with respect to the orientation of $F$.
  The existence of two possible correspondences between polynomial $1$-forms and polynomial vector fields is to due
  to the fact that, when $d=2$, one can identify a $1$-form either with a vector field $\bvec v = (v_1,v_2)$ or with its rotation through a right angle (cf.~\cite[Chapter~6]{Arnold:18});
  in particular, we choose to identify it with the clockwise rotation $\bvec v^\perp = (v_2, -v_1)$ (see Appendix~\ref{sec:appendix} for further details).
  By \eqref{eq:Koly.0.ell=Koly.r.d=0}, we have $\Koly{r}{2}(f_2)=\{0\}$ and, according to whether we consider the vector proxy leading to~\eqref{eq:poly.space.proxy.n=2.a} or~\eqref{eq:poly.space.proxy.n=2.b},
  \[
  \rd\Poly{r+1}{}\Lambda^1(f_2) \leftrightarrow \ROT_F \vPoly{r+1}{}(F) = \Poly{r}{}(F)
  \quad\text{or}\quad  \rd\Poly{r+1}{}\Lambda^1(f_2) \leftrightarrow \DIV_F \vPoly{r+1}{}(F) = \Poly{r}{}(F).
  \]
  Hence, since both $1$-forms and $2$-forms in $\Real^3$ can be identified with vector fields, and accounting for the two-fold identification of $1$-forms in $\Real^2$,
  the decomposition~\eqref{eq:decomposition.Pr:ell>=1} reads, in terms of proxies,
  $$
  \vPoly{r}{}(f_d) = \Goly{r}(f_d) \oplus \cGoly{r}(f_d) = \Roly{r}(f_d) \oplus \cRoly{r}(f_d),\quad d\in\{2,3\},
  $$
  i.e., the same expressions as~\cite[Eqs.~(2.4) and~(2.6)]{Di-Pietro.Droniou:23*1}.
  On the other hand, concerning $0$-forms, the decomposition~\eqref{eq:decomposition.Pr:ell=0} reads, in terms of proxies,
  $$\Poly{r}{}(f_d) = \Poly{0}{}(f_d) \oplus \Poly{r}{\flat}(f_d),\quad d\in \{0,\ldots,3\},$$
  where we have introduced the notation $\Poly{r}{\flat}(f) \coloneq (\bvec{x}-\bvec{x}_f)\cdot\vPoly{r-1}{}(f)$ for any $f\in\Delta_d(\Mh)$.
\end{example}

\subsection{Trimmed local polynomial spaces}

We recall the following local trimmed polynomial spaces (see e.g.~\cite[Theorem~3.5]{Arnold.Falk.ea:06}):
For any $f\in~\Delta_d(\Mh)$, $1\le d\le n$,
\begin{subequations}\label{def:trimmed.spaces}
  \begin{align}\label{eq:trimmed.spaces:ell=0}
    \Poly{r}{-}\Lambda^0(f) & = \Poly{r}{}\Lambda^0(f),
    \\\label{eq:trimmed.spaces:ell>=1}
    \Poly{r}{-}\Lambda^{\ell}(f) & = \rd\Poly{r}{}\Lambda^{\ell-1}(f) \oplus \Koly{r}{\ell}(f)\qquad\text{for $\ell\ge 1$}.
  \end{align}
\end{subequations}
In \eqref{eq:trimmed.spaces:ell>=1}, comparing with the decompositions \eqref{eq:decomposition.Pr}, we have decreased by one the polynomial degree of the first space in the direct sum.
Note that this definition leads to the choice
\begin{equation}\label{eq:trimmed.spaces:0-cells}
  \Poly{r}{-}\Lambda^0(f)\coloneq\Poly{r}{}\Lambda^0(f)\cong\Real
  \qquad\forall f\in\Delta_0(\Mh).
\end{equation}
The $L^2$-orthogonal projector $L^2\Lambda^\ell(f)\to\Poly{r}{-}\Lambda^\ell(f)$ is denoted by $\trimproj{r}{f}{\ell}$, and is defined by \eqref{eq:def.piX} with $\mathcal X=\Poly{r}{-}\Lambda^{\ell}(f)$.

Let us note a few properties of trimmed polynomial spaces.
For $r=0$, only the space \eqref{eq:trimmed.spaces:ell=0} is non-trivial, that is, $\Poly{0}{-}\Lambda^\ell(f)=\{0\}$ if $\ell\in [1,d]$. Applying, if $r\ge 1$ and $\ell\ge 1$, \eqref{eq:decomposition.Pr:ell>=1} with $r-1$ instead of $r$ and noticing that $\Koly{r-1}{\ell}(f)\subset\Koly{r}{\ell}(f)$, we obtain the equality
\begin{equation}\label{eq:trimmed.between}
  \Poly{r}{-}\Lambda^{\ell}(f)=\Poly{r-1}{}\Lambda^\ell(f)+\Koly{r}{\ell}(f).
\end{equation}
This equality, which obviously also holds for $\ell=0$ (see \eqref{eq:decomposition.Pr:ell=0}), shows that trimmed polynomial spaces sit between full polynomial spaces:
\[
\Poly{r-1}{}\Lambda^{\ell}(f)\subset\Poly{r}{-}\Lambda^{\ell}(f)\subset\Poly{r}{}\Lambda^{\ell}(f).
\]

Recalling that $\Koly{r}{d}(f)=\{0\}$ and that $\rd\Poly{r}{}\Lambda^{d-1}(f) = \Poly{r-1}{}\Lambda^d(f)$ (by exactness of the tail of the polynomial de Rham sequence \cite[Corollary 7.3]{Arnold:18}), it holds
\begin{equation}\label{eq:trimmed.k=d}
  \Poly{r}{-}\Lambda^d(f)=\Poly{r-1}{}\Lambda^d(f).
\end{equation}

Applying \eqref{eq:dP.dkappaP} with $\ell-1$ instead of $\ell$, we moreover have
\begin{equation}\label{def:trimmed.spaces.bis}
  \Poly{r}{-}\Lambda^{\ell}(f) = \rd\Koly{r}{\ell-1}(f)+ \Koly{r}{\ell}(f) \qquad\text{for $\ell\ge 1$}.
\end{equation}
Since $\rd$ is one-to-one on $\Koly{r}{\ell-1}(f)$, this gives the following isomorphism, whenever $\ell\ge 1$:
\begin{equation}\label{eq:isomorphism.Prtrimmed.koszul}
  \begin{aligned}
    \Koly{r}{\ell-1}(f) \times \Koly{r}{\ell}(f)\xrightarrow{\cong}{}& \Poly{r}{-}\Lambda^{\ell}(f),\\
    (\mu,\nu) \mapsto{}& \rd\mu+\nu.
  \end{aligned}
\end{equation}
\begin{example}[Interpretation of~\eqref{eq:trimmed.spaces:ell>=1} in terms of vector proxies]
  Let $n=3$. For $d\in \{2,3\}$, denoting again by $f_d$ a $d$-cell,
  we define the (local) {Nédélec} and {Raviart--Thomas} spaces
  $$
  \Nedelec{r}(f_d) \coloneq \Goly{r-1}(f_d) + \cGoly{r}(f_d),
  \qquad
  \RaviartThomas{r}(f_d) \coloneq \Roly{r-1}(f_d) + \cRoly{r}(f_d).
  $$
  Notice that, when $d=3$, the Nédélec and Raviart--Thomas spaces can be obtained as polynomial spaces of vector proxies of \eqref{eq:trimmed.spaces:ell>=1} for $\ell = 1$ and $\ell = 2$, respectively.
  On the other hand, when considering $d=2$, both spaces can be obtained by taking
  the same value $\ell = 1$ in~\eqref{eq:trimmed.spaces:ell>=1}.
  Again, this is linked to the two-fold interpretation of $1$-forms in terms of vector proxies in $\Real^2$, discussed in Example~\ref{example:interpretation.spaces}, and corresponds to the well-known fact that two-dimensional N\'ed\'elec elements coincide with two-dimensional Raviart--Thomas elements rotated by a right angle.
\end{example}
The following result generalises \cite[Proposition~8]{Di-Pietro.Droniou:23*1}.
\begin{lemma}[Traces of trimmed polynomial spaces]\label{lem:traces.trimmed}
  The trace preserves trimmed spaces: For all integers $d\in[0,n]$, $d'\in [0,d]$ and $\ell\in [0,d']$, all $f\in \Delta_d(\Mh)$, and all $f'\in\Delta_{d'}(f)$, we have
  \[
  \tr_{f'}\Poly{r}{-}\Lambda^\ell(f)\subset\Poly{r}{-}\Lambda^\ell(f').
  \]
\end{lemma}

\begin{proof}
  We first notice that the case $\ell=0$ is obvious since, in this case, trimmed spaces are full polynomial spaces (see \eqref{eq:trimmed.spaces:ell=0}), and the trace preserves full polynomial spaces. We therefore assume in the rest of the proof that $\ell\ge 1$.
  As the Koszul operators on differential forms on $f$ and $f'$ are not the same (due to the translation by $\bvec{x}_f$ and $\bvec{x}_{f'}$, respectively), we temporarily denote them in this proof by $\kappa_f$ and $\kappa_{f'}$.

  The trace is a pullback, so it commutes with $\rd$, and we thus have
  \[
  \tr_{f'}(\rd\Poly{r}{}\Lambda^{\ell-1}(f))=\rd(\tr_{f'}\Poly{r}{}\Lambda^{\ell-1}(f))\subset\rd\Poly{r}{}\Lambda^{\ell-1}(f'),
  \]
  where the inclusion holds since the trace preserves full polynomial spaces.
  Given the definition \eqref{eq:trimmed.spaces:ell>=1} of the trimmed spaces, the lemma follows if we show that
  \begin{equation}\label{eq:trace.to.prove}
    \tr_{f'}\Koly{r}{\ell}(f)\subset \Poly{r}{-}\Lambda^\ell(f')
    \overset{\eqref{eq:trimmed.between}}=
    \Poly{r-1}{}\Lambda^\ell(f')+\Koly{r}{\ell}(f').
  \end{equation}
  Let $\omega\in \Poly{r-1}{}\Lambda^{\ell+1}(f)$. The definitions of $\tr_{f'}$ and $\kappa_f$ give, for any $\bvec{x}\in f'$ and $\bvec{v}_1,\ldots,\bvec{v}_\ell$ tangent to $f'$,
  \begin{align*}
    \tr_{f'}(\kappa_f\omega)_{\bvec{x}}(\bvec{v}_1,\ldots,\bvec{v}_\ell) ={}& \omega_{\bvec{x}}(\bvec{x}-\bvec{x}_f,\bvec{v}_1,\ldots,\bvec{v}_\ell)\\
    ={}& \omega_{\bvec{x}}(\bvec{x}_{f'}-\bvec{x}_f,\bvec{v}_1,\ldots,\bvec{v}_\ell)+\omega_{\bvec{x}}(\bvec{x}-\bvec{x}_{f'},\bvec{v}_1,\ldots,\bvec{v}_\ell)\\
    ={}&\alpha_{\bvec{x}}(\bvec{v}_1,\ldots,\bvec{v}_\ell)+(\kappa_{f'}\tr_{f'}\omega)_{\bvec{x}}(\bvec{v}_1,\ldots,\bvec{v}_\ell),
  \end{align*}
  where we have used the linearity of $\omega_{\bvec{x}}$ with respect to its first argument to obtain the second equality, and introduced the differential form $\alpha\coloneq \omega(\bvec{x}_{f'}-\bvec{x}_f,\cdot)$ in the third equality. Hence, $\tr_{f'}(\kappa_f\omega)=\alpha+\kappa_{f'}\tr_{f'}\omega$, which proves \eqref{eq:trace.to.prove} since $\alpha\in\Poly{r-1}{}\Lambda^{\ell}(f')$ (as $\bvec{x}_{f'}-\bvec{x}_f$ is constant) and $\tr_{f'}\omega\in\Poly{r-1}{}\Lambda^{\ell+1}(f')$.
\end{proof}


\section{Discrete de Rham complex}\label{sec:ddr}

We define in this section a discrete counterpart of the de Rham complex of differential forms \eqref{eq:diff.de.rham} in the spirit of \cite{Di-Pietro.Droniou.ea:20,Di-Pietro.Droniou:23*1}.
Let, from this point on, an integer $r\ge 0$ be fixed, corresponding to the polynomial degree of the discrete sequence.
The general idea is, for each form degree $k\in[0,n]$, to select the polynomial components of the discrete spaces in order to reconstruct, on each $d$-cell $f$ and iteratively on the dimension $d$:
\begin{itemize}
\item A \emph{discrete exterior derivative} in $\Poly{r}{}\Lambda^{k+1}(f)$ that can reproduce exactly the exterior derivative of differential forms in $\Poly{r+1}{-}\Lambda^k(f)$;
\item Based on this discrete exterior derivative and on traces on $(d-1)$-cells (either directly available or reconstructed), a \emph{discrete potential} in $\Poly{r}{}\Lambda^k(f)$ that can reproduce exactly differential forms belonging to this same space.
\end{itemize}

\subsection{Definition}

\subsubsection{Discrete spaces}

The discrete counterpart $\ul{X}^k_{r,h}$ of the space $H\Lambda^k(\Omega)$, $0\le k \le n$, is defined as
\begin{equation}\label{eq:global.space}
  \ul X_{r,h}^k \coloneq \bigtimes_{d=k}^n \bigtimes_{f\in \Delta_d(\Mh)}\Poly{r}{-}\Lambda^{d-k}(f),
\end{equation}
with $\times$ denoting the Cartesian product.
We define the restrictions of the global space \eqref{eq:global.space} to a mesh entity or its boundary as follows:
For~all~integers $k$ and $d$ such that $0\le k\le d\le n$ and all $f\in\Delta_d(\Mh)$,
\[
\text{%
  $\ul X_{r,f}^k \coloneq \bigtimes_{d'=k}^{d}
  \bigtimes_{f'\in\Delta_{d'}(f)} \Poly{r}{-}\Lambda^{d'-k}(f')$
  and
  $\ul X_{r,\partial f}^k\coloneq\bigtimes_{d'=k}^{d-1}
  \bigtimes_{f'\in\Delta_{d'}(f)} \Poly{r}{-}\Lambda^{d'-k}(f')$
  if $d\ge 1$.
}
\]
We shall use the notation $\ul\omega_h = (\omega_f)_{f\in\Delta_d(\Mh),\,d\in[k,n]}\in\ul X_{r,h}^k$ for a generic element of the global discrete space of $k$-forms and  $\ul\omega_f = (\omega_{f'})_{f'\in\Delta_{d'}(f),\,d'\in[k,d]}\in \ul X_{r,f}^k$ (resp., $\ul\omega_{\partial f}= (\omega_{f'})_{f'\in\Delta_{d'}(f),\,d'\in[k,d-1]}\in\ul X_{r,\partial f}^k$) for its restriction to $f$ (resp., $\partial f$), obtained collecting the components on the mesh entities $f'\in\Delta_{d'}(f)$, $d'\in[k,d]$ (resp., $d'\in[k,d-1]$).
As a generic convention in this article, underlined letters denote spaces or vectors made of polynomial components on mesh entities.
Table \ref{tab:local.spaces} gives an overview of the polynomial unknowns in $\ul X_{r,f}^k$, along with their vector proxies, in dimensions 0 to 3.

\begin{table}\centering
  \renewcommand{\arraystretch}{1.2}
  \begin{tabular}{c|cccc}
    \toprule
    \diagbox[font=\footnotesize]{$k$}{$d$}  & $0$ & $1$ & $2$ & $3$ \\
    \midrule
    0 & $\Real = \Poly{r}{}\Lambda^0(f_0)$ & $\Poly{r-1}{}\Lambda^1 (f_1)$ & $\Poly{r-1}{} \Lambda^2(f_2)$ & $\Poly{r-1}{} \Lambda^3 (f_3)$ \\
    1 & & $\Poly{r}{}\Lambda^0(f_1)$ & $\Poly{r}{-}\Lambda^1(f_2)$ & $\Poly{r}{-}\Lambda^2(f_3)$ \\
    2 & & & $\Poly{r}{}\Lambda^0(f_2)$ & $\Poly{r}{-}\Lambda^1(f_3)$ \\
    3 & & & & $\Poly{r}{}\Lambda^0(f_3)$ \\
    \midrule[1pt]
    \diagbox[font=\footnotesize]{$k$}{$d$}  & $0$ & $1$ & $2$ & $3$ \\
    \midrule
    0 & $\Real = \Poly{r}{}(f_0)$ & $\Poly{r-1}{}(f_1)$ & $\Poly{r-1}{}(f_2)$ & $\Poly{r-1}{}(f_3)$ \\
    1 & & $\Poly{r}{}(f_1)$ & $\RaviartThomas{r}(f_2)$ & $\RaviartThomas{r}(f_3)$ \\
    2 & & & $\Poly{r}{}(f_2)$ & $\Nedelec{r}(f_3)$ \\
    3 & & & & $\Poly{r}{}(f_3)$ \\
    \bottomrule
  \end{tabular}
  \caption{Polynomial components attached to each mesh entity $f_d$ of dimension $d\in\{0,\ldots,3\}$ for the space $\ul X_{r,h}^k$ for $k\in\{0,\ldots,3\}$ (top) and counterpart through vector proxies (bottom).\label{tab:local.spaces}}
\end{table}

\begin{remark}[Choice of polynomial components]
  The choice of using in \eqref{eq:global.space} component spaces spanned by $(d-k)$-forms instead of $k$-forms is motivated by the desire to recover the DDR sequence of \cite{Di-Pietro.Droniou:23*1} through vector proxies; see Example \ref{ex:ddr.interpretation} below.
  Applying the Hodge-star operator to these components could, in the Euclidean setting at least, enable us to consider components in $\Poly{r}{-}\Lambda^k(f)$ (the situation is however different when designing the method on manifolds \cite{Droniou.Hanot.ea:23}).

  Notice that it is by no means clear that full polynomial spaces could be used instead of trimmed spaces while ensuring that the discrete cohomology is isomorphic to the continuous one (cf.\ Theorem \ref{thm:cohomology:ddr} below).
  As a matter of fact, as noticed in \cite[Section 4.2]{Di-Pietro.Droniou.ea:20}, this cannot hold in dimension $n=2$ on a simply connected polygon.
  In dimension $n=3$, it has been shown in \cite{Wang.Wang.ea:21} that, in the lowest-order case, the dimension of the kernel of the discrete gradient depends on the number of edges of the polygon, which clearly prevents one from establishing an isomorphism with the de Rham cohomology.
\end{remark}

\begin{remark}[Virtual spaces]
  It is possible to identify virtual spaces underlying $\underline{X}_{r,h}^k$ and its restrictions to mesh faces $f$ in the spirit of \cite{Beirao-da-Veiga.Brezzi.ea:13}.
  These spaces, however, play no role in the following discussion, so we do not present them here to avoid confusion.
  In the present framework, the connections between polynomial components attached to a mesh cell and its boundary are not realised by a virtual function, but rather by the reconstructions presented in Section \ref{sec:ddr:potentials.differentials} below.
\end{remark}

\begin{remark}[Comparison with trimmed finite element sequences]
  A detailed comparison between the number of degrees of freedom for the DDR and classical trimmed finite element sequences for $n=3$ has been made in \cite[Table 3]{Di-Pietro.Droniou:23*2}.
  This comparison shows that the DDR complex without serendipity reduction has slightly more degrees of freedom than trimmed finite elements on tetrahedra, but fewer on hexahedra.
  This table also shows that the difference on tetrahedra can be slimmed down (and the advantage on hexahedra increased) using serendipity to reduce face and element degrees of freedom.
\end{remark}

\subsubsection{Interpolators and interpretation of the polynomial components}

The precise meaning of the components in each DDR space is provided by the corresponding interpolator.
For $f\in\Delta_d(\Mh)$ and $k\le d$, the interpolator $\ul I_{r,f}^k:C^0\Lambda^k(\overline{f})\to\ul X_{r,f}^k$ is defined by:
For all $\omega\in C^0\Lambda^k(\overline{f})$,
\begin{equation}\label{eq:interpolator}
  \ul I_{r,f}^k\omega \coloneq (\trimproj{r}{f'}{d'-k}(\star\tr_{f'}\omega) )_{f'\in\Delta_{d'}(f),\,d'\in [k,d]}.
\end{equation}
In other words, a discrete $k$-form on the mesh is made of polynomial forms attached to each mesh entity of dimension $d\ge k$; on each entity, the form is of degree $d-k$ as it corresponds to the Hodge star of an underlying $k$-form. The Hodge star operator is used in the definition of the polynomial components to ensure that the full space $\Poly{r}{}\Lambda^0(f)$ (see \eqref{eq:trimmed.spaces:ell=0}) is attached to the lowest-dimensional cells $f\in\Delta_k(\Mh)$.

\subsubsection{Local discrete potentials and discrete exterior derivative}\label{sec:ddr:potentials.differentials}

Let $0\le k\le n$ be a fixed integer. For all $f\in\Delta_d(\Mh)$ with $d\ge k$, we define the \emph{discrete potential} $\Pot{r}{f}{k} : \ul X_{r,f}^k \to \Poly{r}{} \Lambda^k(f)$ and, if $d\ge k+1$, the \emph{discrete exterior derivative} $\rd_{r,f}^{k}:\ul{X}^k_{r,f}\to\Poly{r}{}\Lambda^{k+1} (f)$ recursively on the dimension $d$ as follows:
\begin{itemize}
\item If $d=k$, then the discrete potential on $f$ is directly given by the component of $\ul\omega_f$ on $f$:
  \begin{equation}\label{eq:discrete.potential:d=k}
    \Pot{r}{f}{k}\ul\omega_f \coloneq \star^{-1}\omega_f \in \Poly{r}{} \Lambda^d(f).
  \end{equation}
\item
  If $k+1 \le d \le n$:
  \begin{enumerate}
  \item First, the discrete exterior derivative is defined by: For all $\ul\omega_f\in\ul X^k_{r,f}$,
    \begin{multline}\label{eq:discrete.exterior.derivative:ddr}
      \int_f \rd_{r,f}^{k} \ul\omega_f \wedge \mu
      = (-1)^{k+1} \int_f \star^{-1}\omega_f\wedge \rd\mu
      + \int_{\partial f} \Pot{r}{\partial f}{k}\ul\omega_{\partial f} \wedge \tr_{\partial f}{\mu}
      \\
      \forall \mu \in \Poly{r}{}\Lambda^{d-k-1}(f),
    \end{multline}
    where we have introduced the piecewise polynomial boundary potential $\Pot{r}{\partial f}{k}:\ul X_{r,\partial f}^k\to\Lambda^k(\partial f)$ such that $(\Pot{r}{\partial f}{k})_{|f'} \coloneq \Pot{r}{f'}{k}$ for all $f'\in\Delta_{d-1}(f)$ ($\Pot{r}{f'}{k}$ being the discrete potential on the $(d-1)$-cell $f'$ defined at the previous step).
  \item Then, the discrete potential on the $d$-cell $f$ is given by: For all $\ul\omega_f\in\ul X^k_{r,f}$,
    \begin{multline}\label{eq:discrete.potential}
      (-1)^{k+1}\int_f \Pot{r}{f}{k} \ul\omega_f \wedge (\rd\mu+\nu)\\
      = \int_f \rd_{r,f}^{k} \ul\omega_f \wedge \mu
      -  \int_{\partial f}  \Pot{r}{\partial f}{k}\ul\omega_{\partial f} \wedge \tr_{\partial f}{\mu}
      + (-1)^{k+1}\int_f \star^{-1}\omega_f \wedge\nu
      \\
      \forall (\mu,\nu) \in \Koly{r+1}{d-k-1}(f) \times \Koly{r}{d-k}(f).
    \end{multline}
  \end{enumerate}
\end{itemize}

Some remarks are in order.

\begin{remark}[Boundary integration and orientation]
  Above and in the rest of the paper, any integral $\int_{\partial f}$ on the boundary a cell $f$ is considered according to the orientation induced
  by the cell $f$. As a consequence, if $\varepsilon_{f\!f’}\in\{-1,1\}$ denotes the orientation of $f’\in\Delta_{d-1}(f)$ relative to $f$, we have
  \begin{equation}\label{eq:integral.relative.orientation}
    \int_{\partial f} \bullet= \sum_{f’\in\Delta_{d-1}(f)}\varepsilon_{f\!f’}\int_{f’}\bullet.
  \end{equation}
\end{remark}

\begin{remark}[Definitions \eqref{eq:discrete.exterior.derivative:ddr} and \eqref{eq:discrete.potential}]\label{rem:riesz.d.pot}
  The fact that condition \eqref{eq:discrete.exterior.derivative:ddr} defines $\rd_{r,f}^{k} \ul\omega_f$ uniquely is an immediate consequence of the Riesz representation theorem for $\Poly{r}{}\Lambda^{k+1}(f)$ equipped with the \mbox{$L^2$-product} $(\rho,\beta)\ni L^2\Lambda^{k+1}(f)\times L^2\Lambda^{k+1}(f)\mapsto\int_f\rho\wedge\star\beta\in\Real$, after observing that \eqref{eq:discrete.exterior.derivative:ddr} can be equivalently reformulated as follows (notice the change in the degree of the test differential form, with $\beta$ below corresponding to $\star^{-1}\mu$ in \eqref{eq:discrete.exterior.derivative:ddr}):
  \[
  \int_f\rd_{r,f}^{k} \ul\omega_f\wedge\star\beta
  = (-1)^{k+1}\int_f \omega_f\wedge\star \rd\star\beta
  + \int_{\partial f} \Pot{r}{\partial f}{k}\ul\omega_{\partial f}\wedge \tr_{\partial f}\star\beta
  \quad\forall\beta\in\Poly{r}{}\Lambda^{k+1}(f),
  \]
  where we have additionally used \eqref{eq:commut.star.wedge} for the first term in the right-hand side.
  Similar considerations apply to the definition \eqref{eq:discrete.potential} of $\Pot{r}{f}{k}$, applying the isomorphism \eqref{eq:isomorphism.Pr.koszul} with $\ell=d-k\ge 1$.

  Notice that one cannot substitute \eqref{eq:discrete.exterior.derivative:ddr} into \eqref{eq:discrete.potential}, as the polynomial degree of the test function $\mu$ in this second relation is one unit higher.
  In view of \eqref{eq:proj.potential} below, the potential reconstruction can be regarded as a higher-order \emph{enhancement} of $\star\omega_f$ exploiting the additional information provided by the components on the boundary of the subcells.
\end{remark}

\begin{remark}[Validity of \eqref{eq:discrete.potential}]\label{rem:discrete.potential:validity}
  For $k+1\le d\le n$, equation \eqref{eq:discrete.potential} actually holds for all $\mu\in\Poly{r+1}{-}\Lambda^{d-k-1}(f)$.
  To prove this assertion, since \eqref{eq:discrete.potential} holds for $\mu\in\Koly{r+1}{d-k-1}(f)$, it suffices to show that it also holds for $\nu=0$ and $\mu$ belonging to $\Poly{0}{}\Lambda^0(f)$ if $d=k+1$ (see \eqref{eq:trimmed.spaces:ell=0} and \eqref{eq:decomposition.Pr:ell=0}) or $\rd\Poly{r+1}{}\Lambda^{d-k-2}(f)$ if $d\ge k+2$ (see \eqref{eq:trimmed.spaces:ell>=1}).
  In both cases, we have $\rd\mu=0$, so that the left-hand side of \eqref{eq:discrete.potential} vanishes; since $\mu\in\Poly{r}{}\Lambda^{d-k-1}(f)$, the right-hand side of \eqref{eq:discrete.potential} also vanishes due to the definition~\eqref{eq:discrete.exterior.derivative:ddr} of the discrete exterior derivative, which concludes the argument.
\end{remark}

\begin{remark}[Potential for $k=0$]\label{rem:improved.Pot.k=0}
  In the case $k=0$, we can define an improved potential $\Pot{r+1}{f}{0}:\ul X_{r,f}^0\to \Poly{r+1}{}\Lambda^0(f)$ of polynomial degree $r+1$ (instead of $r$) as follows:
  For all $\ul\omega_f\in \ul X_{r,f}^0$,
  \begin{itemize}
  \item If $d=0$, then $\Pot{r+1}{f}{0}\ul\omega_f=\star^{-1}\omega_f\in\Poly{r}{}\Lambda^0(f)\cong \Real\cong \Poly{r+1}{}\Lambda^{0}(f)$ (since $f$ has dimension $0$);
  \item If $1\le d\le n$,
    \begin{equation}\label{eq:discrete.potential.k=0}
      - \int_f \Pot{r+1}{f}{0} \ul\omega_f \wedge \rd\mu= \int_f \rd_{r,f}^{0} \ul\omega_f \wedge \mu
      -  \int_{\partial f}  \Pot{r+1}{\partial f}{0}\ul\omega_{\partial f} \wedge \tr_{\partial f}{\mu}\qquad
      \forall \mu \in \Koly{r+2}{d-1}(f).
    \end{equation}
  \end{itemize}
  This definition is justified by the isomorphism \eqref{eq:isomorphism.Pr.koszul} with $\ell=d$ and $r+1$ instead of $r$ (recalling that $\Koly{r+1}{d}(f)=\{0\}$), and it can easily be checked, testing \eqref{eq:discrete.potential} and \eqref{eq:discrete.potential.k=0} with $\mu \in \Koly{r+1}{d-1}(f)$, that $\lproj{r}{f}{0}\Pot{r+1}{f}{0}\ul\omega_f=\Pot{r}{f}{0}\ul\omega_f$. We will moreover see in Remark \ref{rem:consistency.improved.Pot.k=0} that $\Pot{r+1}{f}{0}$ enjoys optimal consistency properties.
\end{remark}

\begin{remark}[Space of DDR potentials]
  The space of DDR reconstructed potentials, that is,
  \[
  \{(\Pot{r}{f}{k}\ul\omega_h)_{f\in\Delta_d(\Mh),\,d\in[k,n]}\,:\,\ul\omega_h\in\ul X_{r,h}^k\}
  \]
  cannot be considered as a space of differential forms with global regularity, as the reconstructed polynomials do not have any compatibility condition of the traces; they are inherently piecewise discontinuous polynomials.
\end{remark}

\begin{example}[Interpretation in terms of vector proxies]\label{ex:ddr.interpretation}
  We start by considering $k=0$. In this case, formula~\eqref{eq:discrete.potential:d=k} means that (constant) real values are attached to the vertices $f_0=V\in \Vh=\Delta_0(\Mh)$ of the mesh, so that an iterative procedure can be initialised to reconstruct discrete gradients and related traces/potentials over higher-dimensional cells.
  Indeed, formula~\eqref{eq:discrete.exterior.derivative:ddr} reconstructs a (scalar) gradient over edges $f_1 = E\in \Eh = \Delta_1(\Mh)$ (i.e., the derivative along the direction given by the orientation of~$E$) based on the values at the vertices and the value on the edge itself.
  This edge gradient, in turn, enters~\eqref{eq:discrete.potential} to define a scalar edge trace over $E$.
  When $d$ takes the values $2$ and $3$, the successive application of formulas \eqref{eq:discrete.exterior.derivative:ddr}--\eqref{eq:discrete.potential} defines, respectively, the pairs (face gradient, scalar face trace) on mesh faces $f_2 = F \in \Fh = \Delta_2(\Mh)$, and (element gradient, scalar element potential) on mesh elements $f_3 = T \in \Th = \Delta_3(\Mh)$.
  \smallskip

  Let us now turn to the case $k=1$, for which we provide more details.
  The vector proxy for the space $\ul X_{r,h}^1$ is the space
  \[
  \Xcurl{h}
  = \bigtimes_{E\in\Eh}\Poly{r}{}(E)\times\bigtimes_{F\in\Fh}\RaviartThomas{r}(F)\times\bigtimes_{T\in\Th}\RaviartThomas{r}(T)
  \]
  and, with standard DDR notation, we denote by $\Xcurl{Y}$ its restriction to a mesh element or face $Y\in\Th\cup\Fh$.
  By~\eqref{eq:discrete.potential:d=k} with $d = k = 1$, the reconstruction process is initialised by $1$-forms, whose vector proxies are scalar-valued polynomials of degree $r$ over edges $f_1 = E\in \Eh$ that play the role of edge tangential traces.

  Then, for each mesh face $f_2 = F\in \Fh$, we sequentially reconstruct a scalar face curl $\CF:\Xcurl{F}\to\Poly{r}{}(F)$ by~\eqref{eq:discrete.exterior.derivative:ddr} with $d=k+1=2$ and a vector face tangential trace $\trFt:\Xcurl{F}\to\vPoly{r}{}(F)$ by~\eqref{eq:discrete.potential}.
  Specifically, $\CF$ is such that, for all $\uvec{v}_F=\big((v_E)_{E\in\EF}, \bvec{v}_{F}\big)\in\Xcurl{F}$,
  \[
  \int_F\CF\uvec{v}_F~q
  = \int_F\bvec{v}_{F}\cdot\VROT_F q
  + \sum_{E\in\EF}\varepsilon_{FE}\int_E v_E~q
  \qquad
  \forall q\in\Poly{r}{}(F),
  \]
  where, for all $E\in\EF$ (the set of edges of $F$), $\varepsilon_{FE}\in\{-1,+1\}$ denotes the orientation of $E$ relative to $F$, while $\trFt$ satisfies, for all $\uvec{v}_F\in\Xcurl{F}$,
  \begin{multline*}
    \int_F\trFt\uvec{v}_F\cdot(\VROT_F q + \bvec{w})
    = \int_F\CF\uvec{v}_F~q
    - \sum_{E\in\EF}\varepsilon_{FE}\int_E v_E~q
    + \int_F\bvec{v}_{F}\cdot\bvec{w}, \\
    \forall(q,\bvec{w}) \in\Poly{r+1}{\flat}(F) \times \cRoly{r}(F).
  \end{multline*}
  The alternative interpretation of $1$-forms in dimension $d=2$ results in a rotation of $\Xcurl{F}$ by a right angle.
  Correspondingly, \eqref{eq:discrete.exterior.derivative:ddr} yields a face divergence (see \eqref{eq:d.proxy.n=2.a} and \eqref{eq:d.proxy.n=2.b} in Appendix \ref{sec:exterior.calculus.Rn}).

  Next, for each mesh element $f_3 = T \in \Th$,~\eqref{eq:discrete.exterior.derivative:ddr} defines the element curl $\CT:\Xcurl{T}\to\vPoly{r}{}(T)$ such that, for all $\uvec{v}_T=\big((v_E)_{E\in\ET}, (\bvec{v}_{F})_{F\in\FT}, \bvec{v}_{T}\big)\in\Xcurl{T}$,
  \[
  \int_T\CT\uvec{v}_T\cdot\bvec{w}
  = \int_T\bvec{v}_{T}\cdot\CURL\bvec{w}
  + \sum_{F\in\FT}\varepsilon_{TF}\int_F\trFt\uvec{v}_F\cdot(\bvec{w}\times\normal_F)
  \qquad\forall\bvec{w}\in\vPoly{r}{}(T),
  \]
  where, for all $F\in\FT$ (the set of faces of $T$), $\varepsilon_{TF}\in\{-1,+1\}$ denotes the orientation of $F$ relative to $T$, while \eqref{eq:discrete.potential} defines the vector potential $\Pcurl:\Xcurl{T}\to\vPoly{r}{}(T)$ such that, for all $\uvec{v}_T\in\Xcurl{T}$,
  \begin{multline*}
    \int_T\Pcurl\uvec{v}_T\cdot(\CURL\bvec{w} + \bvec{z})
    = \int_T\CT\uvec{v}_T\cdot\bvec{w}
    - \sum_{F\in\FT}\varepsilon_{TF}\int_F \trFt\uvec{v}_F\cdot(\bvec{w}\times\normal_F)
    +\int_T \bvec{v}_{T}\cdot \bvec{z}
    \\
    \forall(\bvec{w},\bvec{z}) \in\cGoly{r+1}(T) \times \cRoly{r}(T).
  \end{multline*}

  When $k=2$, \eqref{eq:discrete.exterior.derivative:ddr} reconstructs on mesh elements $f_3  = T \in \Th$ a discrete divergence of order $r$ based on the polynomial scalar trace defined by~\eqref{eq:discrete.potential:d=k}, which plays the role of a normal trace on the face $f_2 = F \in \FT$. Then, \eqref{eq:discrete.potential} defines a vector potential of degree $r$ over $T$.
  \smallskip

  Finally, in the case $k=3$, \eqref{eq:discrete.potential:d=k} simply yields a polynomial over mesh elements $f_3=T\in \Th$.
\end{example}

\subsubsection{Global discrete exterior derivative and DDR complex}

To arrange the spaces $\ul X_{r,h}^k$ into a sequence that mimics the continuous de Rham complex, for any form degree $k$ such that $0\le k\le n-1$, we introduce the \emph{global discrete exterior derivative} $\ud_{r,h}^{k} : \ul X^k_{r,h} \to \ul X_{r,h}^{k+1}$ defined as follows:
\begin{equation}\label{eq:global.discrete.exterior.derivative}
  \ud_{r,h}^k\ul\omega_h
  \coloneq \big(
  \trimproj{r}{f}{d-k-1}(\star\rd_{r,f}^k\ul\omega_f)
  \big)_{f\in\Delta_d(\Mh),\,d\in[k+1,n]}.
\end{equation}
In what follows, given a $d$-cell $f\in\Delta_d(\Mh)$ with $d\in[k+1,n]$, we denote by $\ud_{r,f}^k$ the \emph{local discrete exterior derivative} collecting the components of $\ud_{r,h}^k$ on $f$ and its boundary.
The DDR sequence reads
\begin{equation}\label{eq:ddr}
  \begin{tikzcd}
    \DDR{r}\coloneq
    \{0\} \arrow{r}{}
    &\ul X_{r,h}^0 \arrow{r}{\ud_{r,h}^0}
    &\ul X_{r,h}^1 \arrow{r}{}
    &\cdots \arrow{r}{}
    &\ul X_{r,h}^{n-1} \arrow{r}{\ud_{r,h}^{n-1}}
    &\ul X_{r,h}^{n} \arrow{r}{}
    & \{0\}.
  \end{tikzcd}
\end{equation}
The main results concerning this sequence are stated hereafter.

\begin{theorem}[Cohomology of the Discrete de Rham complex]\label{thm:cohomology:ddr}
  The DDR sequence \eqref{eq:ddr} is a complex and its cohomology is isomorphic to the cohomology of the continuous de Rham complex \eqref{eq:diff.de.rham}.
\end{theorem}

\begin{proof}
  See Section~\ref{sec:complex+cohomology:cohomology}.
\end{proof}

\begin{theorem}[Polynomial consistency of the discrete potential and exterior derivative]\label{thm:polynomial.consistency:ddr}
  For all integers $0\le k\le d\le n$ and all $f\in\Delta_d(\Mh)$, it holds
  \begin{equation}\label{eq:discrete.potential:polynomial.consistency}
    \Pot{r}{f}{k}\ul I_{r,f}^k\omega = \omega
    \qquad\forall\omega\in\Poly{r}{}\Lambda^k(f),
  \end{equation}
  and, if $d\ge k+1$,
  \begin{equation}\label{eq:discrete.exterior.derivative:polycons}
    \rd_{r,f}^k\ul I_{r,f}^k\omega = \rd\omega
    \qquad\forall\omega\in\Poly{r+1}{-}\Lambda^k(f).
  \end{equation}
\end{theorem}

\begin{proof}
  See Section~\ref{sec:complex:polynomial.consistency:ddr}.
\end{proof}

\begin{remark}[Consistency of traces]\label{rem:discrete.potential:trace.consistency}
  The above theorem actually implies that, for any $d$-face $f\in\Delta_d(\Mh)$, $d\in[k,n]$, any $\omega\in \Poly{r}{}\Lambda^k(f)$, and any integer $d'\in[k,d]$,
  \begin{equation*}
    P_{r,f'}^k\ul I_{r,f'}^k\omega = \tr_{f'}\omega
    \qquad\forall f'\in\Delta_{d'}(f).
  \end{equation*}
  This can be easily seen noticing that $\ul I_{r,f'}^k\omega = \ul I_{r,f'}^k\tr_{f'}\omega$ and $\tr_{f'}\omega\in\Poly{r}{}\Lambda^k(f')$, and invoking \eqref{eq:discrete.potential:polynomial.consistency} with $(f,\omega)\gets(f',\tr_{f'}\omega)$.
\end{remark}

To state the consistency properties of the potential reconstruction and of the discrete exterior derivative we introduce, for any real number $p\in[1,\infty]$ and any integer $s\ge 0$, the following scaled seminorm on $W^{\max(r+1,s),p}\Lambda^k(f)$:
\begin{equation}\label{eq:def.Wspmax.norm}
  \seminorm{W^{(r+1,s),p}\Lambda^k(f)}{\omega}\coloneq
  \left\{\begin{array}{cc}
	\displaystyle\seminorm{W^{r+1,p}\Lambda^k(f)}{\omega}&\mbox{ if $s\le r+1$},\\[.3em]
	\displaystyle\sum_{t=r+1}^s h_f^{t-r-1}\seminorm{W^{t,p}\Lambda^k(f)}{\omega}&\mbox{ if $s> r+1$}.
  \end{array}\right.
\end{equation}
\begin{corollary}[Consistency of the discrete potential and exterior derivative on smooth forms]
  \label{cor:polynomial.consistency.smooth:ddr}
  Let $0\le k\le d\le n$ be integers, let $f\in\Delta_d(\Mh)$, and take $\theta>0$ such that $f$ is connected by star-shaped sets with parameter $\theta$, see \cite[Definition 1.41]{Di-Pietro.Droniou:20} (in particular, $f$ satisfies this assumption if it is star-shaped with respect to a ball of diameter $\theta h_f$). Then, for all $p\in [1,\infty]$ and all integer $s$ such that $sp>d$, there exists $C>0$ depending only on $\theta$, $d$, $k$, $s$ and $r$ such that, for all integer $0\le m\le r+1$,
  \begin{equation}\label{eq:discrete.potential:polynomial.consistency.smooth}
    \seminorm{W^{m,p}\Lambda^k(f)}{\Pot{r}{f}{k}\ul I_{r,f}^k\omega - \omega}\le Ch_f^{r+1-m}\seminorm{W^{(r+1,s),p}\Lambda^k(f)}{\omega}
    \qquad\forall\omega\in W^{\max(r+1,s),p}\Lambda^k(f),
  \end{equation}
  and, if $d\ge k+1$,
  \begin{multline}\label{eq:discrete.exterior.derivative:polycons.smooth}
    \seminorm{W^{m,p}\Lambda^{k+1}(f)}{\rd_{r,f}^k\ul I_{r,f}^k\omega - \rd\omega}\le Ch_f^{r+1-m}\seminorm{W^{(r+1,s),p}\Lambda^{k+1}(f)}{\rd\omega}
    \\
    \forall\omega\in C^1\Lambda^k(\overline{f})\mbox{ s.t. }\rd\omega\in W^{\max(r+1,s),p}\Lambda^{k+1}(f).
  \end{multline}
\end{corollary}

\begin{proof}
  See Section~\ref{sec:complex:polynomial.consistency:ddr}.
\end{proof}


\subsubsection{Discrete $L^2$-products}

Using the potentials built in Section \ref{sec:ddr:potentials.differentials}, we can define, for all $k\in[0,n]$, an inner product $(\cdot,\cdot)_{k,h}:\ul X_{r,h}^k\times\ul X_{r,h}^k\to\Real$ that induces an $L^2$-structure on $\ul X_{r,h}^k$.
Specifically, we set:
For all $(\ul\omega_h,\ul\mu_h)\in\ul X_{r,h}^k\times\ul X_{r,h}^k$,
\begin{equation}\label{eq:l2prod.X.h}
  \begin{gathered}
    (\ul\omega_h,\ul\mu_h)_{k,h}
    \coloneq\sum_{f\in\Delta_n(\Mh)}(\ul\omega_f,\ul\mu_f)_{k,f}
    \\
    \text{
      with
      $(\ul\omega_f,\ul\mu_f)_{k,f}\coloneq
      \int_f \Pot{r}{f}{k}\ul\omega_f\wedge\star \Pot{r}{f}{k}\ul\mu_f
      + s_{k,f}(\ul\omega_f,\ul\mu_f)$
      for all $f\in\Delta_n(\Mh)$,
    }
  \end{gathered}
\end{equation}
where $s_{k,f}:\ul X_{r,f}^k\times\ul X_{r,f}^k\to\Real$ is the stabilisation bilinear form such that
\begin{multline*}
  s_{k,f}(\ul\omega_f,\ul\mu_f)
  \\
  =
  \sum_{d'=k}^{n-1} h_f^{n-d'}\sum_{f'\in\Delta_{d'}(f)}\int_{f'} (\tr_{f'}\Pot{r}{f}{k}\ul\omega_f - \Pot{r}{f'}{k}\ul\omega_{f'})
  \wedge\star(\tr_{f'}\Pot{r}{f}{k}\ul\mu_f - \Pot{r}{f'}{k}\ul\mu_{f'}),
\end{multline*}
with $h_f$ denoting the diameter of $f$.
The first term in the right-hand side of $(\cdot,\cdot)_{k,f}$ is responsible for consistency, while the second one ensures the positivity and definiteness of this bilinear form (required for $(\cdot,\cdot)_{k,h}$ to define an inner product).
More specifically, by Theorem~\ref{thm:polynomial.consistency:ddr} and Remark~\ref{rem:discrete.potential:trace.consistency} it holds, for all $f\in\Delta_n(\Mh)$,
\begin{equation}\label{eq:consistency:L2product}
  (\ul I_{r,f}^k\omega,\ul\mu_f)_{k,f} =
  \int_f\omega\wedge\star \Pot{r}{f}{k}\ul\mu_f\qquad
  \forall \omega\in\Poly{r}{}\Lambda^k(f)\,,\;\forall\ul\mu_f\in\ul X_{r,f}^k.
\end{equation}
Additionally, by \eqref{eq:proj.potential} below, the mapping $\ul X_{r,f}^k\ni\ul\omega_f\mapsto\|\ul\omega_f\|_{k,f}\coloneq(\ul\omega_f,\ul\omega_f)_{k,f}^{\nicefrac12}\in\Real$ defines a norm on $\ul X_{r,f}^k$.
Numerical schemes for linear PDEs related to the de Rham complex are typically obtained replacing continuous spaces and $L^2$-products with their discrete counterparts, according to the principles illustrated in the next section; see also \cite[Section~7]{Di-Pietro.Droniou:23*1}.

\begin{remark}[Stabilisation]
  A more general expression for the local $L^2$-product in \eqref{eq:l2prod.X.h} is obtained replacing $s_{k,f}$ with
  \[
  s_{\mathcal{B},k,f}(\ul\omega_f,\ul\mu_f)
  = \mathcal{B}_f(\ul I_{r,f}^k \Pot{r}{f}{k}\ul\omega_f - \ul\omega_f, \ul I_{r,f}^k \Pot{r}{f}{k}\ul\mu_f - \ul\mu_f),
  \]
  with $\mathcal{B}_f:\ul X_{r,f}^k\times\ul X_{r,f}^k\to\Real$ denoting a symmetric positive definite bilinear form inducing a norm that scales in $h_f$ as $\|{\cdot}\|_{k,f}$ defined above.
  Crucially, $s_{\mathcal{B},k,f}$ depends on its arguments only through the difference operator $\ul X_{r,f}^k\ni\ul\omega_f\mapsto\ul I_{r,f}^k \Pot{r}{f}{k}\ul\omega_f - \ul\omega_f\in\ul X_{r,f}^k$, which guarantees that it vanishes whenever one of its arguments is the interpolate of a differential form in $\Poly{r}{}\Lambda^k(f)$ as a result of \eqref{eq:discrete.potential:polynomial.consistency}.
\end{remark}

The stabilisation bilinear form not only vanishes on interpolate of polynomials, it also enjoys some consistency property on interpolate of smooth forms. Extending the notation for Sobolev spaces, the $H^{(r+1,s)}$-seminorm corresponds to the $W^{(r+1,s),2}$-seminorm.

\begin{lemma}[Consistency of the stabilisation bilinear form]\label{lem:consistency.stabilisation}
  Under the assumptions on $k$, $d$, and $f$ in Corollary \ref{cor:polynomial.consistency.smooth:ddr}, for all integer $s$ such that $2s>d$ there exists $C>0$ depending only on $\theta$, $d$, $k$, $s$ and $r$ such that
  \begin{equation}\label{eq:consistency.stabilisation}
    s_{k,f}(\ul{I}_{r,f}^k\omega,\ul{I}_{r,f}^k\omega)^{\nicefrac12}
    \le C h_f^{r+1}\seminorm{H^{(r+1,s)}\Lambda^{k}(f)}{\omega}
    \qquad\forall \omega\in H^{\max(r+1,s)}\Lambda^k(f).
  \end{equation}
\end{lemma}

\begin{proof}
  See Section~\ref{sec:complex:polynomial.consistency:ddr}.
\end{proof}


\subsection{Application to the Hodge Laplacian}

In this section we write a DDR scheme for the Hodge Laplacian and use this model problem to showcase the relevant properties for the analysis.

\subsubsection{DDR scheme}

In view of Remark \ref{rem:non-trivial.topology} below, to keep the exposition as simple as possible, we assume that $\Omega\subset\Real^n$ has trivial topology, so that the spaces of harmonic forms are trivial.
Given a form degree $k \ge 0$ and a form $g \in \Lambda^k(\Omega)$ smooth enough, we focus on the following mixed formulation:
Find $(\sigma,u) \in H\Lambda^{k-1}(\Omega) \times H\Lambda^k(\Omega)$ such that
\[
\begin{alignedat}{2}
  \langle\sigma,\tau\rangle_{k-1} - \langle u, \rd^{k-1} \tau\rangle_k &= 0
  &\qquad& \forall\tau\in H\Lambda^{k-1}(\Omega),
  \\
  \langle \rd^{k-1} \sigma, v\rangle_k + \langle \rd^k u, \rd^k v\rangle_{k+1} &= \langle g, v\rangle_k
  &\qquad& \forall v \in H\Lambda^k(\Omega),
\end{alignedat}
\]
where $\langle\omega,\mu\rangle_{\ell}\coloneq\int_\Omega\omega\wedge\star\mu$ denotes the $L^2$-product of $\ell$-forms.
Let a polynomial degree $r \ge 0$ be fixed.
Assuming $g$ smooth enough for $\underline{I}_{r,h}^k g$ to be well-defined, the DDR scheme is obtained with obvious substitutions, and reads:
Find $(\underline{\sigma}_h,\underline{u}_h) \in \underline{X}_{r,h}^{k-1} \times \underline{X}_{r,h}^k$ such that
\[
\begin{alignedat}{2}
  (\underline{\sigma}_h,\underline{\tau}_h)_{k-1,h} - (\underline{u}_h, \ud_{r,h}^{k-1} \underline{\tau}_h)_{k,h} &= 0
  &\qquad& \forall \underline{\tau}_h \in \underline{X}_{r,h}^{k-1},
  \\
  (\ud_{r,h}^{k-1} \underline{\sigma}_h, \underline{v}_h )_{k,h} + ( \ud_{r,h}^k \underline{u}_h, \ud_{r,h}^k \underline{v}_h)_{k+1,h} &= (\underline{I}_{r,h}^k g, \underline{v}_h )_{k,h}
  &\qquad& \forall \underline{v}_h \in \underline{X}_{r,h}^k,
\end{alignedat}
\]
or, equivalently,
\begin{equation}\label{eq:discrete}
  \mathcal{A}_h((\underline{\sigma}_h,\underline{u}_h),(\underline{\tau}_h,\underline{v}_h))
  = (\underline{I}_{r,h}^k g,\underline{v}_h)_{k,h}
  \qquad\forall (\underline{\tau}_h,\underline{v}_h)\in\underline{X}_{r,h}^{k-1}\times\underline{X}_{r,h}^k,
\end{equation}
where the bilinear form $\mathcal{A}_h : \big[ \underline{X}_{r,h}^{k-1}\times\underline{X}_{r,h}^k \big]^2 \to \Real$ is such that
\begin{multline}\label{eq:Ah}
  \mathcal{A}_h((\underline{\sigma}_h,\underline{u}_h),(\underline{\tau}_h,\underline{v}_h))
  \\
  \coloneq (\underline{\sigma}_h,\underline{\tau}_h)_{k-1,h}
  - (\underline{u}_h, \ud_{r,h}^{k-1} \underline{\tau}_h)_{k,h}
  + (\ud_{r,h}^{k-1} \underline{\sigma}_h, \underline{v}_h )_{k,h}
  + ( \ud_{r,h}^k \underline{u}_h, \ud_{r,h}^k \underline{v}_h)_{k+1,h}.
\end{multline}

\begin{remark}[Regularity requirement on $g$]
  The regularity requirements on $g$ can be lowered replacing $(\underline{I}_{r,h}^k g,\underline{v}_h)_{k,h}$ with  $\sum_{f \in \Delta_n(\Mh)} \int_f g \wedge \star\Pot{r}{f}{k} \underline{v}_f$.
  The changes to the following discussion are straightforward, and we leave them to the reader.
\end{remark}

\begin{remark}[Extension to domains with non-trivial topology]\label{rem:non-trivial.topology}
  The extension to domains with non-trivial topology requires to additionally enforce the $L^2$-orthogonality of $u$ to harmonic forms.
  The discrete space of harmonic forms is non-conforming (i.e.~it is not a subspace of the continuous harmonic forms) also when conforming finite element approximations of $H\Lambda^{k-1}(\Omega)$ and $H\Lambda^k(\Omega)$ are used.
  We therefore refer to \cite{Arnold:18} for further discussion on this subject.
\end{remark}

\subsubsection{Existence and uniqueness of a discrete solution and stability analysis}

We equip the product space $\underline{X}_{r,h}^{k-1}\times\underline{X}_{r,h}^k$ with the following norm:
\[
\tnorm{h}{(\underline{\tau}_h,\underline{v}_h)}
\coloneq\left(
\tnorm{k-1,h}{\underline{\tau}_h}^2
+ \tnorm{k,h}{\underline{v}_h}^2
\right)^{\nicefrac12},
\]
where, for all $\ell \ge 0$ and all $\underline{\omega}_h \in \underline{X}_{r,h}^\ell$,
\[
\text{%
  $\tnorm{\ell,h}{\underline{\omega}_h} \coloneq \big( \norm{\ell}{\underline{\omega}_h}^2 + \norm{\ell+1}{\ud_{r,h}^\ell\underline{\omega}_h}^2 \big)^{\nicefrac12}$
  with $\norm{\ell}{{\cdot}}\coloneq (\cdot,\cdot)_{\ell,h}^{\nicefrac12}$.
}
\]
From here on, we will assume that the mesh satisfies regularity properties generalising to the dimension $d$ those of \cite[Definition 1.9]{Di-Pietro.Droniou:20} and use $a \lesssim b$ as a shortcut for $a \le C b$ with $C > 0$ independent of the meshsize (dependencies will be specified more precisely when needed).
The following Poincar\'e-type result has been proved for $n=3$ in \cite{Di-Pietro.Hanot:23} using vector proxies:
For all  form degrees $\ell \ge 0$ and all $\underline{\omega}_h \in \underline{X}_{r,h}^\ell$, there exists $\underline{\mu}_h \in \underline{X}_{r,h}^\ell$ such that
\begin{equation}\label{eq:poincare}
  \text{%
    $\ud_{r,h}^\ell\underline{\mu}_h = \ud_{r,h}^\ell\underline{\omega}_h$
    and $\norm{\ell}{\underline{\mu}_h} \lesssim \norm{\ell+1}{\ud_{r,h}^\ell\underline{\omega}_h}$.
  }
\end{equation}
Based on this relation, we can prove the following inf-sup condition on $\mathcal{A}_h$ proceeding along the lines of \cite[Section 5]{Di-Pietro.Hanot:23}:
For all $(\underline{\upsilon}_h,\underline{w}_h) \in \underline{X}_{r,h}^{k-1} \times \underline{X}_{r,h}^k$,
\begin{equation}\label{eq:inf-sup}
  \tnorm{h}{(\underline{\upsilon}_h,\underline{w}_h)}
  \lesssim
  \sup_{(\underline{\tau}_h,\underline{v}_h) \in \underline{X}_{r,h}^{k-1} \times \underline{X}_{r,h}^k \setminus \{ 0 \}}
  \frac{\mathcal{A}_h((\underline{\upsilon}_h,\underline{w}_h),(\underline{\tau}_h,\underline{v}_h))}{\tnorm{h}{(\underline{\tau}_h,\underline{v}_h)}}.
\end{equation}

\subsubsection{Convergence analysis}

For $\ell \ge 0$, we define the global potential reconstruction $\Pot{r}{h}{\ell} : \underline{X}_{r,h}^\ell \to L^2\Lambda^\ell(\Omega)$ and discrete differential $\rd_{r,h}^\ell : \underline{X}_{r,h}^\ell \to L^2\Lambda^{\ell+1}(\Omega)$ obtained patching the corresponding local counterparts on the mesh $n$-faces:
For all $\underline{\omega}_h \in \underline{X}_{r,h}^\ell$,
\[
\text{
  $(\Pot{r}{h}{\ell} \underline{\omega}_h)_{|f}
  \coloneq \Pot{r}{f}{\ell} \underline{\omega}_f$
  and $(\rd_{r,h}^\ell \underline{\omega}_h)_{|f}
  \coloneq \rd_{r,f}^\ell \underline{\omega}_f$
  for all $f \in \Delta_n(\Mh)$.
}
\]
To establish the convergence of the solution to the discrete problem \eqref{eq:discrete}, we let, for the sake of brevity, $(\hud{\sigma}_h,\hud{u}_h) \coloneq (\underline{I}_{r,h}^{k-1} \sigma, \underline{I}_{r,h}^k u)$ (assuming $\sigma$ and $u$ smooth enough for the interpolation to be possible) and define the following errors:
\[
\begin{gathered}
  (\varepsilon_\pi, e_\pi)
  \coloneq(
  \sigma - \Pot{r}{h}{k-1}\hud{\sigma}_h,
  u - \Pot{r}{h}{k}\hud{u}_h
  ),
  \\
  (\varepsilon_{\rd,\pi}, e_{\rd,\pi})
  \coloneq
  (
  \rd^{k-1} \sigma - \rd_{r,h}^{k-1} \hud{\sigma}_h,
  \rd^k u - \rd_{r,h}^k \hud{u}_h
  ),
  \\
  (\underline{\varepsilon}_h, \underline{e}_h)
  \coloneq\left(
  \underline{\sigma}_h - \hud{\sigma}_h,
  \underline{u}_h - \hud{u}_h
  \right).
\end{gathered}
\]
The approximation properties stated in Corollary \ref{cor:polynomial.consistency.smooth:ddr} and Lemma \ref{lem:consistency.stabilisation} (with $m=0$, $p=2$, and $s=r+1$) and the commutation property \eqref{eq:commut.d} below yield the following estimate:
\begin{multline}\label{eq:e.pi}
  \norm{L^2\Lambda^{k-1}(\Omega)}{\varepsilon_\pi}
  + \norm{L^2\Lambda^k(\Omega)}{\varepsilon_{\rd,h}}
  + \norm{L^2\Lambda^k(\Omega)}{e_\pi}
  + \norm{L^2\Lambda^{k+1}(\Omega)}{e_{\rd,\pi}}
  \\
  + \seminorm{k-1,h}{\hud{\sigma}_h}
  + \seminorm{k,h}{\ud_{r,h}^{k-1}\hud{\sigma}_h}
  + \seminorm{k,h}{\hud{u}_h}
  + \seminorm{k+1,h}{\ud_{r,h}^{k-1}\hud{u}_h}
  \lesssim h^{r+1},
\end{multline}
where, for all $\ell \ge 0$, we introduced the stabilisation seminorm
\[
\seminorm{\ell,h}{{\cdot}} \coloneq s_{\ell,h}(\cdot,\cdot)^{\nicefrac12}.
\]
Notice that the terms involving this seminorm can be interpreted as a measure of the ``jumps'' between the traces of the potential reconstruction in the highest-dimensional cells and the potential reconstructions on their subcells.

Bounding the discrete components $(\underline{\varepsilon}_h, \underline{e}_h)$ of the error requires the following adjoint consistency results that, for $n=3$, follow from the analysis done in \cite[Section 6.2]{Di-Pietro.Droniou:23*1} using vector proxies.
Defining, for any form degree $\ell \ge 0$, any $\omega \in \Lambda^{\ell+1}(\Omega)$ smooth enough for the interpolators to make sense, and any $\underline{\mu}_h \in \underline{X}_{r,h}^\ell$,
\begin{equation}\label{eq:adjoint.error}
  \tilde{\mathcal{E}}_{r,h}^\ell(\omega;\underline{\mu}_h)
  \coloneq (\underline{I}_{r,h}^\ell\delta \omega, \underline{\mu}_h)_{\ell,h}
  - (\underline{I}_{r,h}^{\ell+1} \omega, \ud_{r,h}^\ell \underline{\mu}_h)_{\ell+1,h},
\end{equation}
(where $\delta$ is the co-differential) it holds, under additional piecewise regularity assumptions on $\omega$,
\begin{equation}\label{eq:adjoint.consistency}
  |\tilde{\mathcal{E}}_{r,h}^\ell(\omega;\underline{\mu}_h)|
  \lesssim h^{r+1}\tnorm{\ell,h}{\underline{\mu}_h}
  \qquad \forall \underline{\mu}_h \in \underline{X}_{r,h}^\ell.
\end{equation}
Then, letting, for all $(\underline{\tau}_h,\underline{\mu}_h) \in \underline{X}_{r,h}^{k-1} \times \underline{X}_{r,h}^k$,
\begin{equation}\label{eq:Eh}
  \mathcal{E}_h(\underline{\tau}_h,\underline{v}_h)
  \coloneq (\underline{I}_{r,h}^k g,\underline{v}_h)
  - \mathcal A_h((\hud{\sigma}_h,\hud{u}_h),(\underline{\tau}_h,\underline{v}_h)),
\end{equation}
we write
\begin{equation}\label{eq:est.err.h}
  \tnorm{h}{(\underline{\varepsilon}_h, \underline{e}_h)}
  \overset{\eqref{eq:inf-sup}}\lesssim
  \sup_{(\underline{\tau}_h,\underline{v}_h) \in \underline{X}_{r,h}^{k-1} \times \underline{X}_{r,h}^k \setminus \{ 0 \}}\frac{\mathcal{A}_h((\underline{\varepsilon}_h,\underline{e}_h),(\underline{\tau}_h,\underline{v}_h))}{\tnorm{h}{(\underline{\tau}_h,\underline{v}_h)}}
  \overset{\eqref{eq:discrete},\,\eqref{eq:Eh}}
  = \tnorm{h,*}{\mathcal{E}_h(\cdot,\cdot)}.
\end{equation}
with $\tnorm{h,*}{{\cdot}}$ denoting the norm dual to $\tnorm{h}{{\cdot}}$.
Recalling that $g = \rd \sigma + \delta\rd u$ almost everywhere and expanding the bilinear form $\mathcal A_h$ according to its definition \eqref{eq:Ah}, we next observe that, for all $(\underline{\tau}_h,\underline{\mu}_h) \in \underline{X}_{r,h}^{k-1} \times \underline{X}_{r,h}^k$, and provided $\rd \sigma$ and $\delta\rd u$ are smooth enough for their interpolation to make sense,
\begin{equation}\label{eq:est.Eh}
  \begin{aligned}
    \mathcal{E}_h(\underline{\tau}_h,\underline{\mu}_h)
    &= \cancel{(\underline{I}_{r,h}^k\rd\sigma, \underline{v}_h)_{k,h}}
    + (\underline{I}_{r,h}^k\delta \rd u, \underline{v}_h)_{k,h}
    \\
    &\quad
    - (\hud{\sigma}_h,\underline{\tau}_h)_{k-1,h}
    + (\hud{u}_h, \ud_{r,h}^{k-1} \underline{\tau}_h)_{k,h}
    \\
    &\quad
    - \cancel{(\ud_{r,h}^{k-1} \hud{\sigma}_h, \underline{v}_h )_{k,h}}
    - ( \ud_{r,h}^k \hud{u}_h, \ud_{r,h}^k \underline{v}_h)_{k+1,h}
    \\
    &=
    (\underline{I}_{r,h}^k\delta \rd u, \underline{v}_h)_{k,h}
    - ( \underline{I}_{r,h}^{k+1} \rd u, \ud_{r,h}^k \underline{v}_h)_{k+1,h}
    \\
    &\quad
    - (\underline{I}_{r,h}^{k-1} \delta u,\underline{\tau}_h)_{k-1,h}
    + (\underline{I}_{r,h}^k u, \ud_{r,h}^{k-1} \underline{\tau}_h)_{k,h}
    \\
    \overset{\eqref{eq:adjoint.error}}&=
    \tilde{\mathcal{E}}_{r,h}^k(\rd u;\underline{v}_h)
    - \tilde{\mathcal{E}}_{r,h}^{k-1}(u;\underline{\tau}_h),
  \end{aligned}
\end{equation}
where the cancellation in the first step is a consequence of the commutation property \eqref{eq:commut.d} below,
while, in the second step, we have used again this commutation property to write $\ud_{r,h}^k \hud{u}_h = \underline{I}_{r,h}^{k+1} \rd u$ and $\hud{\sigma}_h = \underline{I}_{r,h}^{k-1} \sigma = \underline{I}_{r,h}^{k-1} \delta u$.
Combining \eqref{eq:est.err.h} and \eqref{eq:est.Eh}, and using the adjoint consistency \eqref{eq:adjoint.consistency}, we finally get
\begin{equation}\label{eq:e.h}
  \tnorm{h}{(\underline{\varepsilon}_h, \underline{e}_h)} \lesssim h^{r+1}.
\end{equation}
Finally, using triangle inequalities, invoking \eqref{eq:e.pi} and \eqref{eq:e.h}, and using the $\tnorm{h}{{\cdot}}$-boundedness of the $L^2$-norm of the potential (resulting from the defintion of this norm), we arrive at the following error estimate:
\begin{multline*}
  \norm{L^2\Lambda^{k-1}(\Omega)}{\sigma - \Pot{r}{h}{k-1}\underline{\sigma}_h}
  + \norm{L^2\Lambda^k(\Omega)}{\rd \sigma - \rd_{r,h}^{k-1}\underline{\sigma}_h}
  \\
  + \norm{L^2\Lambda^k(\Omega)}{u - \Pot{r}{h}{k}\underline{u}_h}
  + \norm{L^2\Lambda^{k+1}(\Omega)}{\rd u - \rd_{r,h}^k\underline{u}_h}
  \\
  + \seminorm{k-1,h}{\underline{\sigma}_h}
  + \seminorm{k,h}{\ud_{r,h}^{k-1}\underline{\sigma}_h}
  + \seminorm{k,h}{\underline{u}_h}
  + \seminorm{k+1,h}{\ud_{r,h}^{k-1}\underline{u}_h}
  \lesssim h^{r+1}.
\end{multline*}
Notice that, as in the finite element framework, improved error estimates for certain components of the above error can be obtained using the Aubin--Nitsche trick, which has been extended to the fully discrete in \cite[Section~2.3]{Di-Pietro.Droniou:18}.
This topic will be explored in a future work.

\subsubsection{Summary of the relevant results}

In summary, to carry out the error analysis above we have used the following relevant results:
\begin{itemize}
\item The \emph{isomorphism in cohomology} between the discrete and continuous de Rham complexes stated in Theorem \ref{thm:cohomology:ddr} below to infer the existence and uniqueness of a solution to the discrete problem \eqref{eq:discrete};
\item The \emph{uniform Poincar\'e-type inequalities} \eqref{eq:poincare} to prove an inf-sup condition on the bilinear form $\mathcal{A}_h$. Such inequalities have been proved for $n=3$ in \cite{Di-Pietro.Hanot:23} using vector proxies, with arguments that lend themselves to an adaptation to the framework of differential forms;
\item The approximation properties stated in Corollary \ref{cor:polynomial.consistency.smooth:ddr} and Lemma \ref{lem:consistency.stabilisation}, consequences of the \emph{polynomial consistency} properties of Theorem \ref{thm:polynomial.consistency:ddr};
\item The \emph{adjoint consistency} estimates \eqref{eq:adjoint.consistency} to bound the discrete components of the error.
  Adjoint consistency results have been proved for $n=3$ in \cite[Section 6.2]{Di-Pietro.Droniou:23*1} using vector proxies.
  The proofs therein, however, use different arguments for each form degree.
  Devising a unified proof valid for all form degrees is still an open problem, which we leave for a future work.
\end{itemize}


\subsection{Complex property}\label{sec:ddr.complex}

We denote by $\Lambda^{k+1}(\partial f)$ the space $\bigtimes_{f'\in \Delta_{d-1}(f)} \Lambda^{k+1}(f')$, which can be intuitively understood as a space of piecewise $(k+1)$-forms on $\partial f$.
For all integers $d\in[k+2,n]$, the piecewise polynomial boundary exterior derivative $\rd^k_{r,\partial f} :\ul X_{r,\partial f}^k\to\Lambda^{k+1}(\partial f)$ is defined such that $(\rd^k_{r,\partial f})_{|f'} \coloneq \rd^k_{r,f'}$ for all $f'\in\Delta_{d-1}(f)$ ($\rd^k_{r,f'}$ being the discrete exterior derivative on the $(d-1)$-cell $f'$ defined by \eqref{eq:discrete.exterior.derivative:ddr}).
The following lemma generalises the links between element gradients (resp., curls) and face gradients (resp., curls) proved in \cite[Propositions 1 and 4]{Di-Pietro.Droniou:23*1} using vector proxies.

\begin{lemma}[Link between discrete exterior derivatives on subcells]\label{lem:link.diff.subcells}
  It holds, for all $d\ge k+2$, all $f\in\Delta_d(\Mh)$, and all $\ul\omega_f\in\ul X_{r,f}^k$,
  \begin{equation}\label{eq:link.discrete.exterior.derivatives}
    \int_f \rd^k_{r,f}\ul\omega_f\wedge \rd\alpha=(-1)^{k+1}\int_{\partial f}\rd^k_{r,\partial f}\ul\omega_{\partial f}\wedge \tr_{\partial f}\alpha\qquad\forall \alpha\in \Poly{r+1}{-}\Lambda^{d-k-2}(f).
  \end{equation}
\end{lemma}

\begin{proof}
  Take $\mu=\rd\alpha\in\Poly{r}{}\Lambda^{d-k-1}(f)$ in \eqref{eq:discrete.exterior.derivative:ddr} and use $\rd\circ\rd=0$ and $\tr_{\partial f}\rd=\rd\tr_{\partial f}$ (since the trace is a pullback, it commutes with the exterior derivative) to get
  \begin{equation}\label{eq:link.1}
    \int_f \rd^k_{r,f}\ul\omega_f\wedge \rd\alpha
    = \int_{\partial f}P^k_{r,\partial f}\ul\omega_{\partial f}\wedge \rd\tr_{\partial f}\alpha.
  \end{equation}
  For each $f'\in\Delta_{d-1}(f)$ forming $\partial f$, by Lemma \ref{lem:traces.trimmed} we have $\tr_{f'}\alpha\in\Poly{r+1}{-}\Lambda^{d-k-2}(f')$ so, by \eqref{eq:discrete.potential} applied to $f'$ instead of $f$ with test function $(\mu,\nu) = (\tr_{f'}\alpha,0)$ (see Remark \ref{rem:discrete.potential:validity}), we have, additionally using the fact that $\tr_{\partial f'} (\tr_{f'} \alpha) = \tr_{\partial f'} \alpha$,
  \[
  (-1)^{k+1}\int_{f'}P^k_{r,f'}\ul\omega_{f'}\wedge \rd\tr_{f'}\alpha = \int_{f'}\rd^k_{r,f'}\ul\omega_{f'}\wedge \tr_{f'}\alpha
  -\int_{\partial f'}P^k_{r,\partial f'}\ul\omega_{\partial f'}\wedge \tr_{\partial f'}\alpha.
  \]
  Multiplying the first term in the right-hand side by the relative orientation $\varepsilon_{f\!f'}$ and summing over $f'\in\Delta_{d-1}(f)$, we have
  \[
  \sum_{f'\in\Delta_{d-1}(f)}\varepsilon_{f\!f'}\int_{f'}\rd^k_{r,f'}\ul\omega_{f'}\wedge \tr_{f'}\alpha \overset{\eqref{eq:integral.relative.orientation}}
  = \int_{\partial f}\rd^k_{r,\partial f}\ul\omega_{\partial f}\wedge \tr_{\partial f}\alpha.
  \]
  Proceeding similarly for the second term in the right-hand side, we have
  \[
  \begin{aligned}
    &\sum_{f'\in\Delta_{d-1}(f)} \varepsilon_{f\!f'}\int_{\partial f'}P^k_{r,\partial f'}\ul\omega_{\partial f'}\wedge\tr_{\partial f'}\alpha
    \\
    &\quad\begin{aligned}[t]
    \overset{\eqref{eq:integral.relative.orientation}}&=
    \sum_{f'\in\Delta_{d-1}(f)}\varepsilon_{f\!f'}\Bigg(\sum_{f^{\prime\prime}\in\Delta_{d-2}(f')}\varepsilon_{f'\!f^{\prime\prime}}\int_{f''}P^k_{r,f''}\ul\omega_{f''}\wedge\tr_{f''} \alpha \Bigg)
    \\
    &=
    \sum_{f^{\prime\prime}\in\Delta_{d-2}(f)}\Bigg(\underbrace{\sum_{f'\in\Delta_{d-1}(f)\text{ s.t. }f^{\prime\prime}\in\Delta_{d-2}(f')}\varepsilon_{f\!f'}\varepsilon_{f'\!f^{\prime\prime}}}_{=0}\Bigg)\int_{f^{\prime\prime}}P^k_{r,f''}\ul\omega_{f''}\wedge\tr_{f''} \alpha
    = 0.
    \end{aligned}
  \end{aligned}
  \]
  Thus, $\int_{\partial f}P^k_{r,\partial f}\ul\omega_{\partial f}\wedge \rd\tr_{\partial f}\alpha = (-1)^{k+1} \int_{\partial f}\rd^k_{r,\partial f}\ul\omega_{\partial f}\wedge \tr_{\partial f}\alpha$ which, plugged into \eqref{eq:link.1}, gives the desired result.
\end{proof}

\begin{theorem}[Link between discrete potentials and exterior derivatives, complex property]\label{thm:link.pot.diff}
  It holds, for all integers $k\in[1,n]$ and $d\ge k$, all $f\in\Delta_d(\Mh)$, and all $\ul\omega_f\in\ul X_{r,f}^{k-1}$,
  \begin{equation}\label{eq:link.pot.diff}
    P^k_{r,f}(\ud_{r,f}^{k-1}\ul\omega_f)=\rd_{r,f}^{k-1}\ul\omega_f,
  \end{equation}
  and, if $d\ge k+1$,
  \begin{equation}\label{eq:complex.prop}
    \rd^k_{r,f}(\ud_{r,f}^{k-1}\ul\omega_f)=0.
  \end{equation}
  As a consequence, the sequence \eqref{eq:ddr} defines a complex.
\end{theorem}

\begin{proof}
  The proof is done by induction on $\rho\coloneq d-k$.
  \smallskip

  If $\rho= 0$ (i.e., $d=k$), by the definitions \eqref{eq:discrete.potential:d=k} of the discrete potential and \eqref{eq:global.discrete.exterior.derivative} of the global discrete exterior derivative with $k-1$ instead of $k$, we have $P^k_{r,f}(\ud_{r,f}^{k-1}\ul\omega_f) = \star^{-1}(\star\rd_{r,f}^{k-1}\ul\omega_f)=\rd_{r,f}^{k-1}\ul\omega_f$ (notice that, in the first passage, we can omit the projector in front of $\star\rd_{r,f}^{k-1}\ul\omega_f$ since this quantity sits in $\Poly{r}{}\Lambda^0(f)=\Poly{r}{-}\Lambda^0(f)$, and is therefore left unchanged by $\trimproj{r}{f}{0}$).
  This proves \eqref{eq:link.pot.diff}, and the relation \eqref{eq:complex.prop} is irrelevant here since $d=k$.
  \smallskip

  Let us now assume that \eqref{eq:link.pot.diff} and \eqref{eq:complex.prop} hold for a given $\rho\ge 0$, and let us consider $d$ and $k$ such that $d - k = \rho + 1$.
  We start by considering \eqref{eq:complex.prop} (which we need to prove since $d\ge k+1$ in the present case).
  Let us take $f\in\Delta_d(\Mh)$. Applying \eqref{eq:discrete.exterior.derivative:ddr} with $\ud^{k-1}_{r,f}\ul\omega_f$ instead of $\ul\omega_f$ and a generic $\mu\in \Poly{r}{}\Lambda^{d-k-1}(f)$, we have, expanding the local discrete exterior derivative $\ud^{k-1}_{r,f}\ul\omega_f$ according to its definition (i.e., the restriction to $f$ of \eqref{eq:global.discrete.exterior.derivative} with $k-1$ instead of $k$),
  \begin{multline}\label{eq:complex.1}
    \int_f \rd^k_{r,f}(\ud^{k-1}_{r,f}\ul\omega_f)\wedge\mu={}
    (-1)^{k+1}\int_f \star^{-1}(\trimproj{r}{f}{d-k}(\star\rd^{k-1}_{r,f}\ul\omega_f))\wedge \rd\mu
    \\
    +\int_{\partial f}P^k_{r,\partial f}(\ud^{k-1}_{r,\partial f}\ul\omega_{\partial f})\wedge \tr_{\partial f}\mu.
  \end{multline}
  By the induction hypothesis, \eqref{eq:link.pot.diff} holds on each $f'\in\Delta_{d-1}(f)$ (since $(d - 1) - k = \rho$), and thus
  \begin{equation}\label{eq:induction}
    P^k_{r,\partial f}(\ud^{k-1}_{r,\partial f}\ul\omega_{\partial f})=\rd_{r,\partial f}^{k-1}\ul\omega_{\partial f}.
  \end{equation}
  Invoking then \eqref{eq:remove.projector} with $(\mathcal X,\omega,\mu)\gets(\Poly{r}{-}\Lambda^{d-k}(f),\rd^{k-1}_{r,f}\ul\omega_f,\rd\mu)$, noticing that $\rd\mu\in\Poly{r-1}{}\Lambda^{d-k}(f)\subset\Poly{r}{-}\Lambda^{d-k}(f)$ (by \eqref{eq:trimmed.between} with $\ell=d-k$) to handle the first term in the right-hand side of \eqref{eq:complex.1}, we infer
  \begin{equation}\label{eq:d.d.initial}
    \int_f \rd^k_{r,f}(\ud^{k-1}_{r,f}\ul\omega_f)\wedge\mu=(-1)^{k+1}\int_f \rd^{k-1}_{r,f}\ul\omega_f\wedge \rd\mu
    +\int_{\partial f} \rd^{k-1}_{r,\partial f}\ul\omega_{\partial f}\wedge \tr_{\partial f}\mu=0,
  \end{equation}
  where the conclusion follows from the link \eqref{eq:link.discrete.exterior.derivatives} between discrete exterior derivatives on subcells applied with $k-1$ instead of $k$ and $\alpha=\mu\in \Poly{r}{}\Lambda^{d-k-1}(f) \subset \Poly{r+1}{-}\Lambda^{d-(k-1)-2}(f)$.
  Since $\mu$ is arbitrary in $\Poly{r}{}\Lambda^{d-k-1}(f)$, \eqref{eq:d.d.initial} proves \eqref{eq:complex.prop}.

  We next prove \eqref{eq:link.pot.diff}.
  For any $(\mu,\nu)\in\Koly{r+1}{d-k-1}(f)\times\Koly{r}{d-k}(f)$, the definition \eqref{eq:discrete.potential} of the potential applied to $\ud_{r,f}^{k-1}\ul\omega_f$ gives
  \begin{multline*}
    (-1)^{k+1}\int_f P^k_{r,f}(\ud_{r,f}^{k-1}\ul\omega_f)\wedge(\rd\mu +\nu)
    =
    \int_f \rd_{r,f}^{k} (\ud_{r,f}^{k-1}\ul\omega_f) \wedge \mu
    \\
    -  \int_{\partial f}  \Pot{r}{\partial f}{k}(\ud_{r,\partial f}^{k-1}\ul\omega_{\partial f}) \wedge \tr_{\partial f}{\mu}
    + (-1)^{k+1}\int_f \star^{-1}\trimproj{r}{f}{d-k}(\star\rd_{r,f}^{k-1}\ul\omega_f) \wedge\nu,
  \end{multline*}
  where we have additionally used, in the last term, the definition of the local discrete exterior derivative $\ud_{r,f}^{k-1}\ul\omega_f$, corresponding to the restriction to $f$ of \eqref{eq:global.discrete.exterior.derivative} with $k-1$ instead of $k$.
  Using the complex property \eqref{eq:complex.prop} that we have just proved, we have $\rd_{r,f}^{k} (\ud_{r,f}^{k-1}\ul\omega_f)=0$.
  Moreover, the induction hypothesis \eqref{eq:induction} yields $\Pot{r}{\partial f}{k}(\ud_{r,\partial f}^{k-1}\ul\omega_{\partial f})=\rd_{r,\partial f}^{k-1}\ul\omega_{\partial f}$.
  Hence, invoking \eqref{eq:link.discrete.exterior.derivatives} with $k-1$ instead of $k$ and $\alpha = \mu$ (notice that $\mu\in\Koly{r+1}{d-k-1}(f)\subset \Poly{r+1}{-}\Lambda^{d-(k-1)-2}(f)$ by \eqref{def:trimmed.spaces}) and applying \eqref{eq:remove.projector} with $(\mathcal X,\omega,\mu)\gets(\Poly{r}{-}\Lambda^{d-k}(f),\rd^{k-1}_{r,f}\ul\omega_f,\nu)$, which is valid since $\nu\in\Koly{r}{d-k}(f)\subset\Poly{r}{-}\Lambda^{d-k}(f)$ by \eqref{eq:trimmed.spaces:ell>=1} with $\ell=d-k\ge 1$, we obtain
  \begin{multline*}
    (-1)^{k+1}\int_f P^k_{r,f}(\ud_{r,f}^{k-1}\ul\omega_f)\wedge(\rd\mu +\nu)
    \\
    =
    -  (-1)^k\int_f \rd_{r, f}^{k-1}\ul\omega_{f} \wedge \rd\mu + (-1)^{k+1}\int_f \rd_{r,f}^{k-1}\ul\omega_f \wedge\nu.
  \end{multline*}
  Simplifying by $(-1)^{k+1}$ and recalling the isomorphism \eqref{eq:isomorphism.Pr.koszul} concludes the proof of \eqref{eq:link.pot.diff}.
\end{proof}

\subsection{Commutation}

The following lemma shows that the reconstructed potential $\Pot{r}{f}{k}\ul\omega_f$ on a $d$-cell $f$ is built by adding a high-order enhancement to $\star^{-1}\omega_f$; this enhancement is designed to obtain a polynomial consistency unachievable by the component alone (see \eqref{eq:discrete.potential:polynomial.consistency}).

\begin{lemma}[Links between component and potential reconstruction]
  For all integers $d\in [0,n]$ and $k\le d$, if $f\in\Delta_d(\Mh)$ and $\ul\omega_f\in\ul X_{r,f}^k$, then it holds
  \begin{multline} \label{def:Pot.correction}
    (-1)^{k+1}\int_f \Pot{r}{f}{k} \ul\omega_f \wedge (\rd\mu+\nu)
    = (-1)^{k+1}\int_f \star^{-1}\omega_f \wedge(\rd (\lproj{r}{f}{d-k-1}\mu)+\nu)
    \\
    +\int_f \rd_{r,f}^{k} \ul\omega_f \wedge (\mu-\lproj{r}{f}{d-k-1}\mu)
    -  \int_{\partial f}  \Pot{r}{\partial f}{k}\ul\omega_{\partial f} \wedge \tr_{\partial f}{(\mu-\lproj{r}{f}{d-k-1}\mu)}
    \\
    \forall (\mu,\nu) \in \Koly{r+1}{d-k-1}(f) \times \Koly{r}{d-k}(f).
  \end{multline}
  As a consequence,
  \begin{equation}\label{eq:proj.potential}
    \trimproj{r}{f}{d-k}(\star P^k_{r,f}\ul\omega_f) = \omega_f.
  \end{equation}
\end{lemma}
\begin{proof}
  If $d=k$, the relation \eqref{def:Pot.correction} follows from $\Koly{r+1}{d-k-1}(f)=\Koly{r+1}{-1}(f)= \{0\}$ and $\Pot{r}{f}{k} \ul\omega_f=\star^{-1}\omega_f$ (see \eqref{eq:discrete.potential:d=k}), which also establishes \eqref{eq:proj.potential} since $\trimproj{r}{f}{0}={\rm Id}$ on $\Poly{r}{}\Lambda^0(f)=\Poly{r}{-}\Lambda^0(f)$.

  Consider now $d\ge k+1$ and take $(\mu,\nu) \in \Koly{r+1}{d-k-1}(f) \times \Koly{r}{d-k}(f)$.
  Inserting $\pm \lproj{r}{f}{d-k-1}\mu$ into the definition \eqref{eq:discrete.potential} of $\Pot{r}{f}{k}\ul\omega_f$ we have
  \begin{equation} \label{def:Pot.correction.1}
    \begin{aligned}
      (-1)^{k+1} {}&\int_f \Pot{r}{f}{k} \ul\omega_f \wedge (\rd\mu+\nu)\\
      ={}& \int_f \rd_{r,f}^{k} \ul\omega_f \wedge \lproj{r}{f}{d-k-1}\mu+\int_f \rd_{r,f}^{k} \ul\omega_f \wedge (\mu-\lproj{r}{f}{d-k-1}\mu)
      -  \int_{\partial f}  \Pot{r}{\partial f}{k}\ul\omega_{\partial f} \wedge \tr_{\partial f}{\mu}\\
      &+ (-1)^{k+1}\int_f \star^{-1}\omega_f \wedge\nu.
    \end{aligned}
  \end{equation}
  On the other hand, the definition \eqref{eq:discrete.exterior.derivative:ddr} of $\rd_{r,f}^k\ul\omega_f$ applied to $\lproj{r}{f}{d-k-1}\mu$ yields
  \begin{multline*}
    \int_f \rd_{r,f}^{k} \ul\omega_f \wedge \lproj{r}{f}{d-k-1}\mu
    \\
    =(-1)^{k+1} \int_f \star^{-1}\omega_f\wedge \rd (\lproj{r}{f}{d-k-1}\mu)
    + \int_{\partial f} \Pot{r}{\partial f}{k}\ul\omega_{\partial f} \wedge \tr_{\partial f}(\lproj{r}{f}{d-k-1}\mu).
  \end{multline*}
  Substituting this relation into \eqref{def:Pot.correction.1} yields \eqref{def:Pot.correction}.

  To prove \eqref{eq:proj.potential} we apply \eqref{def:Pot.correction} with $(\mu,\nu)\in\Koly{r}{d-k-1}(f)\times \Koly{r}{d-k}(f)$ and notice that $\mu=\lproj{r}{f}{d-k-1}\mu$ since $\Koly{r}{d-k-1}(f)\subset \Poly{r}{}\Lambda^{d-k-1}(f)$, to get
  \begin{equation}\label{eq:Proj.trimmed.1}
    \int_f P^k_{r,f}\ul\omega_f\wedge(\rd\mu+\nu)
    = \int_f \star^{-1}\omega_f\wedge (\rd\mu + \nu).
  \end{equation}
  The isomorphism \eqref{eq:isomorphism.Prtrimmed.koszul} with $\ell=d-k\ge 1$ shows that $\rd\mu + \nu$ spans $\Poly{r}{-}\Lambda^{d-k}(f)$ when $(\mu,\nu)$ span $\Koly{r}{d-k-1}(f)\times \Koly{r}{d-k}(f)$. Hence, \eqref{eq:Proj.trimmed.1} gives
  \[
  \int_f\star^{-1}\omega_f\wedge\alpha
  = \int_fP^k_{r,f}\ul\omega_f\wedge\alpha
  \stackrel{\eqref{eq:remove.projector}}{=} \int_f\star^{-1}\trimproj{r}{f}{d-k}(\star P^k_{r,f}\ul\omega_f)\wedge\alpha
  \qquad\forall\alpha\in \Poly{r}{-}\Lambda^{d-k}(f),
  \]
  proving \eqref{eq:proj.potential} since $\star^{-1}\omega_f\in \Poly{r}{-}\Lambda^{d-k}(f)$ and $\star^{-1}$ is an isomorphism.
\end{proof}

\begin{theorem}[Commutation property for the local discrete exterior derivative]
  For all integers $d\in [1,n]$ and $k\le d-1$, and for all $f\in\Delta_d(\Mh)$, recalling the definition \eqref{eq:interpolator} of the interpolators, it holds
  \begin{equation}\label{eq:commut.d}
    \ud_{r,f}^k(\ul I_{r,f}^k\omega) = \ul I_{r,f}^{k+1}(\rd\omega)\qquad\forall\omega\in C^1\Lambda^k(\overline{f}),
  \end{equation}
  expressing the commutativity of the following diagram:%
  \[
  \begin{tikzcd}
    C^1\Lambda^k(\overline{f})\arrow{r}{\rd}\arrow{d}{\ul I_{r,f}^k}
    & C^0\Lambda^{k+1}(\overline{f})\arrow{d}{\ul I_{r,f}^{k+1}}
    \\
    \ul X_{r,f}^k\arrow{r}{\ud_{r,f}^k}
    & \ul X_{r,f}^{k+1}.
  \end{tikzcd}
  \]
\end{theorem}

\begin{proof}
  Given the definitions \eqref{eq:interpolator} of the interpolator and \eqref{eq:global.discrete.exterior.derivative} of the discrete exterior derivative, we have to prove that, for all $f'\in\Delta_{d'}(f)$ with $d'\in [k+1,d]$, $\trimproj{r}{f'}{d'-k-1}(\star \rd_{r,f'}^k\ul I_{r,f'}^k\omega)=\trimproj{r}{f'}{d'-k-1}(\star \tr_{f'}(\rd\omega))$.
  Recalling the definition of the projector $\trimproj{r}{f'}{d'-k-1}$ (i.e., \eqref{eq:def.piX} with $\mathcal X=\Poly{r}{-}\Lambda^{d'-k-1}(f')$), we need to prove that, for any $\mu\in\Poly{r}{-}\Lambda^{d'-k-1}(f')$
  \begin{equation*}
    \int_{f'}\star \rd_{r,f'}^k\ul I_{r,f'}^k\omega\wedge\star\mu=\int_{f'}\star \tr_{f'}(\rd\omega)\wedge\star\mu.
  \end{equation*}
  Applying \eqref{eq:commut.star.wedge}, this amounts to proving that
  \begin{equation}\label{eq:toprove}
    \int_{f'}\rd_{r,f'}^k\ul I_{r,f'}^k\omega\wedge \mu=\int_{f'} \tr_{f'}(\rd\omega)\wedge\mu.
  \end{equation}
  Using the definitions \eqref{eq:discrete.exterior.derivative:ddr} of the discrete exterior derivative on $f'$ and \eqref{eq:interpolator} of $\ul I_{r,f'}^k$, we have
  \begin{multline}\label{eq:toprove.2}
    \int_{f'} \rd_{r,f'}^k\ul I_{r,f'}^k\omega\wedge\mu
    \\
    =
    (-1)^{k+1}\int_{f'} \cancelto{\tr_{f'}\omega}{\star^{-1}\trimproj{r}{f'}{d'-k}(\star\tr_{f'}\omega)}\wedge \rd\mu+\int_{\partial f'}P^k_{r,\partial f'} \ul I_{r,\partial f'}^k\omega\wedge\tr_{\partial f'}\mu,
  \end{multline}
  where the substitution is justified by \eqref{eq:remove.projector} with $(\mathcal X,\omega,\mu)\gets(\Poly{r}{-}\Lambda^{d'-k}(f'),\tr_{f'}\omega,\rd\mu)$, since $\rd\mu\in\rd\Poly{r}{-}\Lambda^{d'-k-1}(f')\subset \rd\Poly{r}{}\Lambda^{d'-k-1}(f')\subset\Poly{r}{-}\Lambda^{d'-k}(f')$ (see  \eqref{eq:trimmed.spaces:ell>=1}). For all $f''\in\Delta_{d'-1}(f')$ we have $\tr_{f''}\mu\in\Poly{r}{-}\Lambda^{d'-1-k}(f'')$ (see Lemma \ref{lem:traces.trimmed}), so
  \[
  \begin{alignedat}{2}
    \int_{f''}P^k_{r,f''} \ul I_{r,f''}^k\omega\wedge\tr_{f''}\mu
    \overset{\eqref{eq:remove.projector}}&= \int_{f''}\star^{-1}\trimproj{r}{f''}{d'-1-k}(\star P^k_{r,f''} \ul I_{r,f''}^k\omega)\wedge\tr_{f''}\mu
    \\
    \overset{\eqref{eq:proj.potential},\, \eqref{eq:interpolator}}&=
    \int_{f''}\star^{-1}\trimproj{r}{f''}{d'-1-k}(\star \tr_{f''}\omega)\wedge\tr_{f''}\mu
    \\
    \overset{\eqref{eq:remove.projector}}&=
    \int_{f''} \tr_{f''}\omega\wedge\tr_{f''}\mu.
  \end{alignedat}
  \]
  Summing this relation over $f''\in\Delta_{d'-1}(f')$ and substituting the result into \eqref{eq:toprove.2} we obtain
  \[
  \int_{f'} \rd_{r,f'}^k\ul I_{r,f'}^k\omega\wedge\mu
  =
  (-1)^{k+1}\int_{f'} \tr_{f'}\omega\wedge \rd\mu+\int_{\partial f'}\tr_{\partial f'}\omega\wedge\tr_{\partial f'}\mu.
  \]
  The proof of \eqref{eq:toprove} is concluded invoking the integration by part formula~\eqref{eq:ipp} and writing $\rd\tr_{f'}=\tr_{f'}\rd$ (since the trace is a pullback).
\end{proof}

\subsection{Consistency}\label{sec:complex:polynomial.consistency:ddr}

\begin{proof}[Proof of Theorem~\ref{thm:polynomial.consistency:ddr}]
  The proof is made, as in Theorem~\ref{thm:link.pot.diff}, by induction on $\rho\coloneq d-k$.
  \smallskip

  If $\rho=0$, then $d=k$ and the definitions \eqref{eq:discrete.potential:d=k} of the discrete potential and \eqref{eq:interpolator} of the interpolator give $\Pot{r}{d}{k}\ul I_{r,f}^k\omega = \star^{-1}\trimproj{r}{f}{0}(\star\omega) = \star^{-1}\star\omega = \omega$, where, to remove the projector, we have used the fact that $\star\omega\in\Poly{r}{}\Lambda^0(f) = \Poly{r}{-}\Lambda^0(f)$ (cf.~\eqref{eq:trimmed.spaces:ell=0}).
  \smallskip

  Let us now assume that the lemma holds for a given $\rho\ge 0$, and let us consider $d$ and $k$ such that $d-k=\rho+1$.
  We first consider\eqref{eq:discrete.exterior.derivative:polycons}.
  By relation \eqref{eq:link.pot.diff} applied to $k+1$ instead of $k$ and the commutation property \eqref{eq:commut.d}, we have, for $\omega\in \Poly{r+1}{-}\Lambda^k(f)$,
  \[
  \rd_{r,f}^k\ul I_{r,f}^k\omega=\Pot{r}{f}{k+1}(\ud_{r,f}^k\ul I_{r,f}^k\omega)=\Pot{r}{f}{k+1}\ul I_{r,f}^{k+1}(\rd\omega).
  \]
  We have $\rd\omega\in \rd\Poly{r+1}{-}\Lambda^k(f)\subset\Poly{r}{}\Lambda^{k+1}(f)$, and the pair $(d,k+1)$ satisfies $d-(k+1)=\rho$. We can therefore apply the induction hypothesis to see that \eqref{eq:discrete.potential:polynomial.consistency} holds for this pair and $\rd\omega$ instead of $\omega$; this gives $\rd_{r,f}^k\ul I_{r,f}^k\omega=\rd\omega$ and proves \eqref{eq:discrete.exterior.derivative:polycons}.

  We now turn to \eqref{eq:discrete.potential:polynomial.consistency}. For $\omega\in\Poly{r}{}\Lambda^k(f)$, applying the property \eqref{eq:discrete.exterior.derivative:polycons} that we have just proved to $\omega$ and recalling the definitions \eqref{eq:discrete.potential} and \eqref{eq:interpolator} of the potential and of the interpolator, we find, for all $(\mu,\nu)\in\Koly{r+1}{d-k-1}(f)\times\Koly{r}{d-k}(f)$,
  \begin{multline*}
    (-1)^{k+1}\int_f \Pot{r}{f}{k} \ul I_{r,f}^k\omega \wedge (\rd\mu+\nu)
    \\
    = \int_f \rd\omega \wedge \mu
    -  \int_{\partial f}  \Pot{r}{\partial f}{k}\ul I_{r,\partial f}^k\tr_{\partial f}\omega \wedge \tr_{\partial f}{\mu}
    + (-1)^{k+1}\int_f \cancelto{\omega}{(\star^{-1}\trimproj{r}{f}{d-k}\star\omega)}\wedge\nu,
  \end{multline*}
  the replacement being justified by \eqref{eq:remove.projector} and $\nu\in\Poly{r}{-}\Lambda^{d-k}(f)$ (see \eqref{def:trimmed.spaces}).
  We can then apply the polynomial consistency \eqref{eq:discrete.potential:polynomial.consistency} on each $f'\in\Delta_{d-1}(f)$ (as $(d-1)-k=\rho$) to write $\Pot{r}{\partial f}{k}\ul I_{r,\partial f}^k\tr_{\partial f}\omega = \tr_{\partial f}\omega$ , and then integrate by parts to conclude, since $\mu,\nu$ are generic elements, that $\Pot{r}{f}{k} \ul I_{r,f}^k\omega=\omega$.
\end{proof}

\begin{remark}[Consistency property of the improved potential for $k=0$]\label{rem:consistency.improved.Pot.k=0}
  In the case $k=0$, the improved potential defined in Remark \ref{rem:improved.Pot.k=0} satisfies the following consistency property:
  \[
  \Pot{r+1}{f}{0}\ul I_{r,f}^0\omega = \omega\qquad\forall \omega\in\Poly{r+1}{}\Lambda^0(f).
  \]
  To see this, first notice that when $d=k=0$ we have $\Pot{r+1}{f}{0}=\Pot{r}{f}{0}$ since $\Poly{r+1}{}\Lambda^0(f)=\Poly{r}{}\Lambda^0(f)\cong\Real$, and then, for $d\ge k+1$, invoke the definition \eqref{eq:discrete.potential.k=0} of $\Pot{r+1}{f}{0}\ul I_{r,f}^0\omega$, apply \eqref{eq:discrete.exterior.derivative:polycons} (since $\Poly{r+1}{-}\Lambda^0(f)=\Poly{r+1}{}\Lambda^0(f)$) and a recursion argument on $d$.
\end{remark}

\begin{proof}[Proof of Corollary \ref{cor:polynomial.consistency.smooth:ddr}]
  Since $sp>d$, the Sobolev embedding give $W^{s,p}\Lambda^k(f)\subset C^0(\overline{f})$ and thus the mapping $\Pot{r}{f}{k}\circ\ul I_{r,f}^k: W^{s,p}\Lambda^k(f)\to \Poly{r}{}\Lambda^k(f)$ is well defined.
  Introducing $\lproj{r}{f}{k}\omega=\Pot{r}{f}{k}\ul I_{r,f}^k\lproj{r}{f}{k}\omega$ (the equality coming from \eqref{eq:discrete.potential:polynomial.consistency} applied to $\lproj{r}{f}{k}\omega$ instead of $\omega$) we write,  with hidden constants in $\lesssim$ having the same dependencies as $C$ in \eqref{eq:discrete.potential:polynomial.consistency.smooth},
  \begin{align}
    \seminorm{W^{m,p}\Lambda^k(f)}{\Pot{r}{f}{k}\ul I_{r,f}^k\omega-\omega}\le{}&
    \seminorm{W^{m,p}\Lambda^k(f)}{\Pot{r}{f}{k}\ul I_{r,f}^k(\omega-\lproj{r}{f}{k}\omega)}+\seminorm{W^{m,p}\Lambda^k(f)}{\lproj{r}{f}{k}\omega-\omega}\nonumber\\
    \lesssim{}&
    \sum_{t=m}^sh_f^{t-m}\seminorm{W^{t,p}\Lambda^k(f)}{\omega-\lproj{r}{f}{k}\omega},
    \label{eq:est.PIomega.omega.1}
  \end{align}
  where the second inequality follows from the boundedness \cite[Eq.~(A.21)]{Di-Pietro.Droniou.ea:23*1} of $\Pot{r}{f}{k}\ul I_{r,f}^k$ (the last term in the first line has been included in the sum for $t=m$). If $t\le r+1$, the approximation properties of polynomial $L^2$-projector (\cite[Theorem 1.45]{Di-Pietro.Droniou:20} applied to each component of a fixed basis of alternate forms) yield
  $$
  \seminorm{W^{t,p}\Lambda^k(f)}{\omega-\lproj{r}{f}{k}\omega}\lesssim h_f^{r+1-t}\seminorm{W^{r+1,p}\Lambda^k(f)}{\omega}.
  $$
  If $t > r+1$, since derivatives of order $> r+1$ of $\lproj{r}{f}{k}\omega$ vanish, we have $\seminorm{W^{t,p}\Lambda^k(f)}{\omega-\lproj{r}{f}{k}\omega}
  =\seminorm{W^{t,p}\Lambda^k(f)}{\omega}$. Splitting the sum in the right-hand side of \eqref{eq:est.PIomega.omega.1} between $t\le \min(r+1,s)$ and $t\ge\min(r+1,s)+1$ and applying the results above yields
  \begin{align*}
    \seminorm{W^{m,p}\Lambda^k(f)}{\Pot{r}{f}{k}\ul I_{r,f}^k\omega-\omega}\lesssim{}&
    h_f^{r+1-m}\seminorm{W^{r+1,p}\Lambda^k(f)}{\omega}
    +\sum_{t=\min(r+1,s)+1}^s h_f^{t-m}\seminorm{W^{t,p}\Lambda^k(f)}{\omega}.
  \end{align*}
  (the second sum is actually absent if $s\le r+1$). Writing in the last sum $h_f^{t-m}=h_f^{r+1-m}h_f^{t-r-1}$ and recalling the definition \eqref{eq:def.Wspmax.norm} of $\seminorm{W^{(r+1,s),p}\Lambda^k(f)}{{\cdot}}$ concludes the proof of \eqref{eq:discrete.potential:polynomial.consistency.smooth}.

  \medskip

  To prove \eqref{eq:discrete.exterior.derivative:polycons.smooth}, use the link \eqref{eq:link.pot.diff} between discrete potential and exterior derivative (with $k+1$ instead of $k$) together with the commutation property \eqref{eq:commut.d} to write
  \[
  \rd_{r,f}^k\ul I_{r,f}^k\omega=\Pot{r}{f}{k+1}\ud_{r,f}^k\ul{I}_{r,f}^k\omega=\Pot{r}{f}{k+1}\ul{I}_{r,f}^{k+1}(\rd\omega)
  \]
  and conclude by applying \eqref{eq:discrete.potential:polynomial.consistency.smooth} to $\rd\omega$ instead of $\omega$ and with $k+1$ instead of $k$.
\end{proof}

\begin{proof}[Proof of Lemma \ref{lem:consistency.stabilisation}]
  Notice first that $\omega\in H^{\max(r+1,s)}\Lambda^k(f)$ is continuous over $\overline{f}$ (by Sobolev embedding since $2s>d$), and therefore that $\ul I_{r,f}^k\omega$ is well-defined. It is easily deduced from the polynomial consistency \eqref{eq:discrete.potential:polynomial.consistency} that
  the stabilisation bilinear form vanishes whenever one of its arguments is the interpolate of a polynomial of degree $\le r$. Hence,
  \[
  s_{k,f}(\ul I_{r,f}^k\omega,\ul I_{r,f}^k\omega)=
  s_{k,f}(\ul I_{r,f}^k(\omega-\lproj{r}{f}{k}\omega),\ul I_{r,f}^k(\omega-\lproj{r}{f}{k}\omega)).
  \]
  We then invoke \cite[Lemmas~10 and 11]{Di-Pietro.Droniou.ea:23*1} (with $r\leftarrow s$ and $s\leftarrow 2$) to infer
  \[
  s_{k,f}(\ul I_{r,f}^k\omega,\ul I_{r,f}^k\omega)
  \lesssim\left(\sum_{t=0}^s h_f^t\seminorm{H^t\Lambda^k(f)}{\omega-\lproj{r}{f}{k}\omega}\right)^2
  \]
  with hidden constant having the same dependencies as $C$ in \eqref{eq:consistency.stabilisation}.
  The conclusion follows as in the proof of \eqref{eq:discrete.potential:polynomial.consistency.smooth}: for $t\le r+1$, we invoke the approximation properties \cite[Theorem 1.45]{Di-Pietro.Droniou:20} of $\lproj{r}{f}{k}$ to write $\seminorm{H^t\Lambda^k(f)}{\omega-\lproj{r}{f}{k}\omega}\lesssim
  h_f^{r+1-t}\seminorm{H^{r+1}\Lambda^k(f)}{\omega}$ while, for $t > r+1$, we eliminate $\lproj{r}{f}{k}\omega$ from the semi-norms since its derivatives of degree $> r+1$ vanish.
\end{proof}


\subsection{Cohomology}\label{sec:complex+cohomology:cohomology}

A strategy to establish the exactness of the de Rham complex (for a domain with trivial topology) is to design a Poincar\'e operator $p:C^1\Lambda^k(\overline{\Omega})\to C^1\Lambda^{k-1}(\overline{\Omega})$, that satisfies $\rd p+p\rd={\rm Id}$.
The Poincar\'e operator is built integrating a certain flow of contracted differential forms; see \cite{Christiansen.Hu:18,Lang:99} for details and applications to the design of finite element complexes. Extending such a construction to the context of fully discrete spaces is not trivial, as it is not clear how the discrete polynomial components on cells should evolve with such a flow. We therefore select an alternative approach, more suited to hierarchical discrete spaces.

The starting point is the following idea: if $\eta\in C^1\Lambda^k(\overline{\Omega})$ satisfies $\rd\eta=0$ and we have $\omega\in C^2\Lambda^{k-1}(\overline{\Omega})$ such that $\rd \omega=\eta$, then \eqref{eq:ipp} shows that, for any $d$-cell $f$,
\begin{equation}\label{eq:hierarchical.construction}
  (-1)^k\int_f \omega\wedge \rd\mu = \int_f \eta\wedge \mu - \int_{\partial f}\tr_{\partial f}\omega\wedge\tr_{\partial f}\mu\qquad\forall\mu\in C^1\Lambda^{d-k}(\overline{\Omega}).
\end{equation}
In the discrete setting, $\omega$ is built starting from the lowest-dimensional cells, and \eqref{eq:hierarchical.construction} thus gives a condition on $\omega$ over $f$ based on the already constructed $\tr_{\partial f}\omega$. To start this process, we must fix the values of $\omega$ on the lowest-dimensional cells, which is not an easy task in general. Actually, from the point of view of differential forms, the lowest-dimensional cells encode the topology of the domain, and thus the cohomology of the complex; for a generic $\eta$, the recursive construction of $\omega$ can therefore only be fully complete if the complex is exact, and thus the topology trivial.

This limitation is circumvented by using the following idea: if $\eta$ has zero average on $k$-cells, then $\omega$ can be set to zero on $(k-1)$-cells, which completes the construction above (see Lemma \ref{lem:exact.Xsharp} below). This result is then exploited, through the extension/reduction strategy developed in \cite{Di-Pietro.Droniou:23*2,Di-Pietro.Droniou.ea:23}, to compare the cohomology of the arbitrary-order \DDR{r} complex to that of the lowest-order \DDR{0} complex, which is trivially isomorphic to the CW complex based on the mesh.

We therefore start by considering the subspace $\ul X_{r,h,\flat}^k$ of $\ul X_{r,h}^k$ made of vectors of differential forms whose integrals over cells of dimension $d=k$ vanish:
\[
\ul X_{r,h,\flat}^k
\coloneq\left\{
\ul\omega_h=(\omega_f)_{f\in\Delta_d(\Mh),\,d\in [k,n]}\,:\, \int_f \star^{-1}\omega_f=0\quad\forall f\in\Delta_k(\Mh)
\right\}.
\]

\begin{lemma}[Exactness property for $\ul X_{r,h,\flat}^k$]\label{lem:exact.Xsharp}
  For any integer $k\in [0,n]$, if $\ul\eta_h\in \ul X_{r,h,\flat}^k$ satisfies $\ud_{r,h}^k\ul\eta_h=\ul 0$, then there exists $\ul\omega_h\in\ul X_{r,h,\flat}^{k-1}$ such that $\ul\eta_h=\ud_{r,h}^{k-1}\ul\omega_h$, where, in accordance with \eqref{eq:ddr}, we have set $\ud_{r,h}^{-1} = \ud_{r,h}^n \coloneq 0$.
\end{lemma}

\begin{remark}[Exact sub-complex]
  It can easily be checked that $\ud_{r,h}^k:\ul X_{r,h,\flat}^k\to\ul X_{r,h,\flat}^{k+1}$. As a consequence, the previous lemma shows that $(\ul X_{r,h,\flat}^k,\ud_{r,h}^k)_k$ is an exact sub-complex of \DDR{r} (even if the latter complex is not exact).
\end{remark}

\begin{proof}[Proof of Lemma~\ref{lem:exact.Xsharp}]
  We first notice that the case $r=0$ is trivial since, for all $k$, $\ul X_{0,h,\flat}^k=\{(0)_{f\in\Delta_k(\Mh)}\}$. This comes from the fact that the space $\ul X_{0,h}^k$ only has non-zero components (which are moreover constant) on cells of dimension $d=k$; to check this, notice that the spaces \eqref{eq:trimmed.spaces:ell>=1} are all trivial since the first component vanishes for $k$-forms with constant coefficients, while the second is zero by \eqref{eq:Koly.0.ell=Koly.r.d=0}.
  We can therefore assume that $r\ge 1$. The cases $k=0$ and $k\ge 1$ have to be handled separately.
  \medskip\\
  \noindent\underline{Case $k=0$.} We prove that, if $\ul\eta_h\in\ul X_{r,h,\flat}^0$ and $\ud_{r,h}^0\ul\eta_h=\ul 0$, then $\eta_f=0$ for all $f\in\Delta_d(\Mh)$, $d\in [0,n]$. This is done by induction on $d$. The case $d=0$ follows immediately from the definition of $\ul X_{r,h,\flat}^0$ which shows that the value of $\star^{-1}\eta_f$ on any vertex $f\in\Delta_0(\Mh)$ is zero. Assuming that all components of $\ul\eta_h$ on cells of dimension $d-1\ge 0$ vanish, we now prove that $\eta_f=0$ for all $f\in\Delta_d(\Mh)$.
  Note first that, by \eqref{eq:link.pot.diff}, the property $\ud_{r,f}^0\ul\eta_f=\ul 0$ implies $\rd_{r,f}^0\ul\eta_f=0$. Enforcing then $\ul\eta_{\partial f}=\ul 0$ (by induction hypothesis) in the definition \eqref{eq:discrete.exterior.derivative:ddr} of $\rd_{r,f}^0\ul\eta_f$ gives
  \[
  \int_f \star^{-1}\eta_f\wedge \rd\mu=0\qquad\forall \mu\in\Poly{r}{}\Lambda^{d-1}(f).
  \]
  By definition \eqref{eq:trimmed.spaces:ell>=1} of the trimmed space with $\ell=d$, and accounting for \eqref{eq:Koly.0.ell=Koly.r.d=0}, we have $\rd\Poly{r}{}\Lambda^{d-1}(f)=\Poly{r}{-}\Lambda^d(f)$, so the relation above and \eqref{eq:commut.star.wedge} with $(\omega,\mu)\gets(\eta_f,\rd\mu)$ and $\rho = \rd\mu$ show that $\int_f\eta_f\wedge\star\rho=0$ for all $\rho\in\Poly{r}{-}\Lambda^d(f)$. Since $\eta_f$ belongs to this same space, we conclude that $\eta_f=0$.
  \medskip\\
  \noindent\underline{Case $k\ge 1$.}
  Let $\ul\eta_h\in\ul X_{r,h,\flat}^k$ be such that $\ud_{r,h}^k\ul\eta_h=\ul 0$, and let us construct $\ul\omega_h\in \ul X_{r,h,\flat}^{k-1}$ such that $\ud_{r,h}^{k-1}\ul\omega_h=\ul\eta_h$. This construction of $\ul\omega_h$ is done by increasing dimension $d\in [k-1,n]$ of the cells. For all $f\in\Delta_{k-1}(\Mh)$, we set $\omega_f=0$ (which ensures, in particular, that the zero-average condition embedded in the space $\ul X_{r,h,\flat}^{k-1}$ is fulfilled).
  Assume now that the components of $\ul\omega_h$ have been constructed up to cells of dimension $d-1\ge k-1$, and consider $f\in\Delta_d(\Mh)$. We choose $\omega_f\in \Poly{r}{-}\Lambda^{d-k+1}(f)$ such that the following relation holds:
  \begin{equation}\label{def:omega_f}
    (-1)^k\int_f \star^{-1}\omega_f\wedge\rd\mu = \int_f \Pot{r}{f}{k}\ul\eta_f\wedge\mu-\int_{\partial f}\Pot{r}{\partial f}{k-1}\ul\omega_{\partial f}\wedge\tr_{\partial f}\mu
    \qquad\forall\mu\in \Koly{r}{d-k}(f).
  \end{equation}
  Notice that, since the construction is recursive on the dimension of the cells, $\ul\omega_{\partial f}$ has already been constructed at this stage. Owing to the isomorphism \eqref{eq:isomorphism.Prtrimmed.koszul} with $\ell = d - k + 1\ge 1$, this relation completely defines the projection of $\omega_f$ on $\rd\Koly{r}{d-k}(f)\subset\Poly{r}{-}\Lambda^{d-k+1}(f)$. The projection of $\omega_f$ on the remaining component $\Koly{r}{d-k+1}(f)$ of $\Poly{r}{-}\Lambda^{d-k+1}(f)$ is not relevant to the rest of the proof and can be set to 0.
  \medskip

  Let us now prove that $\ud_{r,h}^{k-1}\ul\omega_h=\ul\eta_h$. It suffices to show that
  \begin{equation}\label{eq:domega=Peta}
    \rd_{r,f}^{k-1}\ul\omega_f=\Pot{r}{f}{k}\ul\eta_f\qquad\forall f\in\Delta_d(\Mh)\,,\;d\in [k,n].
  \end{equation}
  Indeed, applying $\trimproj{r}{f}{d-k}\star$ to this relation and using \eqref{eq:proj.potential}
  yields $\trimproj{r}{f}{d-k}(\star\rd_{r,f}^{k-1}\ul\omega_f)=\eta_f$; using this relation for all cells $f$, and recalling the definition \eqref{eq:global.discrete.exterior.derivative} of the global discrete exterior derivative (with $k-1$ instead of $k$), then gives $\ud_{r,h}^{k-1}\ul\omega_h=\ul\eta_h$ as claimed.

  The relation \eqref{eq:domega=Peta} is a direct consequence of the following property:
  \begin{equation}\label{eq:domega=eta}
    \int_f \rd_{r,f}^{k-1}\ul\omega_f \wedge\mu=\int_f\Pot{r}{f}{k}\ul\eta_f\wedge\mu\qquad\forall\mu\in\Poly{r}{}\Lambda^{d-k}(f).
  \end{equation}
  Owing to \eqref{eq:decomposition.Pr}, we only need to prove this relation first for $\mu\in\Koly{r}{d-k}(f)$, and then $\mu\in\Poly{0}{}\Lambda^0(f)$ if $d=k$ or $\mu\in\rd\Poly{r+1}{}\Lambda^{d-k-1}(f)$ if $d\ge k+1$.

  If $\mu\in\Koly{r}{d-k}(f)$, the definition \eqref{eq:discrete.exterior.derivative:ddr} of $\rd_{r,f}^{k-1}\ul\omega_f$ together with the property \eqref{def:omega_f} immedia\-tely give \eqref{eq:domega=eta}.

  Let us consider the case $d=k$ and $\mu\in\Poly{0}{}\Lambda^0(f)$. Then $\rd\mu=0$, so the definition \eqref{eq:discrete.exterior.derivative:ddr} of $\ud_{r,f}^{k-1}\ul\omega_f$ and $\ul\omega_{\partial f}=0$ (by construction, $\ul\omega_h$ vanishes on cells of dimension $d-1=k-1$) show that the left-hand side of \eqref{eq:domega=eta} vanishes. Since $\Pot{r}{f}{k}\ul\eta_f=\star^{-1}\eta_f$ (see \eqref{eq:discrete.potential:d=k}) and $\int_f\star^{-1}\eta_f=0$ as $\ul\eta_h\in\ul X_{r,h,\flat}^{k}$, the right-hand side of \eqref{eq:domega=eta} vanishes as well, and this relation holds.

  Finally, we turn to the case $d\ge k+1$ and $\mu\in \rd\Poly{r+1}{}\Lambda^{d-k-1}(f)$, which is proved by induction on $d$ (the base case $d=k$ having already been covered).
  By \eqref{eq:dP.dkappaP} with $(\ell,r)\gets(d-k-1,r+1)$, we have $\mu\in \rd\Koly{r+1}{d-k-1}(f)$, and we can therefore write $\mu=\rd\alpha$ with $\alpha\in\Koly{r+1}{d-k-1}(f)\subset \Poly{r+1}{-}\Lambda^{d-k-1}(f)$ (see \eqref{def:trimmed.spaces}).
  Invoking the link \eqref{eq:link.discrete.exterior.derivatives} between discrete exterior derivatives on subcells (notice that $d\ge (k-1)+2$), we obtain
  \[
  \int_f \rd_{r,f}^{k-1}\ul\omega_f \wedge\mu=(-1)^k\int_{\partial f}\rd_{r,\partial f}^{k-1}\ul\omega_{\partial f}\wedge\tr_{\partial f}\alpha
  =(-1)^k\int_{\partial f}\Pot{r}{\partial f}{k}\ul\eta_{\partial f}\wedge\tr_{\partial f}\alpha,
  \]
  where the second equality follows from the induction hypothesis that \eqref{eq:domega=eta} holds on subcells of $f$.
  We have $\ud_{r,f}^k\ul\eta_f=0$ and $d\ge k+1$, so we can apply \eqref{eq:link.pot.diff} with $k+1$ instead of $k$ to get $\rd_{r,f}^k\ul\eta_f=0$; the definition \eqref{eq:discrete.potential} of $\Pot{r}{f}{k}\ul\eta_f$ (with $(\mu,\nu)\gets(\alpha,0)$, see Remark \ref{rem:discrete.potential:validity} for the validity of this choice of $\mu$) allows us to continue with
  \[
  \int_f \rd_{r,f}^{k-1}\ul\omega_f \wedge\mu=-(-1)^k\times(-1)^{k+1}\int_f\Pot{r}{f}{k}\ul\eta_f\wedge\rd\alpha.
  \]
  Recalling that $\rd\alpha=\mu$ concludes the proof of \eqref{eq:domega=eta}.
\end{proof}

\begin{proof}[Proof of Theorem~\ref{thm:cohomology:ddr}]
  As in \cite[Lemma 4]{Di-Pietro.Droniou.ea:23}, it is straightforward to see that the (discrete) de Rham map establishes a chain isomorphism between the lowest-degree complex \DDR{0} and the CW complex defined by $\Mh$.
  Since this CW complex has the same cohomology as the de Rham complex \eqref{eq:diff.de.rham}, the proof is complete if we show that the cohomology of \DDR{r} is isomorphic to the cohomology of \DDR{0}.
  This obviously means that we can assume $r\ge 1$ in the following.
  \medskip\\
  \underline{\emph{Step 1: Reductions and extensions.}}
  With the goal of applying \cite[Proposition 2]{Di-Pietro.Droniou:23*2}, we define reduction and extension maps between \DDR{r} and \DDR{0} as in \eqref{eq:diagram.coho}.
  \begin{equation}\label{eq:diagram.coho}
    \begin{tikzcd}
      \DDR{r}:\quad\cdots\arrow{r} & \ul X_{r,h}^k
      \arrow{r}{\ud_{r,h}^k}\arrow[bend left]{d}{\Red{h}{k}} &[3em] \ul X_{r,h}^{k+1}
      \arrow{r}{}\arrow[bend left]{d}{\Red{h}{k+1}} & \cdots\\[2em]
      \DDR{0}:\quad\cdots\arrow{r} & \ul X_{0,h}^k
      \arrow{r}{\ud_{0,h}^k}\arrow[bend left]{u}{\Ext{h}{k}} &[3em] \ul X_{0,h}^{k+1}
      \arrow{r}{}\arrow[bend left]{u}{\Ext{h}{k+1}} & \cdots
    \end{tikzcd}
  \end{equation}
  The reduction $\Red{h}{k}:\ul X_{r,h}^k\to \ul X_{0,h}^k$ is defined taking the average of components on the cells of dimension $k$ (recall that vectors in $\ul X_{0,h}^k$ only have components on these cells): For all $\ul\omega_h\in\ul X_{r,h}^k$,
  \begin{equation}\label{def:RdE}
    \Red{h}{k}\ul\omega_h = (\lproj{0}{f}{0}\omega_f)_{f\in\Delta_k(\Mh)}.
  \end{equation}
  The extension $\Ext{h}{k}:\ul X_{0,h}^k\to\ul X_{r,h}^k$ is defined by induction on the cell dimension:
  For all $\ul\eta_h\in\ul X_{0,h}^k$, $\Ext{h}{k}\ul\eta_h\coloneq(\fExt{f}{k}\ul\eta_f)_{f\in\Delta_d(\Mh),\,d\in[k,n]}$, where
  \begin{subequations}
    \begin{itemize}
    \item If $d=k$,
      \begin{equation}\label{def:Ext.d=k}
        \fExt{f}{k}\ul\eta_f=\eta_f\in \Poly{0}{-}\Lambda^0(f)\subset \Poly{r}{-}\Lambda^0(f);
      \end{equation}
    \item If $d\ge k+1$, $\fExt{f}{k}\ul\eta_f\in\Poly{r}{-}\Lambda^{d-k}(f)$ satisfies
      \begin{multline}\label{def:Ext.d>=k+1}
        (-1)^{k+1}
        \int_f \star^{-1}\fExt{f}{k}\ul\eta_f\wedge (\rd\mu+\nu)
        \\
        =\int_f \rd_{0,f}^k\ul\eta_f\wedge\mu
        -\int_{\partial f}\Pot{r}{\partial f}{k}\Ext{\partial f}{k}\ul\eta_{\partial f}\wedge\tr_{\partial f}\mu
        +(-1)^{k+1}\int_f \Pot{0}{f}{k}\ul\eta_f\wedge\nu
        \\
        \forall (\mu,\nu)\in \Koly{r}{d-k-1}(f)\times\Koly{r}{d-k}(f),
      \end{multline}
      where $\Ext{\partial f}{k}\ul\eta_{\partial f}=(\Ext{f'}{k}\ul\eta_{f'})_{f'\in\Delta_{d-1}(f)}$ gathers the extensions already built at previous steps on the subcells of dimension $d-1$ of $f$. The isomorphism \eqref{eq:isomorphism.Prtrimmed.koszul} with $\ell = d - k\ge 1$ ensures that the relation above fully and properly defines $\fExt{f}{k}\ul\eta_f$.
    \end{itemize}
  \end{subequations}
  Extensions are designed in such a way that, for all $f\in\Delta_d(\Mh)$ with $d\ge k+1$,
  \[
  \int_f \rd_{r,f}^k\Ext{f}{k}\ul\eta_f\wedge \mu
  = \int_f \rd_{0,f}^k\ul\eta_f\wedge\mu\qquad\forall \mu\in \Koly{r}{d-k-1}(f),
  \]
  as can be checked combining the definitions \eqref{eq:discrete.exterior.derivative:ddr} of $\rd_{r,f}^k\Ext{f}{k}\ul\eta_f$ and \eqref{def:Ext.d>=k+1} of $\fExt{f}{k}\ul\eta_f$ (with $\nu=0$).
  Since $\rd_{0,f}^k\ul\eta_f=\Pot{0}{f}{k+1}\ud_{0,f}^k\ul\eta_f$ by \eqref{eq:link.pot.diff}, using the definition of $\fExt{f}{k+1}\ud_{0,f}^k\ul\eta_f$ (namely, \eqref{def:Ext.d=k} if $d=k+1$, or \eqref{def:Ext.d>=k+1} with $(\mu,\nu)\gets(0,\mu)$ if $d\ge k+2$) we deduce that
  \begin{equation}\label{eq:dE.P0d0}
    \int_f \rd_{r,f}^k\Ext{f}{k}\ul\eta_f\wedge \mu =
    \int_f \star^{-1}\fExt{f}{k+1}\ud_{0,f}^k\ul\eta_f\wedge\mu\qquad\forall \mu\in \Koly{r}{d-k-1}(f).
  \end{equation}
  \medskip

  \noindent\underline{\emph{Step 2: Proof of the theorem.}}
  To apply \cite[Proposition 2]{Di-Pietro.Droniou:23*2}, we need to prove that
  \begin{equation}\label{eq:RdE}
    \ud_{0,h}^k=\Red{h}{k+1}\ud_{r,h}^k\Ext{h}{k}
  \end{equation}
  and that \cite[Assumption 1]{Di-Pietro.Droniou:23*2} holds, that is:
  \begin{itemize}
  \item[(C1)] $\Red{h}{k}\Ext{h}{k}={\rm Id}$ on $\Ker\ud_{0,h}^k$;
  \item[(C2)] $(\Ext{h}{k}\Red{h}{k}-{\rm Id})(\Ker\ud_{r,h}^k)\subset \Image\ud_{r,h}^{k-1}$;
  \item[(C3)] The graded maps $\Ext{h}{\bullet}$ and $\Red{h}{\bullet}$ are cochain maps.
  \end{itemize}

  We start by noticing that, since \DDR{0} is already known to be a complex, (C1) and (C3) imply \eqref{eq:RdE}.
  Indeed, (C3) gives $\Red{h}{k+1}\ud_{r,h}^k\Ext{h}{k}=\Red{h}{k+1}\Ext{h}{k+1}\ud_{0,h}^{k}$ and, by the complex property, $\Image\ud_{0,h}^k\subset\Ker\ud_{0,h}^{k+1}$, so (C1) applied to $k+1$ instead of $k$ yields \eqref{eq:RdE}.
  \medskip\\
  \underline{1. \emph{Proof of (C1).}}
  The definitions \eqref{def:RdE} and \eqref{def:Ext.d=k} of the reduction and the extension components on the lowest dimensional cells directly shows that $\Red{h}{k}\Ext{h}{k}={\rm Id}$ on $\ul X_{r,h}^k$, which establishes a stronger result than (C1).
  \medskip\\
  \underline{2. \emph{Proof of (C3) for the extension.}}
  We now turn to (C3), considering first the case of the extension. We have to show that, for all $\ul\eta_h\in\ul X_{0,h}^k$ it holds $\ud_{r,h}^k\Ext{h}{k}\ul\eta_h=\Ext{h}{k+1}\ud_{0,h}^k\ul\eta_h$.
  Given the definitions \eqref{eq:global.discrete.exterior.derivative} of the global discrete exterior derivative and of the extension, this boils down to showing that
  \begin{equation*}
    \star^{-1}\trimproj{r}{f}{d-k-1}(\star \rd_{r,f}^k\Ext{f}{k}\ul\eta_f)=\star^{-1}\fExt{f}{k+1}\ud_{0,f}^k\ul\eta_f
    \qquad\forall f\in\Delta_d(\Mh)\mbox{ with }d\ge k+1,
  \end{equation*}
  which, testing against $\rho\in\Poly{r}{-}\Lambda^{d-k-1}(f)$ and recalling the relation \eqref{eq:remove.projector}, can be recast as
  \begin{multline}\label{eq:Ext.cochain}
    \int_f \rd_{r,f}^k\Ext{f}{k}\ul\eta_f\wedge \rho
    = \int_f \star^{-1}\fExt{f}{k+1}\ud_{0,f}^k\ul\eta_f\wedge\rho
    \\
    \forall f\in\Delta_d(\Mh)\mbox{ with }d\ge k+1,\,
    \forall\rho\in\Poly{r}{-}\Lambda^{d-k-1}(f).
  \end{multline}
  We start by noticing that, by \eqref{eq:dE.P0d0}, the relation \eqref{eq:Ext.cochain} holds for $\rho\in\Koly{r}{d-k-1}(f)$.
  The decompositions \eqref{eq:decomposition.Pr:ell=0} of $\Poly{r}{}\Lambda^0(f)=\Poly{r}{-}\Lambda^0(f)$ (if $d=k+1$) and \eqref{def:trimmed.spaces.bis}  of $\Poly{r}{-}\Lambda^{d-k-1}(f)$ (if $d\ge k+2$) then show that we only have to prove \eqref{eq:Ext.cochain} for $\rho\in\Poly{0}{}\Lambda^0(f)$ (if $d=k+1$) or $\rho\in \rd\Koly{r}{d-k-2}(f)$ (if $d\ge k+2$).
  This fact is proved by induction on $d$:

  \begin{itemize}
  \item Let us first consider $d=k+1$ and take $\rho\in\Poly{0}{}\Lambda^0(f)$. We can use this polynomial form as a test function in the definition \eqref{eq:discrete.exterior.derivative:ddr} of $\rd_{0,f}^k\ul\eta_f$ to get
    \begin{equation}\label{eq:Ext.cochain.P0}
      \int_{\partial f}\Pot{0}{\partial f}{k}\ul\eta_{\partial f}\wedge\tr_{\partial f}\rho=\int_f \rd_{0,f}^k\ul\eta_f\wedge\rho=
      \int_f \star^{-1}\fExt{f}{k+1}\ud_{0,f}^k\ul\eta_f\wedge\rho
      \qquad\forall\rho\in\Poly{0}{}\Lambda^0(f),
    \end{equation}
    where the second equality follows from \eqref{def:Ext.d=k} with $(k,\ul\eta_f)\gets({k+1},\ud_{0,f}^k\ul\eta_f)$.
    For all $f'\in\Delta_k(f)$, by definition \eqref{eq:discrete.potential:d=k} of $\Pot{0}{f'}{k}$ and \eqref{def:Ext.d=k} of $\fExt{f'}{k}$, we have $\Pot{0}{f'}{k}\ul\eta_{f'}=\star^{-1}\eta_{f'}=\star^{-1}\fExt{f'}{k}\ul\eta_{f'}=\Pot{r}{f'}{k}\Ext{f'}{k}\ul\eta_{f'}$, where the last relation follows applying the definition \eqref{eq:discrete.potential:d=k} of $\Pot{r}{f'}{k}$.
    We infer from this equality and \eqref{eq:Ext.cochain.P0} that
    \begin{equation*}
      \int_{\partial f}\Pot{r}{\partial f}{k}\Ext{\partial f}{k}\ul\eta_{\partial f}\wedge\tr_{\partial f}\rho=\int_f \star^{-1}\fExt{f}{k+1}\ud_{0,f}^k\ul\eta_f\wedge\rho
      \qquad\forall\rho\in\Poly{0}{}\Lambda^0(f).
    \end{equation*}
    Applying the definition \eqref{eq:discrete.exterior.derivative:ddr} of $\rd_{r,f}^k\Ext{f}{k}\ul\eta_{\partial f}$ with $\mu=\rho$ (which satisfies $\rd\rho=0$) to the left-hand side then concludes the proof of \eqref{eq:Ext.cochain}.
  \item We now take $d\ge k+2$ and $\rho\in \rd\Koly{r}{d-k-2}(f)$, which we write $\rho=\rd\alpha$ with $\alpha\in \Koly{r}{d-k-2}(f)\subset \Poly{r}{-}\Lambda^{d-k-2}(f)$. Applying the link \eqref{eq:link.discrete.exterior.derivatives} between discrete exterior derivatives on $f$ and $\partial f$, we have
    \begin{equation}\label{eq:Ext.cochain.k+2}
      \int_f  \rd_{r,f}^k\Ext{f}{k}\ul\eta_f\wedge \rho=(-1)^{k+1}\int_{\partial f}\rd_{r,\partial f}^k\Ext{\partial f}{k}\ul\eta_{\partial f}\wedge\tr_{\partial f}\alpha.
    \end{equation}
    By Lemma \ref{lem:traces.trimmed}, for all $f'\in\Delta_{d-1}(f)$, $\tr_{f'}\alpha\in\Poly{r}{-}\Lambda^{d-k-2}(f')$, so we can apply \eqref{eq:Ext.cochain} on $f'$ (by the induction hypothesis) to get
    \[
    \int_{f'}\rd_{r,f'}^k\Ext{f'}{k}\ul\eta_{f'}\wedge\tr_{f'}\alpha
    =\int_{f'}\star^{-1}\fExt{f'}{k+1}\ud_{0,f'}^k\ul\eta_{f'}\wedge\tr_{f'}\alpha
    =\int_{f'}\Pot{r}{f'}{k+1}\Ext{f'}{k+1}\ud_{0,f'}^k\ul\eta_{f'}\wedge\tr_{f'}\alpha,
    \]
    the second equality being justified by \eqref{eq:proj.potential} and \eqref{eq:remove.projector} (with $(\mathcal X,f,d,k)\gets (\Poly{r}{-}\Lambda^{(d-1)-(k+1)}(f'), f', d-1, k+1)$) and the fact that $\tr_{f'}\alpha\in \Poly{r}{-}\Lambda^{(d-1)-(k+1)}(f')$. Plugging this relation into \eqref{eq:Ext.cochain.k+2} yields
    \[
    \int_f  \rd_{r,f}^k\Ext{f}{k}\ul\eta_f\wedge \rho=(-1)^{k+1}\int_{\partial f}\Pot{r}{\partial f}{k+1}\Ext{\partial f}{k+1}\ud_{0,\partial f}^k\ul\eta_{\partial f}\wedge\tr_{\partial f}\alpha.
    \]
    Invoking then the definition \eqref{def:Ext.d>=k+1} of $\fExt{f}{k+1}\ud_{0,f}^k\ul\eta_f$ with $(k,\mu,\nu,\ul\eta_f)\gets(k+1,\alpha,0,\ud_{0,f}\ul\eta_f)$, and using the property $\rd_{0,f}^{k+1}\circ\ud_{0,f}^k=0$ (consequence of \eqref{eq:complex.prop} with $k+1$ instead of $k$) we infer
    \[
    \int_f  \rd_{r,f}^k\Ext{f}{k}\ul\eta_f\wedge \rho=\int_f \star^{-1}\fExt{f}{k+1}\ud_{0,f}^k\ul\eta_f\wedge\rd\alpha
    \]
    and \eqref{eq:Ext.cochain} follows by recalling that $\rho=\rd\alpha$.
  \end{itemize}
  \underline{3. \emph{Proof of (C3) for the reduction.}}
  To conclude the proof of (C3), it remains to show that $\Red{h}{k+1}\ud_{r,h}^k\ul\omega_h=\ud_{0,h}^k\Red{h}{k}\ul\omega_h$ for all $\ul\omega_h\in\ul X_{r,h}^k$. Since vectors in $\ul X_{0,h}^{k+1}$ only have constant components on cells of dimension $k+1$, and since $\Red{h}{k+1}$ is defined by \eqref{def:RdE}, we only have to show that
  \begin{equation}\label{eq:Red.cochain}
    \int_f \rd_{r,f}^k\ul\omega_f\wedge\rho = \int_f \rd_{0,f}^k\Red{f}{k}\ul\omega_f\wedge\rho\qquad\forall f\in\Delta_{k+1}(\Mh)\,,\;\forall
    \rho\in\Poly{0}{}\Lambda^0(f).
  \end{equation}
  Let $\rho$ as above and apply the definition \eqref{eq:discrete.exterior.derivative:ddr} of $\rd_{r,f}^k\ul\omega_f$ to $\mu=\rho$; accounting for $\rd\rho=0$, we obtain
  \begin{equation}\label{eq:C3.red.1}
    \int_f \rd_{r,f}^k\ul\omega_f\wedge\rho = \int_{\partial f}\Pot{r}{\partial f}{k}\ul\omega_{\partial f}\wedge\tr_{\partial f}\rho.
  \end{equation}
  For each $f'\in\Delta_k(f)$, by definition \eqref{eq:discrete.potential:d=k} of $\Pot{r}{f'}{k}$, we can write
  \begin{equation}\label{eq:C3.red.2}
    \int_{f'}\Pot{r}{f'}{k}\ul\omega_f\wedge\tr_{f'}\rho
    = \int_{f'}\star^{-1}\omega_f\wedge\tr_{f'}\rho
    = \int_{f'}\star^{-1}\lproj{0}{f}{0}\omega_f\wedge\tr_{f'}\rho
    = \int_{f'}\Pot{0}{f'}{k}\Red{f'}{k}\ul\omega_{f'}\wedge\tr_{f'}\rho,
  \end{equation}
  where we have used the fact that $\tr_{f'}\rho\in\Poly{0}{}\Lambda^0(f')$ to insert the projector in the second equality
  and the definitions \eqref{def:RdE} of $\Red{f'}{k}$ and \eqref{eq:discrete.potential:d=k} of $\Pot{0}{f'}{k}$ to conclude. Combining \eqref{eq:C3.red.1} and \eqref{eq:C3.red.2}, we find
  \[
  \int_f \rd_{r,f}^k\ul\omega_f\wedge\rho
  = \int_{\partial f}\Pot{0}{\partial f}{k}\Red{\partial f}{k}\ul\omega_{\partial f}\wedge\tr_{\partial f}\rho.
  \]
  Applying the definition \eqref{eq:discrete.exterior.derivative:ddr} of $\rd_{0,f}^k\Red{f}{k}\ul\omega_f$ then concludes the proof of \eqref{eq:Red.cochain}.
  \medskip\\
  \underline{4. \emph{Proof of (C2).}}
  Finally, to prove (C2), we notice that if $\ul\omega_h\in\ul X_{r,h}^k$, then by \eqref{def:RdE} and \eqref{def:Ext.d=k} the components of $\Ext{h}{k}\Red{h}{k}\ul\omega_h$ on the lowest dimensional cells $f\in\Delta_k(\Mh)$ are just the averages of the components of $\ul\omega_h$ on these cells; hence, $\Ext{h}{k}\Red{h}{k}\ul\omega_h-\ul\omega_h\in\ul X_{r,h,\flat}^k$.
  Moreover, by the cochain map property (C3), $\ud_{r,h}^k(\Ext{h}{k}\Red{h}{k}\ul\omega_h-\ul\omega_h)=\Ext{h}{k}\Red{h}{k}\ud_{r,h}^k\ul\omega_h-\ud_{r,h}^k\ul\omega_h= \ul 0$ whenever $\ul\omega_h\in\Ker\ud_{r,h}^k$. We can thus, for such an $\ul\omega_h$, apply Lemma \ref{lem:exact.Xsharp} with $\ul\omega_h \gets \Ext{h}{k}\Red{h}{k}\ul\omega_h-\ul\omega_h$ to see that this element belongs to $\Image\ud_{r,h}^{k-1}$, establishing (C2).
\end{proof}


\section{A VEM-inspired complex}\label{sec:vem}

In this section we consider an alternative construction inspired by the Virtual Element complex of~\cite{Beirao-da-Veiga.Brezzi.ea:18}.
This complex hinges on Koszul complements, unlike the one of \cite{Beirao-da-Veiga.Brezzi.ea:16}, which was based on orthogonal complements (as noticed in \cite{Di-Pietro.Droniou:23*1}, the latter are less natural to prove analytical properties).
Notice that we make here no effort to reduce the polynomial degree of certain components of the discrete spaces, which is known to be possible; see, e.g., \cite{Beirao-da-Veiga.Brezzi.ea:18*1} and also \cite{Di-Pietro.Droniou:23*2} for a general framework with application to DDR methods.
Notice also that we work in a fully discrete spirit, without attempting to identify the underlying virtual spaces (which are not needed for the purposes of the present work).

Let again a polynomial degree $r\ge 0$ be fixed.
The general principle to design the VEM-inspired sequence is to select polynomial components that make it possible to reconstruct, for each $d$-cell and inductively  on the dimension $d$, a discrete potential capable of reproducing polynomial forms in $\Poly{r+1}{-}\Lambda^k(f)$.
The main difference with respect to the DDR approach illustrated in Section \ref{sec:ddr} is that, with the exception of $(k+1)$-cells, the required information on the discrete exterior derivative is directly encoded in the discrete spaces.

Adopting this approach has several, far-reaching, consequences.
The first one is that the discrete spaces contain a mix of both traces and exterior derivatives (which, in passing, requires higher regularity on the domains of the interpolators).
The components on $k$- and $(k+1)$-cells in the discrete space of $k$-forms play a slightly different role than the others (and are, as a result, treated separately in the definition of the space).
The second consequence is that the proofs of key properties (polynomial consistency, cohomology, etc.) are carried out by induction on the dimension (and not on the difference between the dimension and the form degree, as in Theorems \ref{thm:polynomial.consistency:ddr} and \ref{thm:link.pot.diff}).
This leads to somewhat simpler arguments, at the cost of larger discrete spaces.
Also, the commutation property is essentially obtained by definition of the local discrete exterior derivative (with the exception of lowest-dimensional cells).

\subsection{Definition}

\subsubsection{Discrete spaces}

We define the following discrete counterpart of $H\Lambda^k(\Omega)$, $0\le k\le n$:
\begin{multline}\label{eq:spaces:vem}
  \ul V_{r,h}^k
  \coloneq
  \bigtimes_{f\in\Delta_k(\Mh)}\Poly{r}{}\Lambda^0(f)
  \times
  \bigtimes_{f\in\Delta_{k+1}(\Mh)}\left(\Koly{r+1}{1}(f)\times\Koly{r}{0}(f)\right)
  \\
  \times
  \bigtimes_{d=k+2}^n\bigtimes_{f\in\Delta_d(\Mh)}\left(\Koly{r+1}{d-k}(f)\times\Koly{r+1}{d-k-1}(f)\right).
\end{multline}
Notice that, on $(k+1)$-cells, the second component has polynomial degree reduced by one compared to $d$-cells with $d\ge k+2$, i.e., we have $\Koly{r}{0}(f)$ instead of $\Koly{r+1}{0}(f)$.
A generic element of $\ul V_{r,h}^k$ will be denoted by
\begin{equation}\label{eq:vem:omega.h}
  \ul\omega_h = \big( (\omega_f)_{f\in\Delta_k(\Mh)}, (\omega_f, D_{\omega,f})_{f\in\Delta_d(\Mh),\,d\in[k+1,n]}\big).
\end{equation}
The notation $D_{\omega,f}$ is reminiscent of the fact that these polynomial components are interpreted as Hodge stars of exterior derivatives.
We refer to Table \ref{tab:local.spaces:vem} for an overview of the polynomial unknowns in $\ul V_{r,f}^k$ in dimensions 0 to 3, as well as their vector proxies.

\begin{table}\centering
  \renewcommand{\arraystretch}{1.3}
  \begin{tabular}{c|cccc}
    \toprule
    \diagbox[font=\footnotesize]{$k$}{$d$}  & $0$ & $1$ & $2$ & $3$ \\
    \midrule
    0
    & $\Real = \Poly{r}{}\Lambda^0(f_0)$
    & $\{0\}\times\Koly{r}{0}(f_1)$
    & $\{0\}\times\Koly{r+1}{1}(f_2)$
    & $\{0\}\times\Koly{r+1}{2}(f_3)$
    \\
    1
    &
    & $\Poly{r}{}\Lambda^0(f_1)$
    & $\Koly{r+1}{1}(f_2) \times \Koly{r}{0}(f_2)$
    & $\Koly{r+1}{2}(f_3) \times \Koly{r+1}{1}(f_3)$
    \\
    2
    &
    &
    & $\Poly{r}{}\Lambda^0(f_2)$
    & $\Koly{r+1}{1}(f_3) \times \Koly{r}{0}(f_3)$
    \\
    3 & & & & $\Poly{r}{}\Lambda^0(f_3)$ \\
    \midrule[1pt]
    \diagbox[font=\footnotesize]{$k$}{$d$}  & $0$ & $1$ & $2$ & $3$ \\
    \midrule
    0
    & $\Real = \Poly{r}{}(f_0)$
    & $\{0\}\times\Poly{r}{\flat}(f_1)$
    & $\{0\}\times\cRoly{r+1}(f_2)$
    & $\{0\}\times\cRoly{r+1}(f_3)$ \\
    1
    &
    & $\Poly{r}{}(f_1)$
    & $\cRoly{r+1}(f_2)\times\Poly{r}{\flat}(f_2)$
    & $\cRoly{r+1}(f_3)\times\cGoly{r+1}(f_3)$ \\
    2
    &
    &
    & $\Poly{r}{}(f_2)$ & $\cGoly{r+1}(f_3)\times\Poly{r}{\flat}(f_3)$ \\
    3 & & & & $\Poly{r}{}(f_3)$ \\
    \bottomrule
  \end{tabular}
  \caption{Polynomial components attached to each mesh entity $f_d$ of dimension $d\in\{0,\ldots,3\}$ for the space $\ul V_{r,h}^k$ for $k\in\{0,\ldots,3\}$ (top) and counterparts through vector proxies (bottom).\label{tab:local.spaces:vem}}
\end{table}

\subsubsection{Interpolators}

For all integers $0\le k\le d\le n$ and any $f\in\Delta_d(\Mh)$, the local interpolator is such that, for all $\omega\in C^1\Lambda^k(\overline{f})$,
\begin{equation}\label{eq:interpolator:vem}
  \begin{aligned}
    \ul I_{r,f}^k\omega
    \coloneq\Big(
    &\big(\lproj{r}{f'}{0}(\star\tr_{f'}\omega)\big)_{f'\in\Delta_k(f)},
    \\
    &\big(
    \kproj{r+1}{f'}{1}(\star\tr_{f'}\omega),
    \kproj{r}{f'}{0}(\star\tr_{f'}\rd\omega)\big)_{f'\in\Delta_{k+1}(f)}
    \big),
    \\
    &\big(
    \kproj{r+1}{f'}{d'-k}(\star\tr_{f'}\omega),
    \kproj{r+1}{f'}{d'-k-1}(\star\tr_{f'}\rd\omega)\big)_{f'\in\Delta_{d'}(f),\,d'\in[k+2,d]}
    \Big).
  \end{aligned}
\end{equation}

\begin{remark}[Domain of the interpolator]
  Owing to the presence of polynomial components that are interpreted as exterior derivatives (compare \eqref{eq:interpolator} with \eqref{eq:interpolator:vem}), the interpolator in the VEM-inspired construction requires higher regularity of the interpolated functions compared to the DDR complex presented in Section \ref{sec:ddr}, namely $C^1\Lambda^k(\overline{f})$ instead of $C^0\Lambda^k(\overline{f})$.
\end{remark}

\subsubsection{Global discrete exterior derivative and VEM complex}

For all $f\in\Delta_{k+1}(\Mh)$, we define the \emph{discrete exterior derivative} $\rd_{r,f}^k:\ul V_{r,f}^k\to\Poly{r}{}\Lambda^{k+1}(f)$ such that, for all $\ul\omega_f\in\ul V_{r,f}^k$,
\begin{multline}\label{eq:local.exterior.derivative:vem}
  \int_f\rd_{r,f}^k\ul\omega_f\wedge(\mu+\nu)
  = \int_{\partial f}\star^{-1}\omega_{\partial f}\wedge\tr_{\partial f}\mu
  +\int_f\star^{-1}D_{\omega,f}\wedge\nu
  \\
  \forall(\mu,\nu)\in\Poly{0}{}\Lambda^0(f)\times\Koly{r}{0}(f),
\end{multline}
where, as before, $\omega_{\partial f}$ is defined by $(\omega_{\partial f})_{|f'}=\omega_{f'}\in \Poly{r}{}\Lambda^0(f')$ for all $f'\in\Delta_k(f)$.
Notice that the above equation defines $\rd_{r,f}^k\ul\omega_f$ uniquely since, by \eqref{eq:decomposition.Pr:ell=0}, $\mu + \nu$ spans $\Poly{r}{}\Lambda^0(f)$ as $(\mu,\nu)$ spans $\Poly{0}{}\Lambda^0(f)\times\Koly{r}{0}(f)$.
Moreover, taking $\mu = 0$ and letting $\nu$ span $\Koly{r}{0}(f)$, we infer, using \eqref{eq:remove.projector} with $(\mathcal X,\omega,\mu)\gets(\Koly{r}{0}(f),\rd_{r,f}^k\ul \omega_f,\nu)$,
\begin{equation}\label{eq:Domegaf=proj.drfk}
  D_{\omega,f} = \kproj{r}{f}{0}(\star\rd_{r,f}^k\ul\omega_f)\qquad\forall f\in\Delta_{k+1}(\Mh).
\end{equation}

Unlike the DDR complex, the construction of a \emph{global discrete exterior derivative} for the VEM complex does not require to first reconstruct traces on lower-dimensional cells, as all the necessary information is encoded in the polynomial components $(D_{\omega,f})_{f\in\Delta_d(\Mh),\,d\in[k+2,n]}$
supplemented by $(\rd_{r,f}^k\ul\omega_f)_{f\in\Delta_{k+1}(\Mh)}$.
More specifically, for all integers $k\in[0,n-1]$, we let $\ud_{r,h}^k:\ul V_{r,h}^k\to\ul V_{r,h}^{k+1}$ be such that, for all $\ul\omega_h\in\ul V_{r,h}^k$,
\begin{equation}\label{eq:discrete.exterior.derivative:vem}
  \ud_{r,h}^k\ul\omega_h\coloneq\big(
  (\star\rd_{r,f}^k\ul\omega_f)_{f\in\Delta_{k+1}(\Mh)},
  (D_{\omega,f},0)_{f\in\Delta_d(\Mh),\,d\in[k+2,n]}
  \big)
\end{equation}
(compare with \eqref{eq:vem:omega.h} and notice the different positioning, compared to $\ul\omega_h$, of the polynomial components $D_{\omega,f}$).
As for the DDR complex, we will denote by $\ud_{r,f}^k$ the restriction of $\ud_{r,h}^k$ to $f\in\Delta_d(\Mh)$ with $d\in [0,n]$ such that $k\le d-1$.

The VEM sequence of spaces and operators then reads
\begin{equation}\label{eq:vem}
  \begin{tikzcd}
    \VEM{r}\coloneq
    \{0\} \arrow{r}{}
    &\ul V_{r,h}^0 \arrow{r}{\ud_{r,h}^0}
    &\ul V_{r,h}^1 \arrow{r}{}
    &\cdots \arrow{r}{}
    &\ul V_{r,h}^{n-1} \arrow{r}{\ud_{r,h}^{n-1}}
    &\ul V_{r,h}^{n} \arrow{r}{}
    & \{0\}.
  \end{tikzcd}
\end{equation}

\subsubsection{Local discrete potentials and discrete exterior derivatives}

Given a form degree $k\in[0,n]$, for all $f\in\Delta_d(\Mh)$, $k\le d\le n$, we define the \emph{local discrete potential} $\Pot{r}{f}{k}:\ul V_r^k(f)\to\Poly{r+1}{-}\Lambda^k(f)$ by induction on $d$ as follows:
For all $\ul\omega_f\in\ul V_{r,f}^k$,
\begin{itemize}
\item If $d=k$, we simply set
  \begin{equation}\label{eq:vem:discrete.potential:d=k}
    \Pot{r}{f}{k}\ul\omega_f \coloneq \star^{-1}\omega_f
    \in\Poly{r}{}\Lambda^d(f) = \Poly{r+1}{-}\Lambda^d(f),
  \end{equation}
  where the last equality follows from \eqref{eq:trimmed.spaces:0-cells} if $d=0$ (after noticing that $\Poly{r}{}\Lambda^d(f)\cong\Real\cong\Poly{r+1}{-}\Lambda^d(f)$) and from \eqref{eq:trimmed.k=d} if $d\ge 1$;
\item If $k+1\le d\le n$, using the isomorphism \eqref{eq:isomorphism.Prtrimmed.koszul} with $\ell = d-k\ge 1$ and $r$ replaced by $r+1$, we define $\Pot{r}{f}{k}\ul\omega_f\in\Poly{r+1}{-}\Lambda^k(f)$ as the unique solution of the following equation:
  \begin{multline}\label{eq:vem:discrete.potential:d>=k+1}
    (-1)^{k+1}\int_f \Pot{r}{f}{k}\ul\omega_f\wedge(\rd\mu + \nu)
    \\
    = \int_f\star^{-1}\widetilde{D}_{\omega,f}\wedge\mu
    - \int_{\partial f} P_{r,\partial f}^k\ul\omega_{\partial f}\wedge\tr_{\partial f}\mu
    + (-1)^{k+1}\int_f\star^{-1}\omega_f\wedge\nu
    \\
    \forall (\mu,\nu)\in\Koly{r+1}{d-k-1}(f)\times\Koly{r+1}{d-k}(f),
  \end{multline}
  where
  \begin{equation}\label{eq:tilde.D.omega.f}
    \widetilde{D}_{\omega,f} \coloneq \begin{cases}
      \star\rd_{r,f}^k\ul\omega_f & \text{if $d=k+1$,}
      \\
      D_{\omega,f} & \text{if $d\ge k+2$,}
    \end{cases}
  \end{equation}
  and we have introduced the piecewise polynomial boundary potential $P_{r,\partial f}^k:\ul V_{r,\partial f}^k\to\Lambda^k(\partial f)$ such that $(P_{r,\partial f}^k)_{|f'}\coloneq P_{r,f'}^k$ for all $f'\in\Delta_{d-1}(f)$.
\end{itemize}
Leveraging the above-defined discrete potentials, we can define the \emph{discrete exterior derivative} $\rd_{r,f}^k:\ul V_{r,f}^k\to \Poly{r+1}{-}\Lambda^{k+1}(f)$ for all $f\in\Delta_d(\Mh)$, $k+2\le d\le n-1$, setting:
\begin{equation}\label{eq:discrete.exterior.derivative:d>=k+1:vem}
  \rd_{r,f}^k\ul\omega_f
  \coloneq P_{r,f}^{k+1}\ud_{r,f}^k\ul\omega_f
  \qquad\forall \ul\omega_f\in\ul V_{r,f}^k.
\end{equation}
These discrete exterior derivatives, which were previously only defined for $d=k+1$ (see \eqref{eq:local.exterior.derivative:vem}), are not relevant in the definition of the VEM complex, but may be useful in practical applications.

\subsection{Main properties of the VEM complex}

The main results for the VEM complex are stated below.

\begin{theorem}[Cohomology of the VEM complex]\label{thm:cohomology:vem}
  The VEM sequence \eqref{eq:vem} is a complex and its cohomology is isomorphic to the cohomology of the continuous de Rham complex \eqref{eq:diff.de.rham}.
\end{theorem}

\begin{proof}
  See Section \ref{sec:cohomology.VEM}.
\end{proof}

\begin{theorem}[Polynomial consistency of the discrete potential and exterior derivative] \label{thm:vem.poly.consistency}
  For all integers $0\le k\le d\le n$ and all $f\in\Delta_d(\Mh)$, it holds
  \begin{equation}\label{eq:discrete.potential:polynomial.consistency:vem}
    \Pot{r}{f}{k}\ul I_{r,f}^k\omega = \omega
    \qquad\forall\omega\in\Poly{r+1}{-}\Lambda^k(f),
  \end{equation}
  and, if $d\ge k+1$,
  \begin{equation}\label{eq:discrete.exterior.derivative:polycons:vem}
    \rd_{r,f}^k\ul I_{r,f}^k\omega = \rd\omega
    \qquad\forall\omega\in\Poly{r+1}{-}\Lambda^k(f).
  \end{equation}
\end{theorem}

\begin{proof}
  See Section \ref{sec:vem.poly.consistency}.
\end{proof}

\subsection{Complex property}

\begin{lemma}[Complex property]
  The sequence \eqref{eq:vem} defines a complex, i.e., for all integers $k\in[1,n-1]$ and all $\ul\omega_h\in\ul V_{r,h}^{k-1}$,
  \[
  \ud_{r,h}^k(\ud_{r,h}^{k-1}\ul\omega_h) = \ul 0.
  \]
\end{lemma}

\begin{proof}
  Applying the definition \eqref{eq:discrete.exterior.derivative:vem} of the global discrete exterior derivative for $k-1$, we obtain
  \begin{equation}\label{eq:complex:vem:1}
    \ud_{r,h}^{k-1}\ul\omega_h
    = \left(
    (\star\rd_{r,f}^{k-1}\ul\omega_f)_{f\in\Delta_k(\Mh)},
    (D_{\omega,f},0)_{f\in\Delta_d(\Mh),\,d\in[k+1,n]}
    \right)\in\ul V_{r,h}^k,
  \end{equation}
  which shows that, for all $d\in[k+1,n]$ and all $f\in\Delta_d(\Mh)$, the exterior derivative components of $\ud_{r,h}^{k-1}\ul\omega_h$ are zero, and thus that
  \[
  \ud_{r,h}^k(\ud_{r,h}^{k-1}\ul\omega_h)
  = \left(
  \big( \star\rd_{r,f}^k(\ud_{r,f}^{k-1}\ul\omega_f) \big)_{f\in\Delta_{k+1}(\Mh)},
  (0,0)_{f\in\Delta_d(\Mh),\,d\in[k+2,n]}
  \right)\in\ul V_{r,h}^{k+1}.
  \]
  The assertion is therefore proved if we show that $\rd_{r,f}^k(\ud_{r,f}^{k-1}\ul\omega_f) = 0$ for all $f\in\Delta_{k+1}(\Mh)$.
  Applying the definition of the local discrete exterior derivative (see \eqref{eq:local.exterior.derivative:vem}) with $\ul\omega_f$ replaced by $\ud_{r,f}^{k-1}\ul\omega_f$ obtained by restricting \eqref{eq:complex:vem:1} to $f$, we get:
  For all $(\mu,\nu)\in\Poly{0}{}\Lambda^0(f)\times\Koly{r}{0}(f)$,
  \[
  \int_f\rd_{r,f}^k(\ud_{r,f}^{k-1}\ul\omega_f)\wedge(\mu+\nu)
  = \int_{\partial f}\rd_{r,\partial f}^{k-1}\ul\omega_{\partial f}\wedge\tr_{\partial f}\mu
  = 0,
  \]
  where the conclusion follows using the definition \eqref{eq:local.exterior.derivative:vem} of $\rd_{r,f'}^{k-1}\ul\omega_{f'}$ with $(\mu,\nu)\gets(\tr_{f'}\mu,0)$ for all $f'\in\Delta_k(f)$ and noticing, as at the end of the proof of Lemma \ref{lem:link.diff.subcells}, that the sum over $f'$ of the integrals over $\partial f'$ is zero.
\end{proof}

\subsection{Commutation}

\begin{proposition}[Commutation property for the discrete exterior derivative in dimension $d = k+1$]
  For all $f\in\Delta_{k+1}(\Mh)$, it holds
  \begin{equation}\label{eq:discrete.exterior.derivative:commutativity:vem}
    \rd_{r,f}^k\ul I_{r,f}^k\omega
    = \star^{-1}\lproj{r}{f}{0}(\star\rd\omega)
    \qquad\forall\omega\in C^1\Lambda^k(\overline{f}),
  \end{equation}
  expressing the commutativity of the following diagram:%
  \[
  \begin{tikzcd}
    C^1\Lambda^k(\overline{f})\arrow{r}{\rd}\arrow{d}{\ul I_{r,f}^k}
    & C^0\Lambda^{k+1}(\overline{f})\arrow{d}{\star^{-1}\lproj{r}{f}{0}\star}
    \\
    \ul V_r^k(f)\arrow{r}{\rd_{r,f}^k}
    & \Poly{r}{}\Lambda^{k+1}(f).
  \end{tikzcd}
  \]
\end{proposition}

\begin{proof}
  Plugging the definition \eqref{eq:interpolator:vem} of the interpolator into \eqref{eq:local.exterior.derivative:vem} we get, for all $(\mu,\nu)\in\Poly{0}{}\Lambda^0(f)\times\Koly{r}{0}(f)$,
  \[
  \int_f\rd_{r,f}^{k}\ul I_{r,f}^k\omega\wedge(\mu+\nu)
  = \int_{\partial f}\star^{-1}\lproj{r}{\partial f}{0}(\star\tr_{\partial f}\omega)\wedge\tr_{\partial f}\mu
  + \int_f\star^{-1}\kproj{r}{f}{0}(\star\rd\omega)\wedge\nu,
  \]
  where $\lproj{r}{\partial f}{0}$ denotes the piecewise $L^2$-orthogonal projector obtained patching together the $\lproj{r}{f'}{0}$, $f'\in\Delta_k(f)$.
  Using \eqref{eq:remove.projector} with $(\mathcal X,\omega,\mu)\gets(\Koly{r}{0}(f'),\rd\omega,\nu)$
  for the second term and, for each $f'\in\Delta_k(f)$, $(\mathcal X,d,f)\gets(\Poly{r}{}\Lambda^0(f'),k,f')$ for the first term, the projectors can be removed.
  The Stokes formula \eqref{eq:ipp} along with $\rd\mu = 0$ (since $\mu$ has constant coefficients) then yields
  \begin{equation*}
    \int_f\rd_{r,f}^{k}\ul I_{r,f}^k\omega\wedge(\mu+\nu)
    = \int_f\rd\omega\wedge\mu +\int_f\rd\omega\wedge\nu
    =\int_f\star^{-1}\lproj{r}{f}{0}(\star\rd\omega)\wedge(\mu+\nu),
  \end{equation*}
  where the conclusion follows from \eqref{eq:remove.projector} with $(\mathcal X,\omega,\mu)\gets(\Poly{r}{}\Lambda^0(f),\rd\omega,\mu+\nu)$.
  Since, by \eqref{eq:decomposition.Pr:ell=0}, $\mu + \nu$ spans $\Poly{r}{}\Lambda^0(f)$ as $(\mu,\nu)$ spans $\Poly{0}{}\Lambda^0(f)\times\Koly{r}{0}(f)$, this concludes the proof.
\end{proof}

\begin{proposition}[Commutation property for the local discrete exterior derivative]
  For all integers $d\in [1,n]$ and $k\le d-1$, and all $f\in\Delta_d(\Mh)$, it holds
  \begin{equation}\label{eq:commut.d:vem}
    \ud_{r,f}^k(\ul I_{r,f}^k\omega)=\ul I_{r,f}^{k+1}(\rd\omega)\qquad\forall\omega\in C^2\Lambda^k(\overline{f}),
  \end{equation}
  expressing the commutativity of the following diagram:
  \[
  \begin{tikzcd}
    C^2\Lambda^k(\overline{f})\arrow{r}{\rd}\arrow{d}{\ul I_{r,f}^k}
    & C^1\Lambda^{k+1}(\overline{f})\arrow{d}{\ul I_{r,f}^{k+1}}
    \\
    \ul V_r^k(f)\arrow{r}{\ud_{r,f}^k}
    & \ul V_r^{k+1}(f).
  \end{tikzcd}
  \]
\end{proposition}

\begin{proof}
  Immediate consequence of \eqref{eq:discrete.exterior.derivative:commutativity:vem} along with the definition \eqref{eq:interpolator:vem} of the interpolator, and the property $\rd\circ \rd=0$.
\end{proof}

\subsection{Polynomial consistency} \label{sec:vem.poly.consistency}

\begin{proof}[Proof of Theorem \ref{thm:vem.poly.consistency}]
  The proof proceeds by induction on the dimension $d$.
  When $d=k$, \eqref{eq:discrete.potential:polynomial.consistency:vem} is a direct consequence of the definitions \eqref{eq:vem:discrete.potential:d=k} of the potential and \eqref{eq:interpolator:vem} of the interpolator, which give $\Pot{r}{f}{k}\ul I_{r,f}^k\omega = \star^{-1}\lproj{r}{f}{0}(\star\omega) = \omega$, where, to remove the projector, we have used the fact that $\star\omega\in\Poly{r}{}\Lambda^0(f)$, since $\omega\in\Poly{r+1}{-}\Lambda^d(f)=\Poly{r}{}\Lambda^d(f)$ (by \eqref{eq:trimmed.k=d} with $r+1$ instead of $r$).
  \smallskip

  We next prove \eqref{eq:discrete.potential:polynomial.consistency:vem} for $d\ge k+1$ assuming that it holds for $d-1$.
  Writing the definition \eqref{eq:vem:discrete.potential:d>=k+1} of the potential for $\ul\omega_f = \ul I_{r,f}^k\omega$, we get, for all $(\mu,\nu)\in\Koly{r+1}{d-k-1}(f)\times\Koly{r+1}{d-k}(f)$,
  \begin{equation}\label{eq:discrete.potential:polynomial.consistency:vem:intermediate}
    \begin{aligned}
      &(-1)^{k+1}\int_f \Pot{r}{f}{k}(\ul I_{r,f}^k\omega)\wedge(\rd\mu + \nu)
      \\
      &\quad
      = \int_f\star^{-1}\widetilde{D}_{\omega,f}\wedge\mu
      - \int_{\partial f}\Pot{r}{\partial f}{k}(\ul I_{r,\partial f}^k \tr_{\partial f}\omega)\wedge\tr_{\partial f}\mu
      \\
      &\qquad
      + (-1)^{k+1}\int_f\star^{-1}(\kproj{r+1}{f}{d-k}\star\omega)\wedge\nu
      \\
      &\quad
      = \int_f\star^{-1}\widetilde{D}_{\omega,f}\wedge\mu
      - \int_{\partial f}\tr_{\partial f}\omega\wedge\tr_{\partial f}\mu
      + (-1)^{k+1}\int_f\omega\wedge\nu,
    \end{aligned}
  \end{equation}
  where we have used the induction hypothesis for the second term in the right-hand side after noticing that, by Lemma \ref{lem:traces.trimmed} with $\ell = k$, $\tr_{f'}\omega\in\Poly{r+1}{-}\Lambda^k(f')$ for all $f'\in\Delta_{d-1}(f)$, together with \eqref{eq:remove.projector} for the third one.
  Recalling the definition \eqref{eq:tilde.D.omega.f} of $\widetilde{D}_{\omega,f}^k$, we distinguish two cases for the first term in the right-hand side.
  If $d = k+1$, \eqref{eq:discrete.exterior.derivative:polycons:vem} (immediate consequence of \eqref{eq:discrete.exterior.derivative:commutativity:vem} after observing that $d\Poly{r+1}{-}\Lambda^k(f)\subset\Poly{r}{}\Lambda^{k+1}(f)$) gives $\star^{-1}\widetilde{D}_{\omega,f} = \star^{-1}\star\rd_{r,f}^k(\ul I_{r,f}^k\omega) = \rd\omega$.
  If, on the other hand, $d \ge k+2$, recalling the definition \eqref{eq:interpolator:vem} of the interpolator, we have $\int_f\star^{-1}\widetilde{D}_{\omega,f}\wedge\mu = \int_f\star^{-1}(\kproj{r+1}{f}{d-k-1}\star\rd\omega)\wedge\mu \stackrel{\eqref{eq:remove.projector}}{=} \int_f\rd\omega\wedge\mu$.
  Plugging these relations into \eqref{eq:discrete.potential:polynomial.consistency:vem:intermediate}, using the Stokes formula \eqref{eq:ipp}, and simplifying, we get
  \[
  \int_f \Pot{r}{f}{k}(\ul I_{r,f}^k\omega)\wedge(\rd\mu + \nu)
  = \int_f\omega\wedge(\rd\mu + \nu),
  \]
  which yields \eqref{eq:discrete.potential:polynomial.consistency:vem} for $d\ge k+1$ since, by \eqref{eq:isomorphism.Prtrimmed.koszul} with $\ell = d - k \ge 1$, $\rd\mu + \nu$ spans $\Poly{r+1}{-}\Lambda^{d-k}(f)$ as $(\mu,\nu)$ spans $\Koly{r+1}{d-k-1}(f)\times\Koly{r+1}{d-k}(f)$.
  \smallskip

  We have already seen above that \eqref{eq:discrete.exterior.derivative:polycons:vem} holds for $d= k+1$. To prove this relation for $d\ge k+2$,
  it suffices to recall \eqref{eq:discrete.exterior.derivative:d>=k+1:vem} and \eqref{eq:commut.d:vem} to write
  \[
  \rd_{r,f}^k\ul I_{r,f}^k\omega
  = \Pot{r}{f}{k+1}(\ud_{r,f}^k\ul I_{r,f}^k\omega)
  = \Pot{r}{f}{k+1}(\ul I_{r,f}^{k+1}\rd\omega)
  = \rd\omega,
  \]
  where the conclusion follows from \eqref{eq:discrete.potential:polynomial.consistency:vem} after observing that $\rd\omega\in\Poly{r}{}\Lambda^{k+1}(f)\subset\Poly{r+1}{-}\Lambda^{k+1}(f)$.
\end{proof}

\subsection{Cohomology}\label{sec:cohomology.VEM}

As in Section \ref{sec:complex+cohomology:cohomology}, given a form degree $k\in[0,n]$, we first consider the following subspace of $\ul V_{r,h}^k$:
\begin{multline*}
  \ul V_{r,h,\flat}^k\coloneq\bigg\{
  \ul\omega_h = \big( (\omega_f)_{f\in\Delta_k(\Mh)},
  (\omega_f, D_{\omega,f})_{f\in\Delta_d(\Mh),\,d\in[k+1,n]}
  \big)\st
  \\
  \int_f\star^{-1}\omega_f = 0\quad\forall f\in\Delta_k(\Mh)
  \bigg\}.
\end{multline*}
\begin{lemma}[Exactness property for $\ul V_{r,h,\flat}^k$]\label{lem:exact.Vflat}
  For all $k\in[0,n]$, if $\ul\eta_h\in\ul V_{r,h,\flat}^k$ satisfies $\ud_{r,h}^k\ul\eta_h = \ul 0$, then there exists $\ul\omega_h\in\ul V_{r,h,\flat}^{k-1}$ such that $\ul\eta_h = \ud_{r,h}^{k-1}\ul\omega_h$,
  where, in accordance with the sequence \eqref{eq:vem}, we have set $\ud_{r,h}^{-1} = \ud_{r,h}^n \coloneq 0$.
\end{lemma}
\begin{proof}
  Recalling the definition \eqref{eq:discrete.exterior.derivative:vem} of $\ud_{r,h}^k\ul\eta_h$, we have
  \[
  \ud_{r,h}^k\ul\eta_h
  = \big(
  (\star\rd_{r,f}^k\ul\eta_f)_{f\in\Delta_{k+1}(\Mh)},
  (D_{\eta,f},0)_{f\in\Delta_d(\Mh),\,d\in[k+2,n]}
  \big).
  \]
  If $k=0$, then $\int_f\star^{-1}\eta_f = 0$ implies $\eta_f = 0$ for all $f\in\Delta_0(\Mh)$; moreover, $\eta_f = 0$ for all $f\in\Delta_d(\Mh)$, $d\in[1,n]$, by definition \eqref{eq:spaces:vem} of $\ul V_{r,h}^0$ (recall that $\Koly{r}{d}(f)=\{0\}$ for all $r$ and all $f \in \Delta_d(\Mh)$, cf.~\eqref{eq:Koly.0.ell=Koly.r.d=0}). The condition $\ud_{r,h}^k\ul\eta_h = \ul 0$ together with \eqref{eq:Domegaf=proj.drfk} yields $D_{\omega,f} = 0$ for all $f\in\Delta_d(\Mh)$, $d\ge k+1$, and thus
  \[
  \ul\eta_h = \big( (0)_{f\in\Delta_0(\Mh)}, (0, 0)_{f\in\Delta_d(\Mh),\,d\in[k+1,n]}\big)
  = \ud_{r,h}^{-1} 0.
  \]
  If $1\le k\le n-1$, on the other hand, from $\ud_{r,h}^k\ul\eta_h = \ul 0$ and \eqref{eq:Domegaf=proj.drfk} we infer
  \begin{subequations}\label{eq:expression.etah}
    \begin{equation}
      \ul\eta_h = \big(
      (\eta_f)_{f\in\Delta_k(\Mh)},
      (\eta_f,0)_{f\in\Delta_d(\Mh),\,d\in[k+1,n]}
      \big),
    \end{equation}
    while, if $k=n$, we simply have
    \begin{equation}
      \ul\eta_h = (\eta_f)_{f\in\Delta_n(\Mh)}.
    \end{equation}
  \end{subequations}
  Let now
  \[
  \ul\omega_h = \big(
  (0)_{f\in\Delta_{k-1}(\Mh)},
  (0,\kproj{r}{f}{0}\eta_f)_{f\in\Delta_k(\Mh)},
  (0,\eta_f)_{f\in\Delta_d(\Mh),\,d\in[k+1,n]}
  \big)\in\ul V_{r,h,\flat}^{k-1}.
  \]
  To check that this $\ul\omega_h$ is well defined, it suffices to notice that, if $f\in\Delta_d(\Mh)$ with $d\ge k+1=(k-1)+2$, then $\eta_f\in\Koly{r+1}{d-k}(f)=\Koly{r+1}{d-(k-1)-1}(f)$ is a suitable choice for the corresponding component of $\ul\omega_h$.
  By definition \eqref{eq:local.exterior.derivative:vem} of $\rd_{r,f}^{k-1}$, we have:
  For all $f\in\Delta_k(\Mh)$ and all $(\mu,\nu)\in\Poly{0}{}\Lambda^0(f)\times\Koly{r}{0}(f)$, since $\omega_{f'}=0$ for all $f'\in\Delta_k(f)$,
  \[
  \int_f\rd_{r,f}^{k-1}\ul\omega_f\wedge(\mu + \nu)
  = \int_f\star^{-1}\cancel{\kproj{r}{f}{0}}\eta_f\wedge\nu
  = \int_f\star^{-1}\eta_f\wedge(\mu + \nu),
  \]
  where the cancellation of $\kproj{r}{f}{0}$ is made possible by \eqref{eq:remove.projector} with $(\mathcal X,\omega,\mu)\gets(\Koly{r}{0}(f),\star^{-1}\eta_f,\nu)$, while the introduction of $\mu$ in the last passage is justified observing that $\ul\eta_h\in\ul V_{r,h,\flat}^k$ implies $\int_f\star^{-1}\eta_f = 0$ for all $f\in\Delta_k(\Mh)$.
  This relation gives $\rd_{r,f}^{k-1}\ul\omega_f = \star^{-1}\eta_f$ for all $f\in\Delta_k(\Mh)$ which, combined with the definition \eqref{eq:discrete.exterior.derivative:vem} of the global discrete exterior derivative and the expression \eqref{eq:expression.etah} of $\ul\eta_h$, readily yields $\ul\eta_h = \ud_{r,h}^{k-1}\ul\omega_h$ and concludes the proof.
\end{proof}

\begin{proof}[Proof of Theorem \ref{thm:cohomology:vem}]
  Contrary to the \DDR{0} complex, the \VEM{0} complex is not isomorphic to the CW complex (the VEM spaces for $r=0$ do not have only constant polynomial components on the lowest-dimensional cells).
  As a consequence, designing extensions and reductions between the \VEM{r} and \VEM{0}  complexes in the spirit of Theorem \ref{thm:cohomology:ddr} would not directly characterise the cohomology of the VEM complex.
  To circumvent this difficulty, we will instead design extensions $\Ext{h}{k}:\ul X_{0,h}^k\to\ul V_{r,h}^k$ and reductions $\Red{h}{k}:\ul V_{r,h}^k\to\ul X_{0,h}^k$ between the \VEM{r}, $r\ge 0$, and the \DDR{0} complexes, in order to show that their cohomologies are isomorphic.
  By Theorem \ref{thm:cohomology:ddr}, this will prove that the cohomology of \VEM{r} is isomorphic to the continuous de Rham cohomology.

  Throughout the rest of this proof, $(\Pot{0}{f}{k},\rd^k_{0,f})$ and $(\Pot{r}{f}{k},\rd^k_{r,f})$ denote, respectively, the couple (potential reconstruction, discrete exterior derivative) of the \DDR{0} and \VEM{r} complexes.
  We do not need to differentiate these notations, as the argument removes all ambiguity.
  For all form degrees $k\in[0,d]$, the reduction is obtained setting
  \begin{equation}\label{eq:Red:vem}
    \Red{h}{k}\ul\omega_h\coloneq\big(
    (\lproj{0}{f}{0}\omega_f)_{f\in\Delta_k(\Mh)}
    \big)\qquad\forall\ul\omega_h\in\ul V_{r,h}^k,
  \end{equation}
  while the extension is given by
  \begin{equation}\label{eq:Ext:vem}
    \Ext{h}{k}\ul\eta_h
    \coloneq
    \begin{aligned}[t]
      \Big(
      &(\eta_f)_{f\in\Delta_k(\Mh)},
      \\
      &
      \big(
      \kproj{r+1}{f}{1}(\star\Pot{0}{f}{k}\ul\eta_f),
      \kproj{r}{f}{0}(\star\rd^k_{0,f}\ul\eta_f)
      \big)_{f\in\Delta_{k+1}(\Mh)},
      \\
      &\big(
      \kproj{r+1}{f}{d-k}(\star\Pot{0}{f}{k}\ul\eta_f),
      \kproj{r+1}{f}{d-k-1}(\star\rd^k_{0,f}\ul\eta_f)\big)_{f\in\Delta_d(\Mh),\,d\ge k+2}
      \Big)\qquad\forall \ul\eta_h\in \ul X_{0,h}^k.
    \end{aligned}
  \end{equation}
  As in the proof of Theorem~\ref{thm:cohomology:ddr}, we need to establish the properties (C1)--(C3) of \cite[Assumption 1]{Di-Pietro.Droniou:23*2} to obtain the desired isomorphism in cohomology (also in this case, the relation \eqref{eq:RdE} is an immediate consequence of (C1) and (C3)).
  \medskip\\
  \underline{\emph{Proof of (C1).}}
  An inspection of the definitions \eqref{eq:Red:vem} of the reduction and \eqref{eq:Ext:vem} of the extension shows that $\Red{h}{k}\Ext{h}{k}\ul\eta_h = \ul\eta_h$ for all $\ul\eta_h\in\ul X_{0,h}^k$, and thus (C1) holds a fortiori.
  \medskip\\
  \underline{\emph{Proof of (C3).}}
  We need to prove that both the reduction and extension are cochain maps.
  \smallskip

  Let us start with the extension. We have to prove that, for any integer $k\in [0,n-1]$ and all $\ul\eta_h\in \ul X_{r,h}^k$, $\Ext{h}{k+1}(\ud^k_{0,h}\ul\eta_h)=\ud^k_{r,h}(\Ext{h}{k}\ul\eta_h)$. Owing to the definitions \eqref{eq:Ext:vem} of the extension, \eqref{eq:global.discrete.exterior.derivative} of $\ud^k_{0,f}$, and \eqref{eq:discrete.exterior.derivative:vem} of $\ud^k_{r,f}$, and since $\rd^{k+1}_{0,f}\circ\ud^k_{0,f}=0$ (by \eqref{eq:complex.prop} with $r=0$ and $k+1$ instead of $k$) this amounts to proving that
  \begin{alignat}{4}
    \label{eq:Ext.vem.cochain.k+1}
    \star\rd^k_{0,f}\ul\eta_f ={}&  \star\rd^k_{r,f}(\Ext{f}{k}\ul\eta_f)&\qquad& \forall f\in\Delta_{k+1}(\Mh),
    \\
    \label{eq:Ext.vem.cochain.>=k+2}
    \kproj{r+1}{f}{d-k-1}(\star\Pot{0}{f}{k+1}\ud^k_{0,f}\ul\eta_f)={}&\kproj{r+1}{f}{d-k-1}(\star\rd^k_{0,f}\ul\eta_f)
    &\qquad&\forall f\in\Delta_d(\Mh),\,d\ge k+2.
  \end{alignat}
  The relation \eqref{eq:Ext.vem.cochain.>=k+2} trivially follows from $\Pot{0}{f}{k+1}\ud^k_{0,f}=\rd^k_{0,f}$, which comes from \eqref{eq:link.pot.diff} with $(k,r)\gets(k+1,0)$. To prove \eqref{eq:Ext.vem.cochain.k+1}, we take $(\mu,\nu)\in \Poly{0}{}\Lambda^0(f)\times \Koly{r}{0}(f)$ and apply the definitions \eqref{eq:local.exterior.derivative:vem} of $\rd^k_{r,f}(\Ext{f}{k}\ul\eta_f)$ and \eqref{eq:Ext:vem} of $\Ext{k}{h}\ul\eta_h$ to get
  \begin{align*}
    \int_f \rd^k_{r,f}(\Ext{f}{k}\ul\eta_f)\wedge(\mu + \nu)
    &=
    \int_{\partial f}\star^{-1} \eta_{\partial f}\wedge\mu
    + \int_f \star^{-1}\kproj{r}{f}{0}(\star\rd^k_{0,f}\ul\eta_f)\wedge\nu\\
    &=
    \int_f \rd^k_{0,f}\ul\eta_f\wedge\mu
    + \int_f \rd^k_{0,f}\ul\eta_f\wedge\nu
    = \int_f \rd^k_{0,f}\ul\eta_f\wedge(\mu + \nu),
  \end{align*}
  where the second equality is obtained using the definition \eqref{eq:discrete.exterior.derivative:ddr} of $\rd^k_{0,f}$ for the first term and \eqref{eq:remove.projector} with $(\mathcal X,\omega,\mu)\gets(\Koly{r}{0}(f),\rd^k_{0,f}\ul\eta_f,\nu)$ for the second one.
  By \eqref{eq:decomposition.Pr:ell=0}, and since both $\rd^k_{r,f}(\Ext{f}{k}\ul\eta_f)$ and $\rd^k_{0,f}\ul\eta_f$ belong to $\Poly{r}{}\Lambda^{k+1}(f)$, this relation gives $\rd^k_{r,f}(\Ext{f}{k}\ul\eta_f)=\rd^k_{0,f}\ul\eta_f$, thus proving \eqref{eq:Ext.vem.cochain.k+1}.
  \smallskip

  Let us now turn to the reduction. We need to show that, for any integer $k\in [0,n-1]$ and all $\ul\omega_h\in\ul V_{r,h}^k$, $\Red{h}{k+1}(\ud_{r,h}^k\ul\omega_h) = \ud_{0,h}^k(\Red{h}{k}\ul\omega_h)$, i.e., accounting for the definitions \eqref{eq:Red:vem} of the reduction, \eqref{eq:global.discrete.exterior.derivative} of $\ud_{0,h}^k$ (additionally noticing that $\trimproj{0}{f}{0}$ coincides with $\lproj{0}{f}{0}$ owing to \eqref{eq:trimmed.spaces:ell=0}), and \eqref{eq:discrete.exterior.derivative:vem} of $\ud_{r,h}^k$,
  \begin{equation}\label{eq:Red.vem.cochain}
    \lproj{0}{f}{0}(\star\rd_{r,f}^k\ul\omega_f)
    = \star\rd_{0,f}^k\Red{f}{k}\ul\omega_f
    \qquad\forall f\in\Delta_{k+1}(\Mh).
  \end{equation}
  To check this relation, let $f\in\Delta_{k+1}(\Mh)$ and write, for all $\mu\in\Poly{0}{}\Lambda^0(f)$,
  \[
  \begin{alignedat}{4}
    \int_f\star^{-1}\lproj{0}{f}{0}(\star\rd_{r,f}^k\ul\omega_f)\wedge\mu
    &=
    \int_f\rd_{r,f}^k\ul\omega_f\wedge\mu
    &\qquad&\text{Eq.~\eqref{eq:remove.projector}}
    \\
    &=
    \int_{\partial f}\star^{-1}\omega_{\partial f}\wedge\tr_{\partial f}\mu
    &\qquad&\text{Eq.~\eqref{eq:local.exterior.derivative:vem}}
    \\
    &=
    \int_{\partial f}\star^{-1}\lproj{0}{\partial f}{0}\omega_{\partial f}\wedge\tr_{\partial f}\mu
    \\
    &=
    \int_f\rd_{0,f}^k\Red{f}{k}\ul\omega_f\wedge\mu,
    &\qquad&\text{Eqs.~\eqref{eq:Red:vem}, \eqref{eq:discrete.exterior.derivative:ddr}, \eqref{eq:discrete.potential:d=k}}
  \end{alignedat}
  \]
  where the third equality follows from \eqref{eq:remove.projector} with $(\mathcal X,\omega,\mu)\gets(\Poly{0}{}\Lambda^0(\partial f),\star^{-1}\omega_{\partial f},\tr_{\partial f}\mu)$.
  This proves \eqref{eq:Red.vem.cochain}, and thus that the reductions form a cochain map.
  \medskip\\
  \underline{\emph{Proof of (C2).}}
  For all $\ul\omega_h\in\ul V_{r,h}^k$, by the definitions \eqref{eq:Red:vem} and \eqref{eq:Ext:vem} of the reduction and extension, it holds $\Ext{h}{k}\Red{h}{k}\ul\omega_h - \ul\omega_h\in\ul V_{r,h,\flat}^k$.
  The proof then continues as in point 4.~of the proof of Theorem~\ref{thm:cohomology:ddr} (see Section~\ref{sec:complex+cohomology:cohomology}) with $\ul X_{r,h,\flat}^k$ replaced by $\ul V_{r,h,\flat}^k$ and Lemma~\ref{lem:exact.Xsharp} replaced by Lemma~\ref{lem:exact.Vflat}.
\end{proof}


\section{Related works}\label{sec:links}

We provide here some elements of comparison between the DDR and VEM constructions of Sections \ref{sec:ddr} and \ref{sec:vem}, and two other families of discrete complexes.

\subsection{Finite Element Exterior Calculus and Finite Element Systems}

Finite Element Exterior Calculus (FEEC) is the (conforming) finite element approach for the unified analysis of discrete complexes \cite{Arnold.Falk.ea:06,Arnold:18}. It is based on the selection of piecewise polynomial subspaces of $H\Lambda(\Omega)$ that form a subcomplex of the continuous complex \eqref{eq:diff.de.rham}. Finite Element Systems (FES) is a framework for designing subcomplexes that generalises FEEC to cover finite dimensional spaces spanned by differential forms that may not be piecewise polynomial on the selected mesh \cite{Christiansen.Munthe-Kaas.ea:11,Christiansen.Hu:18}. In FEEC complexes, only the spaces of differential forms in the continuous complex are replaced with discrete counterparts: the graded map that links these spaces is the usual exterior derivative $\rd$.
Generalised FES is an abstract setting which also gives freedom on the definition of the graded maps in the complex \cite{Christiansen.Hu:18}.

A (generalised) FES space is a space of $k$-forms on all $d$-cells with $d\ge k$, with a compatibility condition on the traces:
\begin{multline}\label{def:fes}
  A^k(\Mh)=\Bigg\{
  \underline{v}_h\in\!\!\bigtimes_{\stackrel{\text{$f\in\Delta_d(\Mh)$}}{\text{\tiny $d\in[k,n]$}}}\!\! A^k(f)\st
  \\
  \text{$\tr_{f'}v_f=v_{f'}$ for all $(f,f')\in\Delta_d(\Mh)\times\Delta_{d'}(f)$ with $k\le d'\le d$}
  \Bigg\},
\end{multline}
where each $A^k(f)$ is a finite-dimensional space of $k$-forms and $v_f$ denotes the component of $\underline{v}_h$ on $f$.
In the original FES setting (in which the graded maps are $\rd$), each element of $A^k(\Mh)$ can be identified with an element of $H\Lambda^k(\Omega)$.
This setting contains the usual FEEC complexes (in which case $A^k(f)$ are certain polynomial subspaces -- typically full polynomial spaces or trimmed polynomial spaces depending on the considered finite element), but has also been used to develop other discrete complexes, e.g.~based on macro-elements (in which case $A^k(f)$ is a space of piecewise polynomial forms on a subdivision of $f$) or with higher inter-element regularity ($C^1$ spaces, for example).

The concept of (faithful) \emph{mirror system} plays the role of degrees of freedom in the FES framework.
Mirror systems are constructed on a case-by-case basis for each FES, and are auxiliary tools in the framework: they are not required to design the FES spaces (or maps), but they identify (by duality) a basis of such spaces.
A mirror system for $A^k(\Mh)$ is a family of subspaces of linear forms:
\begin{equation}\label{eq:mirror.system.1}
  Z^k(\Mh)=\bigtimes_{\stackrel{\text{$f\in\Delta_d(\Mh)$}}{\text{\tiny $d\in[k,n]$}}}Z^k(f)\qquad\text{ with $Z^k(f)\subset A^k(f)^*$ for all $f\in\Delta(\Mh)$},
\end{equation}
where $A^k(f)^*$ is the dual space of $A^k(f)$ (actually, to link mirror systems and interpolators, each $Z^k(f)$ is chosen as a subspace of $\hat{X}^k(f)^*$ with $\hat{X}^k(f)\supset A^k(f)$, but we won't need this in the discussion here). As can be seen in \eqref{eq:mirror.system.1}, a mirror system is built hierarchically on the mesh, and each $Z^k(f)$ identifies the modes of the FES forms that are ``interior'' to $f$; to obtain all the modes (interior and boundary) associated with $f$, one must consider $\bigtimes_{f'\in\Delta_{d'}(f),\,d'\in[k,d]}Z^k(f')$.

A particular case of interest in the present context is when $Z^k(f)\subset L^2\Lambda^k(f)^*$ (see Remark \ref{rem:riesz.d.pot}). Using the Riesz representation theorem and applying the Hodge star transformation, $Z^k(f)$ can then be identified with a family of subspaces of $L^2$-integrable $(d-k)$-forms:
\begin{equation}\label{eq:mirror.system.2}
  Z^k(\Mh) \cong \bigtimes_{\stackrel{\text{$f\in\Delta_d(\Mh)$}}{\text{\tiny $d\in[k,n]$}}}\widetilde{Z}^{d-k}(f)\quad\text{ with $\widetilde{Z}^{d-k}(f)\subset L^2\Lambda^{d-k}(f)$}.
\end{equation}
Here, and contrary to \eqref{def:fes}, no compatibility condition of the traces is imposed: the spaces $\widetilde{Z}^{d-k}(f)$ are completely disconnected from each other.

\medskip

The FEEC framework provides a setting for the \emph{algebraic} and \emph{analytical} study of discrete complexes based on piecewise polynomial subspaces of the continuous spaces; the constraint of having piecewise polynomial subspaces and of imposing the suitable inter-element continuities restricts the design of finite element methods to certain types of meshes -- mostly tetrahedral and hexahedral. The FES framework is more general in the sense that it does not, in principle, require to identify conforming subspaces (or accepts conforming subspaces that are not piecewise polynomial on the chosen mesh -- these spaces are then usually not explicitly known). Its main restriction, compared to FEEC, is that it only provides \emph{algebraic} results on the discrete complexes, not analytical results such as Poincar\'e inequalities, or primal and adjoint consistencies -- all critical for the numerical analysis of numerical schemes based on the complex. Moreover, as far as we could see in the literature, all complexes based on the FES framework and fully computable (that is, the spaces and operators are explicitly known) seem to rely on the design of discrete subcomplexes of the continuous complex, which imposes restrictions on the types of meshes that can be considered (as in FEEC, mostly tetrahedral and hexahedral meshes, with the graded map being the exterior derivative).
On the contrary, the DDR and VEM constructions of Sections \ref{sec:ddr} and \ref{sec:vem} provide explicit and computable discrete complexes on generic polytopal meshes, that do not rely on finding computable conforming subspaces of the de Rham spaces.
These polytopal methods can be entirely built using spaces of polynomial functions on the mesh, without any compatibility condition on the traces. The spaces are explicit, their bases are directly given by the polynomial components, and the graded map (acting as a discrete exterior derivative) is explicitly expressed in terms of these components.

Comparing \eqref{eq:global.space} and \eqref{eq:mirror.system.2} for example, we see that the DDR space plays the role of a mirror system, and puts discrete polynomial components at the center of the construction. A similar approach is also true for the VEM-inspired spaces \eqref{eq:spaces:vem}, with, contrary to DDR, some polynomial components representing exterior derivatives; see the definition \eqref{eq:interpolator:vem} of the interpolator.

A closer link between DDR and FES can be drawn by noticing that the FES \cite[Section~2.1]{Christiansen.Gillette:16} has the DDR spaces as mirror system (in the sense of \eqref{eq:mirror.system.2}).
The spaces of this FES are based on liftings of harmonic functions on each cell, which cannot be explicitly described in general, and therefore cannot be directly used, say, in a weak formulation of a PDE to design a numerical scheme. This is in contrast with the fully discrete approach employed by the DDR technique, which not only identifies explicit discrete spaces and exterior derivatives, but also consistent $L^2$-inner products on these spaces, therefore providing all the tools required to build numerical schemes, see, e.g. \cite{Di-Pietro.Droniou:21,Beirao-da-Veiga.Dassi.ea:22,Hanot:23}).
We also notice, in passing, that, in the context of vector proxies, different conforming (non-explicit) spaces having the DDR components as degrees of freedom were also identified in \cite[Section 6.2]{Beirao-da-Veiga.Dassi.ea:22}.

Finally, it can be shown, using the results of Section \ref{sec:ddr.complex} and Lemma \ref{lem:exact.Xsharp}, that the DDR complex fits into the generalised FES framework. \cite[Theorem 1]{Christiansen.Hu:18} then provides an alternative study of the cohomology of the DDR complex.
The approach developed in the proof of Theorem \ref{thm:cohomology:ddr} provides a practical way to compute the cohomology spaces of the DDR complex based on those of the underlying CW complex (see \cite[Remark 13]{Di-Pietro.Droniou.ea:23} for details), for which efficient algorithms are available \cite{Dlotko.Specogna:13}.
Moreover, as mentioned above, the FES framework does not cover any analytical properties of the discrete complexes. In particular, for DDR, it only relies on the global discrete exterior derivative $\ud_{r,h}^k$ defined in \eqref{eq:global.discrete.exterior.derivative}, and would not identify or make use of the local potential reconstructions $\Pot{r}{f}{k}$ and discrete exterior derivatives $\rd_{r,f}^k$ which encode the optimal consistency properties of the method (see Theorem \ref{thm:polynomial.consistency:ddr} and \eqref{eq:consistency:L2product}).

\subsection{Distributional Differential Forms}

The theory of Distributional Differential Forms (DDF) has been introduced in \cite{Licht:17} as a generalisation of the construction in  \cite{Braess.Schoberl:08} for the a posteriori error analysis of N\'ed\'elec edge elements. DDF are built on triangulations of the domain and, using their relation with the underlying simplicial complexes (as well as the concept of double complexes), their cohomology was analysed in \cite{Licht:17} for rather general boundary conditions.
Poincar\'e--Friedrichs inequalities were later established in \cite{Christiansen.Licht:20}.

As is the case for the spaces appearing in the DDR and VEM complexes, DDF spaces are collections of differential forms on cells of various dimensions, with form degree depending on the dimension of the cell: if the domain $\Omega$ has dimension $n$, the DDF space of degree $k$ is made of $(k-n+d)$-forms on $d$-cells. No compatibility of the traces is enforced on these forms, which can be completely discontinuous between two $d$-simplices.
The discrete distributional exterior derivative on the DDF space is then composed of two contributions: the exterior derivative inside the simplices, and a trace term. For example, focusing on the highest dimension $d=n$, if the DDF space of $k$-forms is
\begin{equation}\label{eq:disc.diff.form}
  \hat{\Lambda}^k_{-2}(\Delta_n(\Mh))=\hat{\Lambda}^k_{-1}(\Delta_n(\Mh))\oplus \hat{\Lambda}^{k-1}_{-1}(\Delta_{n-1}(\Mh)),
\end{equation}
(with $\hat{\Lambda}^\ell_{-1}$ subspace of piecewise $C^\infty\Lambda^\ell$ forms, the index $-1$ expressing the absence of continuity properties at the interfaces), for a family $\omega_{n,h}=(\omega_f)_{f\in\Delta_n(\Mh)}\in\hat{\Lambda}^k_{-1}(\Delta_n(\Mh))$, we define the distributional derivative $\hat{\rd}^k_h:\hat{\Lambda}^k_{-1}(\Delta_n(\Mh))\to \hat{\Lambda}^{k+1}_{-2}(\Delta_n(\Mh))$ by
\begin{equation}\label{eq:ddf}
  \hat{\rd}^k_h\omega_{n,h}=\left((\rd^k \omega_f)_{f\in\Delta_n(\Mh)},\left(-\sum_{f\in\Sigma_{n}(f')}\varepsilon_{f\!f’}\tr_{f'}\omega_f\right)_{f'\in\Delta_{n-1}(\Mh)}\right),
\end{equation}
where $\Sigma_n(f')$ is the set of $n$-simplices $f$ that share $f'$ (that is, $f'\in\Delta_{n-1}(f)$), and $\varepsilon_{f\!f'}$ is the relative orientation of the simplex $f'$ with respect to the simplex $f$.
Note that, in \eqref{eq:ddf}, we have adopted a presentation of the distributional derivative that distributes its two contributions ($\mathrm{D}$ and $\mathrm{T}$ in \cite{Licht:17}) on the corresponding components $(\hat{\Lambda}^{k+1-i}_{-1}(\Delta_{n-i}(\Mh)))_{i=0,1}$ of $\hat{\Lambda}^{k+1}_{-2}(\Delta_n(\Mh))$ (see \eqref{eq:disc.diff.form} with $k+1$ instead of $k$), instead of writing $\hat{\rd}^k_h$ as a sum of elements in the global space $\hat{\Lambda}^{k+1}_{-2}(\Delta_n(\Mh))$; this is to better compare with the definition \eqref{eq:global.discrete.exterior.derivative}.
This definition of distributional derivative is a global one, obtained by testing the piecewise smooth form $\omega_{n,h}$ against globally smooth forms, which classically results in a term inside each $f\in\Delta_n(\Mh)$ corresponding to the standard exterior derivative (first component in \eqref{eq:ddf}), and a jump across the $(n-1)$-sub-simplices based on the difference of traces on the two adjacent $n$-simplices (second component in \eqref{eq:ddf}).

A crucial remark is that, in \eqref{eq:ddf}, the component $(\rd^k \omega_f)_{f\in\Delta_n(\Mh)}$ of $\hat{d}_h^k\omega_{n,h}$ on $n$-cells only depends on the values $\omega_{n,h}$ of the discrete distributional differential form on $n$-cells, not on the values of these forms on lower-dimensional cells (e.g., $\hat{\Lambda}^{k-1}_{-1}(\Delta_{n-1}(\Mh))$ in \eqref{eq:disc.diff.form}).
This is in contrast with the discrete exterior derivatives in DDR and VEM complexes, whose definition on higher-dimensional cells depends on polynomial components on their sub-cells; see \eqref{eq:discrete.exterior.derivative:ddr} and \eqref{eq:discrete.exterior.derivative:vem}. Another difference between DDR and DDF can be seen when recasting the discrete exterior derivative: integrating by parts \eqref{eq:discrete.exterior.derivative:ddr} yields the following characterisation:
\begin{multline*}
  \int_f \rd_{r,f}^{k} \ul\omega_f \wedge \mu
  = - \int_f \rd(\star^{-1}\omega_f)\wedge  \mu
  + \int_{\partial f} (\Pot{r}{\partial f}{k}\ul\omega_{\partial f}-\tr_{\partial f}(\star^{-1}\omega_{f})) \wedge \tr_{\partial f}{\mu}
  \\
  \forall \mu \in \Poly{r}{}\Lambda^{d-k-1}(f).
\end{multline*}
This relation reveals that $\rd_{r,f}^k\ul\omega_f$ is, as in DDF, composed of an exterior derivative term in the $d$-cell and a boundary term involving jumps. However, contrary to DDF, the jumps here are between the trace of the $d$-cell unknown and the potential $\Pot{r}{\partial f}{k}\ul\omega_{\partial f}$ reconstructed on $(d-1)$-cells (which depends on the unknowns on all $d'$-subcells of $f$, $k\le d'\le d$), not between traces of two $d$-cells unknowns  (as in \eqref{eq:ddf} with $d=n$).
In this respect, the ``jump'' term in DDR relates more to the kind of face differences encountered in polytopal methods (e.g., the HHO method \cite{Di-Pietro.Droniou:20}) while the jump term in DDF is more akin to those arising in discontinuous Galerkin (DG) methods \cite{Di-Pietro.Ern:12}.

This comparison can be extended to the potential reconstructions themselves. Equation \eqref{def:Pot.correction} shows that $\Pot{r}{f}{k}\ul\omega_f$ is obtained applying a higher-order enhancement to the cell component $\star^{-1}\omega_f$, designed from the discrete exterior derivative on $f$ and the potentials on $\partial f$.
This enhancement ensures the high-order consistency of the method starting from lower-order polynomial unknowns. In the context of elliptic equations, it is commonly used in methods with unknowns in the elements and on the faces of the mesh, but it is not directly available in DG methods. In DDF, as in DG, the cell unknown itself must be used (e.g., in a scheme to discretise the source term), and the consistency is therefore limited by the degree of this unknown.


\appendix

\section{Differential forms and vector proxies}\label{sec:appendix}
In this section, we briefly recall basic concepts on alternating (resp.~differential) forms, and their representation in terms of vectors (resp.~vector fields); these representations are often referred to as ``vector proxies''.
We refer the reader to~\cite[Chapter~6]{Arnold:18} for a presentation in the framework of Finite Element Exterior Calculus, and to~\cite{Blair-Perot.Zusi:14},~\cite[Chapter~1]{Carmo:94},~\cite[Chapter~1]{Hanot:22} for an introduction in more general scientific and engineering contexts.

\subsection{Exterior algebra in $\Real^n$}
\label{sec:exterior.algebra}

\subsubsection{Alternating forms}
\label{sec:alternating.forms}
Let $\{ \bvec e_i\}_{i\in[1,n]}$ be the canonical basis of $\Real^n$, equipped with the standard inner product.
A~basis for the space of linear forms over $\Real^n$, i.e., the dual space $(\Real^n)'$ of $\Real^n$, is given by $\{\rd x^i\}_{i\in[1,n]}$,
with $\rd x^i(\bvec e_j) \coloneq \delta^i_j$ (Kr\"onecker symbol), for all $(i,j) \in [1,n]^2$. The starting point of exterior calculus is to consider \emph{alternating} multilinear forms, vanishing
whenever they are applied to a set of linearly dependent vectors in $\Real^n$. For any integer $k\ge 1$, the set of alternating $k$-linear forms on $\Real^n$ is denoted by $\Alt{k}(\Real^n)$; by convention, we set $\Alt{0}(\Real^n) \coloneq \Real$.
We also note that $\Alt{1}(\Real^n) = (\Real^n)'$ and that $\Alt{k}(\Real^n)=\{0\}$ if $k>n$
(since families of $k>n$ vectors are always linearly dependent).
It can be checked that $\dim\Alt{k}(\Real^n) = {n\choose k}$. In particular, $\Alt{n}(\Real^n)$ is the 1-dimensional space spanned by the determinant in the canonical basis $\vol$ (called the volume form).

\subsubsection{Exterior product}

Given two alternating multilinear forms $\omega\in \Alt{i}(\Real^n)$ and $\mu\in \Alt{j}(\Real^n)$, their \emph{exterior product} $\omega\wedge\mu \in \Alt{i+j}(\Real^n)$ is defined, for any vectors $\bvec v_1,\ldots,\bvec v_{i+j}\in\Real^n$, by
$$
(\omega\wedge\mu)(\bvec v_1,\ldots,\bvec v_{i+j}) \coloneq \sum_{\sigma\in\Sigma_{i,j}} \sign(\sigma)\,\omega(\bvec v_{\sigma_1},\ldots,\bvec v_{\sigma_i})\,
\mu(\bvec v_{\sigma_{i+1}},\ldots,\bvec v_{\sigma_{i+j}}),
$$
where $\Sigma_{i,j}$ is the set of all permutations $\sigma$ of the $(i+j)$-tuple $(1,\ldots,i+j)$ such that $\sigma_1 < \cdots < \sigma_i$ and $\sigma_{i+1} < \cdots < \sigma_{i+j}$. The exterior product satisfies the anticommutativity law
\begin{equation}\label{eq:wedge:anticommutativity}
  \omega\wedge\mu = (-1)^{ij}\mu\wedge\omega,
\end{equation}
so that, in particular, we have
$\rd x^i \wedge \rd x^i = 0$ and $\rd x^i \wedge \rd x^j = - \rd x^j \wedge \rd x^i$.
With these definitions, for $k\in [1,n]$ a basis of the space $\Alt{k}(\Real^n)$ is $\{\rd x^{\sigma_1} \wedge \cdots \wedge \rd x^{\sigma_k}\}_\sigma$ where $\sigma$ spans all strictly increasing functions $[1,k]\to[1,n]$. Hence, any $\omega\in\Alt{k}(\Real^n)$ can be written
\begin{equation}
  \label{eq:general.alternating.form}
  \omega = \sum_{1\le \sigma_1 < \cdots < \sigma_k \le n} a_\sigma \ \rd x^{\sigma_1} \wedge \cdots \wedge \rd x^{\sigma_k},\quad a_\sigma\in\Real.
\end{equation}

\subsubsection{Hodge star operator}

The scalar product in $\Real^n$ induces a scalar product, denoted by $\langle{\cdot,\cdot}\rangle$, on $\Alt{n-k}(\Real^n)$ -- namely, the scalar product for which the aforementioned basis $\{\rd x^{\sigma_1}\wedge\cdots\wedge\rd x^{\sigma_{n-k}}\}_\sigma$ of $\Alt{n-k}(\Real^n)$ is orthonormal. The
\emph{Hodge star operator} is the unique linear mapping $\star:\Alt{k}(\Real^n)\to\Alt{n-k}(\Real^n)$ such that, for all $\omega\in\Alt{k}(\Real^n)$,
$\langle\star\omega,\mu\rangle\vol = \omega\wedge\mu $ for all $\mu\in\Alt{n-k}(\Real^n)$.
It can be checked that
$$
\star(\rd x^{\sigma_1}\wedge\cdots\wedge \rd x^{\sigma_k})
= \sign(\sigma,\tau) (\rd x^{\tau_1} \wedge \cdots \wedge \rd x^{\tau_{n-k}}),
$$
where $(\sigma,\tau) = (\sigma_1,\ldots,\sigma_k,\tau_1,\ldots,\tau_{n-k})$ is a permutation of $(1,\ldots,n)$ such that $\sigma_1 < \cdots < \sigma_k$ and $\tau_1 < \cdots < \tau_{n-k}$. From the above identity, one can infer that
\begin{equation}\label{eq:star.star}
  \star(\star \omega) = (-1)^{k(n-k)}\omega\qquad\forall\omega\in\Alt{k}(\Real^n)
\end{equation}
and, hence, that $\langle\star\omega,\star\mu\rangle = \langle\omega,\mu\rangle$, i.e., $\star$ is an isometry. Formula \eqref{eq:star.star} justifies the definition \eqref{eq:inv.star} of $\star^{-1}$. The anticommutativity \eqref{eq:wedge:anticommutativity} of $\wedge$, the definition of $\star$, and the symmetry of $\langle\cdot,\cdot\rangle$ then give
\begin{equation}\label{eq:commut.star.wedge}
  {\star^{-1}\omega}\wedge\mu=\mu\wedge\star\omega=\omega\wedge\star\mu\qquad\forall \omega,\mu\in\Alt{k}(\Real^n).
\end{equation}
\begin{example}[Hodge star operator in two and three dimensions]
  \label{rem:hodge.2.3}
  If $\omega \in \Alt{2}(\Real^3)$, i.e., $\omega = a_{12} \ \rd x^1 \wedge \rd x^2 + a_{13} \ \rd x^1 \wedge \rd x^3 + a_{23} \ \rd x^2 \wedge \rd x^3$ (see \eqref{eq:general.alternating.form}), one obtains $\star\omega \in \Alt{1}(\Real^3)$ with
  $$\star\omega = a_{12} \ \rd x^3 - a_{13} \ \rd x^2 + a_{23} \ \rd x^1.$$
  If $\omega \in \Alt{1}(\Real^2)$, i.e., $\omega = a_1\ \rd x^1 + a_2 \ \rd x^2$, then $\star\omega \in \Alt{1}(\Real^2)$ with
  $$\star \omega = a_1 \ \rd x^2 - a_2 \ \rd x^1.$$
\end{example}

\subsubsection{Vector proxies for alternating forms}

As already mentioned in Section~\ref{sec:alternating.forms}, $\Alt{0}(\Real^n) = \Real$ and $\Alt{n}(\Real^n)\cong \Real$.
Using the Riesz representation theorem to identify $(\Real^n)'$ and $\Real^n$, we can identify two more spaces of alternating forms: $\Alt{1}(\Real^n) = (\Real^n)'\cong \Real^n$ and, writing $\star\Alt{n-1}(\Real^n)=\Alt{1}(\Real^n)\cong \Real^n$ since $\star$ is bijective, $\Alt{n-1}(\Real^n)\cong\Real^n$.

Applied with $n=3$, and recalling the formula for Hodge star transformations of 2-forms in Remark~\ref{rem:hodge.2.3},
these identifications lead to considering a vector $\bvec v = (a,b,c) \in \Real^3$ as a \emph{proxy} for both the alternating linear and bilinear forms
$$
\Alt{1}(\Real^3) \ni \omega = a\ \rd x^1 + b\  \rd x^2 + c\ \rd x^3\ \text{and}\
\Alt{2}(\Real^3)\ni \mu = a\ \rd x^2 \wedge \rd x^3 - b\ \rd x^1 \wedge \rd x^3 + c\ \rd x^1 \wedge \rd x^2.
$$
On the other hand, when $n=2$, the discussion above gives two possible ways to identify $\Alt{1}(\Real^2) =\Alt{2-1}(\Real^2)$ with $\Real^2$. This leads to associating $a\ \rd x^1 + b\ \rd x^2=\omega\in\Alt{1}(\Real^2)$ either to the vector $\bvec v = (a,b) \in \Real^2$, or to its rotation by a right angle
$\rotation{-\pi/2}\bvec v = (b,-a)\in \Real^2$.

Based on the above identifications, when $n=3$, one can interpret the exterior product of two alternating multilinear forms $\omega\wedge\mu$ in terms of vector proxies $(\bvec{w}, \bvec{v})$ as follows:
\begin{itemize}
\item the vector product $\Real^3\times\Real^3 \ni (\bvec w, \bvec v) \mapsto \bvec w \times \bvec v \in \Real^3$
  when $(\omega,\mu)\in \Alt{1}(\Real^3) \times \Alt{1}(\Real^3)$;
\item the dot product $\Real^3\times\Real^3 \ni (\bvec w, \bvec v) \mapsto \bvec w \cdot \bvec v \in \Real$ when $(\omega,\mu)\in \Alt{1}(\Real^3)\times \Alt{2}(\Real^3).$
\end{itemize}
On the other hand, if $n=2$ and $\omega,\mu\in \Alt{1}(\Real^2)$, we can write $\omega\wedge\mu = (a\ \rd x^1 + b \ \rd x^2) \wedge (f\ \rd x^1 + g \ \rd x^2) = (ag-bf) \ \rd x^1 \wedge \rd x^2$. Considering the correspondences $\omega \leftrightarrow \bvec w = (a,b)$ and $\mu\leftrightarrow \bvec v = (f,g)$, we obtain
\begin{equation}
  \label{eq:exterior.product.2D}
  \omega\wedge\mu = (\bvec w\cdot \rotation{-\pi/2}\bvec v)\ \rd x^1 \wedge \rd x^2.
\end{equation}

\subsubsection{Contraction and trace}

For a given vector $\bvec{v}\in\Real^n$, the \emph{contraction} $\omega\lrcorner \bvec v \in \Alt{k-1}(\Real^n)$ of $\omega\in\Alt{k}(\Real^n)$ with $\bvec v$ is defined, for any $\bvec v_1, \ldots, \bvec v_{k-1} \in \Real^n$, by
\begin{equation}
  \label{eq:contraction}
  (\omega\lrcorner \bvec v)(\bvec v_1,\ldots,\bvec v_{k-1}) \coloneq \omega(\bvec v, \bvec v_1,\ldots,\bvec v_{k-1}).
\end{equation}
In terms of vector proxies, in the case where $n=3$, this contraction with $\bvec{v}$ corresponds to
\begin{itemize}
\item  the scalar product $\Real^3\ni \bvec w \mapsto \bvec v \cdot \bvec w \in \Real$ when $\bvec{w}\leftrightarrow \omega\in\Alt{1}(\Real^3)$;
\item  the vector product $\Real^3 \ni \bvec w \mapsto \bvec v \times \bvec w \in \Real^3$ when $\bvec{w}\leftrightarrow \omega\in\Alt{2}(\Real^3)$;
\item the multiplication of a real number $\Real \ni w \mapsto w\bvec v \in \Real^3$ when $w\leftrightarrow \omega\in\Alt{3}(\Real^3)$.
\end{itemize}

Let now $V\subset W$ be finite dimensional subspaces of $\Real^n$, and $\iota_V : V \hookrightarrow W$ be the inclusion of $V$ in $W$.
The \emph{trace} ${\tr_V}:\Alt{k}(W) \to \Alt{k}(V)$ is the pullback under $\iota_V$: For any $\bvec v_1,\ldots,\bvec v_k \in V$,
\begin{equation}
  \label{eq:trace.alt}
  \tr_V\omega(\bvec v_1, \ldots, \bvec v_k) \coloneq \omega(\iota_V\bvec v_1,\ldots,\iota_V\bvec v_k).
\end{equation}
The trace respects the exterior product, i.e., ${\tr_V}(\omega\wedge\mu) = {\tr_V}\omega \wedge {\tr_V}\mu$.

It is easy to see that, through the vector proxy of $\Alt{1}$ spaces, ${\tr_V}: \Alt{1}(W) \to \Alt{1}(V)$ is the orthogonal projection $\pi_V: W \to V$ of a vector $\bvec w\in W$ onto $V$.

Let us fix an integer $m\in[1,n]$ and suppose that $\dim(W) = m$ and $\dim(V) = m-1$, and that both spaces are oriented; let $\bvec{n}_V$ be the unit normal to $V$ such that, given a positively oriented basis $(\bvec{e}_1,\ldots,\bvec{e}_{m-1})$ of $V$, the family $(\bvec{n}_V,\bvec{e}_1,\ldots,\bvec{e}_{m-1})$ forms a positively oriented basis of $W$. Then, an identification of the trace ${\tr_V}:\Alt{m-1}(W)\to\Alt{m-1}(V)$ through vector proxies  is the scalar product with the vector $\bvec{n}_V$, that is, $W \ni \bvec w \mapsto \bvec w \cdot \bvec n_V \in \Real$.

\subsection{Exterior calculus in $\Real^n$}\label{sec:exterior.calculus.Rn}

\subsubsection{Differential forms}
Let $M$ be an $n$-dimensional flat manifold.
When the coefficients in~\eqref{eq:general.alternating.form} are functions
$a_\sigma:{M} \to \Real$, the map $\omega:{M} \to \Alt{k}(\Real^n)$ is referred to as a \emph{differential form}, or simply a $k$-form.
Consistently with the notation adopted in Section~\ref{sec:setting:differential.forms}, the
space of $k$-forms over ${M}$ without any specific smoothness requirement on the coefficients $a_\sigma$ is denoted by $\Lambda^k({M})$. If $\omega\in\Lambda^k({M})$, the value of $\omega$ at $\bvec{x}\in{M}$ is denoted by $\omega_{\bvec{x}}\in\Alt{k}(\Real^n)$.

If the coefficients $a_\sigma$ in~\eqref{eq:general.alternating.form} are polynomial functions, $\omega$ is said to be a \emph{polynomial differential form}.
Specifically, for an integer $r\ge 0$, the space of polynomial $k$-forms of degree $\le r$ is defined as
\[ 
\Poly{r}{}\Lambda^k({M}) \coloneq \left\{  \sum_{1\le\sigma_1 < \cdots < \sigma_k\le n} p_\sigma \ \rd x^{\sigma_1}\wedge\cdots\wedge \rd x^{\sigma_k} \ : \ p_\sigma \in
\Poly{r}{}({M})\right\},
\] 
where $\Poly{r}{}({M})$ is the space of scalar polynomials of degree $\le r$ over ${M}$.
All the arguments concerning vector proxies presented in Section~\ref{sec:exterior.algebra} for alternating $k$-linear forms
can be immediately extended to the case of $k$-forms. Hence, when $n\in\{2,3\}$, their corresponding vector proxies are scalar fields over ${M}$ when $k\in\{0,n\}$, and vector fields over ${M}$ when $k\in\{1,n-1\}$.

\subsubsection{Exterior derivative and de Rham complexes}

Provided that the coefficients $a_\sigma$ in~\eqref{eq:general.alternating.form}
are smooth enough, the \emph{exterior derivative} of a $k$-form $\omega\in\Lambda^k({M})$ is the linear unbounded
operator $\rd : \Lambda^k({M}) \to \Lambda^{k+1}({M})$ such that, in terms of standard coordinates on $\Real^n$,
\[ 
\rd \omega = \sum_{1\le\sigma_1 < \cdots < \sigma_k\le n} \sum_{i=1}^n \frac{\partial a_\sigma}{\partial x_i} \ \rd x^i \wedge \rd x^{\sigma_1}
\wedge \cdots \wedge \rd x^{\sigma_k}.
\] 

The interpretation of the exterior derivative in terms of vector calculus operators, through vector proxies of alternating forms and when $M$ is a domain $\Omega$ of $\Real^3$, is given in  \eqref{eq:d.proxy.n=3}. We have used in this diagram the spaces defined in the introduction of the paper.

\begin{equation}\label{eq:d.proxy.n=3}
  \begin{tikzcd}
    \text{Differential forms:}
    & {H\Lambda^0(\Omega)} \arrow{r}{\rd}\arrow[<->]{d}
    & {H\Lambda^1(\Omega)}\arrow{r}{\rd}\arrow[<->]{d}
    & {H\Lambda^2(\Omega)}\arrow{r}{\rd}\arrow[<->]{d}
    & {H\Lambda^3(\Omega)}\arrow[<->]{d} \\
    \text{Vector proxies:}
    & {H^1(\Omega)}\arrow{r}{\GRAD}
    & {\Hcurl{\Omega}} \arrow{r}{\CURL}
    & {\Hdiv{\Omega}}\arrow{r}{\DIV}
    & {L^2(\Omega)}.
  \end{tikzcd}
\end{equation}

In the case $n=2$, as we have two possible vector proxies for $\Alt{1}(\Real^2)$. These interpretations are illustrated in \eqref{eq:d.proxy.n=2.a} when $\omega=a\ \rd x^1+b\ \rd x^2\in\Alt{1}(\Real^2)$ is identified with $\bvec{v}=(a,b)$, and in \eqref{eq:d.proxy.n=2.b} when $\omega\in\Alt{2-1}(\Real^2)$ is identified with $\rotation{-\nicefrac{\pi}{2}}\bvec{v}$ (with $\ROT=\DIV\rotation{-\nicefrac{\pi}{2}}$ and $\VROT=\rotation{-\nicefrac{\pi}{2}}\GRAD$, respectively, denoting the scalar and vector curls, and $\Hrot{\Omega}$ the space of square-integrable vector-valued functions whose $\ROT$ is also square-integrable).

\begin{equation}\label{eq:d.proxy.n=2.a}
  \begin{tikzcd}
    \text{Differential forms:} & {H\Lambda^0(\Omega)} \arrow{r}{\rd}\arrow[<->]{d} & {H\Lambda^1(\Omega)}\arrow{r}{\rd}\arrow[<->]{d}& {H\Lambda^2(\Omega)}\arrow[<->]{d} \\
    \text{Vector proxies:}& {H^1(\Omega)}\arrow{r}{\GRAD} & {\Hrot{\Omega}}\arrow{r}{\ROT} & {L^2(\Omega)}.
  \end{tikzcd}
\end{equation}

\begin{equation}\label{eq:d.proxy.n=2.b}
  \begin{tikzcd}
    \text{Differential forms:} & H\Lambda^0(\Omega) \arrow{r}{\rd}\arrow[<->]{d} & H\Lambda^1({\Omega})\arrow{r}{\rd}\arrow[<->]{d}& H\Lambda^2({\Omega})\arrow[<->]{d} \\
    \text{Vector proxies:}& {H^1(\Omega)}\arrow{r}{\VROT} & {\Hdiv{\Omega}}\arrow{r}{\DIV} & {L^2(\Omega)}.
  \end{tikzcd}
\end{equation}

Notice, finally, that the exterior derivative satisfies the \emph{complex property} $\rd \circ \rd = 0$. This property translates, through vector proxies, into the well-known identities $\CURL\GRAD=\mathbf{0}$ and $\DIV\CURL=0$ for $n=3$, and $\ROT\GRAD=0$, $\DIV\VROT=0$ when $n=2$.

\subsubsection{Koszul differential}

Given $\bvec{x}_M\in\Real^n$, the \emph{Koszul differential} $\kappa_M : \Lambda^k({M}) \to \Lambda^{k-1}({M})$
is defined pointwise over ${M}$ as follows: For all $\bvec{x}\in{M}$, recalling the definition \eqref{eq:contraction} of the contraction $\lrcorner$,
\[
(\kappa_M \omega)_{\bvec{x}}
\coloneq
\omega_{\bvec{x}}\lrcorner(\bvec{x}-\bvec{x}_M).
\]
Its interpretation in terms of vector fields proxy is then analogous to that of a contraction of an alternating multilinear form with a vector,
except that the contraction is made pointwise with the vector field $\Real^n \ni \bvec x \mapsto \bvec x - \bvec x_M \in \Real^n$.
The terminology ``differential'' is legitimate, as $\kappa_M$ satisfies the complex
property $\kappa_M \circ \kappa_M = 0$ (since any alternating form applied to the same vector twice vanishes).

\subsubsection{Trace}

If $P\subset Q$ are (relatively) open sets in affine subspaces $V\subset W$ of $\Real^n$, the trace operator $\tr_P : C^0\Lambda^k(\overline{Q}) \rightarrow C^0\Lambda^k(\overline{P})$ on differential forms is defined pointwise, using the trace operator \eqref{eq:trace.alt} on alternating forms: For all $\omega \in C^0\Lambda^k(\overline{Q})$,
$$
(\tr_P \omega)_{\bvec{x}} \coloneq \tr_{V}\omega_{\bvec{x}}
\qquad\forall \bvec{x}\in P.
$$
Note that, in the case $P=Q$, the trace is simply the identity operator (and can be defined without any continuity assumption): $\tr_P \omega=\omega$ for all $\omega\in \Lambda^k(P)$.

Applying the same arguments as in Section~\ref{sec:exterior.algebra} pointwise over $P$, the trace operator in terms of vector fields proxy gives
\begin{itemize}
\item the restriction of functions, when $k=0$;
\item the orthogonal projection onto $V$ (that is, $\tr_P\omega\leftrightarrow \pi_{V} \bvec{w}$ if $\omega\leftrightarrow \bvec{w}$),
  when $k=1$;
\item the normal component on $P$ along the direction $\bvec{n}$ (that is, $\tr_P\omega\leftrightarrow \bvec{w}\cdot\bvec{n}$ if $\omega\leftrightarrow \bvec{w}$), with $\bvec{n}$ unit normal vector field preserving the orientations of $V$ and $W$,
  when $k=\dim(P) = \dim(Q)-1$.
\end{itemize}%
\begin{example}[Interpretation of the Stokes formula for $\ell=1$ and $n=3$]
  We rewrite here, for the reader's convenience, the integration by parts formula~\eqref{eq:ipp} for $\ell=1$
  and $n=3$:
  \begin{equation}\label{eq:stokes.appendix}
    \int_{M} \rd\omega\wedge\mu
    = \int_{M} \omega \wedge \rd\mu
    + \int_{\partial {M}} \tr_{\partial {M}} \omega \wedge \tr_{\partial  {M}}\mu
    \qquad\forall(\omega,\mu) \in \Lambda^1({M})\times\Lambda^{1}({M}).
  \end{equation}
  Given the previous interpretations of the exterior derivative and product in terms of vector proxies, if $\omega\leftrightarrow\bvec{w}$ and $\mu\leftrightarrow\bvec{v}$, then $\rd\omega\wedge\mu\leftrightarrow  \CURL\bvec{w}\cdot\bvec{v}$ and $\omega\wedge\rd\mu\leftrightarrow\bvec{w}\cdot\CURL\bvec{v}$. This leads to the following integration by parts formula for the curl:
  \begin{equation}\label{eq:stokes.curl.appendix}
    \int_{M} \CURL \bvec w \cdot \bvec v = \int_{M} \bvec w \cdot \CURL\bvec v + \int_{\partial{M}} (\bvec n \times(\bvec w \times \bvec n))
    \cdot (\bvec v \times \bvec n),
  \end{equation}
  where $\bvec n$ is the outer unit normal vector field over $\partial{M}$.
  For any fixed $\bvec{x}\in\partial{M}$, we have $\bvec n(\bvec x)
  \times(\bvec w(\bvec{x}) \times \bvec n(\bvec x)) = \pi_{T_{\bvec x}\partial{M}}\bvec w(\bvec x)$ (here, $T_{\bvec{x}}\partial{M}$ is the tangent space of $\partial{M}$ at $\bvec{x}$),
  whereas
  $\bvec v(\bvec x) \times \bvec n(\bvec x) = \rotation{-\pi/2}(\pi_{T_{\bvec x}\partial{M}}\bvec v(\bvec x))$,
  where the rotation is considered with respect to the orientation of the tangent plane given by $\bvec n(\bvec x)$.
  The boundary terms of \eqref{eq:stokes.appendix} and \eqref{eq:stokes.curl.appendix} therefore coincide, through the vector proxy for the exterior product of 1-forms in dimension 2 (see~\eqref{eq:exterior.product.2D}).
\end{example}


\begin{funding}
  Funded by the European Union (ERC Synergy, NEMESIS, project number 101115663).
  Views and opinions expressed are however those of the authors only and do not necessarily reflect those of the European Union or the European Research Council Executive Agency. Neither the European Union nor the granting authority can be held responsible for them.

  Francesco Bonaldi additionally acknowledges the partial support of \emph{Agence Nationale de la Recherche} and Université de Montpellier through the grant ANR-16-IDEX-0006 ``RHAMNUS''.
  Kaibo Hu also acknowledges the partial support of a \emph{Royal Society University Research Fellowship} through the grant URF$\backslash${\rm R}1$\backslash$221398.
\end{funding}


\bibliographystyle{plain}
\bibliography{ddr-pec}

\end{document}